\documentclass[final]{siamltex}
\usepackage{amsmath, amssymb}

\usepackage{float,epsfig}
\usepackage{algpseudocode}
\usepackage{algorithm}
\usepackage{mathrsfs}
\usepackage{color}
\allowdisplaybreaks[4]

\newtheorem{remark}[theorem]{ Remark}
\newtheorem{example}[theorem]{ Example}

\def\C{{\mathbb C}}

\def\lam{\lambda}
\def\sig{\sigma}
\def\Lam{\Lambda}

\def\rar{\rightarrow}

\newcommand{\ba}{\begin{array}}
\newcommand{\ea}{\end{array}}

\newcommand{\be}{\begin{equation}}
\newcommand{\ee}{\end{equation}}
\newcommand{\beano}{\begin{eqnarray*}}
\newcommand{\eeano}{\end{eqnarray*}}

\def \noin{\noindent}
\def \eig{\mathrm{eig}}
\def \nrk{\mathrm{nrank}}

\def \RCISS{\text{RCISS}}

\def \CISS{\text{CISS}}

\usepackage[us,12hr]{datetime}
\title{Structured strong linearizations of structured rational matrices }
\author{Ranjan Kumar Das \thanks{Department of Mathematics, IIT Guwahati,
Guwahati-781039, India ({\tt d.ranjan@iitg.ac.in}). } \and Rafikul Alam \thanks{Corresponding author, Department of Mathematics, IIT Guwahati, Guwahati - 781039, India, ({\tt rafik@iitg.ac.in, rafikul68@gmail.com}) Fax: +91-361-2690762/2582649.}   }

\begin{document}


\maketitle

\maketitle
\begin{abstract} Structured rational matrices such as symmetric,  skew-symmetric,  Hamiltonian, skew-Hamiltonian, Hermitian, and  para-Hermitian  rational matrices arise in many applications. Linearizations of rational matrices have been introduced recently for computing  poles,  eigenvalues, eigenvectors, minimal bases and minimal indices of rational matrices. For structured rational matrices, it is desirable to construct structure-preserving linearizations so as to preserve the symmetry in the eigenvalues and poles of the rational matrices. With a view to constructing structure-preserving linearizations of structured rational matrices, we propose a family of Fiedler-like pencils and show that the family of Fiedler-like pencils is a rich source of structure-preserving strong linearizations of structured rational matrices. We construct symmetric, skew-symmetric,  Hamiltonian, skew-Hamiltonian, Hermitian, skew-Hermitian,  para-Hermitian and para-skew-Hermitian strong linearizations of a rational matrix $G(\lam)$ when $G(\lam)$ has the same structure. Further, when $G(\lam)$ is real and symmetric, we show that the transfer functions of real symmetric linearizations of $G(\lam)$ preserve the Cauchy-Maslov index of $G(\lam).$  We  describe the recovery of eigenvectors, minimal bases and minimal indices of $G(\lam)$ from those of the linearizations of $G(\lam)$ and show that the recovery is operation-free.
\end{abstract}



\section{Introduction}
Structured rational matrices such as symmetric,   Hamiltonian,  skew-symmetric, skew-Hamiltonian, Hermitian, skew-Hermitian,  para-Hermitian and para-skew-Hermitian rational matrices arise in many applications, see~\cite{hilscot, kbrr, hrw, genin, fuhr,  volkervoss, bai11, willem} and the references therein. For example, the  Hermitian rational eigenvalue problem
$$  G(\lam) u :=\Big (\lambda^2 M + K - \sum^k_{i=1}\frac{1}{1 + \lambda  b_i }\Delta K_i \Big)u =0
$$ arises in the study of  damped vibration of a structure, where $M$ and $K$ are positive definite, $b_i$ is a relaxation parameter and $\Delta K_i $ is  an assemblage of element stiffness
matrices~\cite{ volkervoss, bai11}. Also various structured rational matrices arise as transfer functions of linear time-invariant (LTI) systems, see~\cite{hilscot, kbrr, hrw, fuhr, ran,  willem}.

Our main aim in this paper is to construct structure-preserving strong linearizations of structured rational matrices and to recover eigenvectors, minimal bases and  minimal indices of rational matrices from those of the linearizations. Let $G(\lam)$ be an $n\times n$ rational matrix, that is, the entries of $G(\lam)$ are scalar rational functions of the form $p(\lam)/{q(\lam)}$, where $p(\lam)$ and $q(\lam)$ are scalar polynomials.   We consider the following structures:
{\small
\be \label{struc}  \begin{array}{l@{}l|l@{}l}
	\mathrm{symmetric:} & G(\lam)^T = G(\lam) &  \mathrm{ Hermitian: } & G(\lam)^* = G(\bar \lam)   \\
	 \mathrm{ skew\mbox{-}symmetric:}  & G(\lam)^T = -G(\lam) & \mathrm{ skew\mbox{-}Hermitian: }  & G(\lam)^* = -G(\bar \lam) \\
	  \mathrm{ Hamiltonian: }  & G(\lam)^T = G(-\lam) &  \mathrm{ para\mbox{-}Hermitian:}  &  G(\lam)^* = G(- \bar\lam)   \\
	    \mathrm{skew\mbox{-}Hamiltonian:} \;  & G(\lam)^T = - G(-\lam)  &   \mathrm{para\mbox{-}skew\mbox{-}Hermitian:} \;  & G(\lam)^* = -G(- \bar \lam),\end{array}\ee}where $X^T$  (resp., $X^*$) denotes the transpose (resp., conjugate transpose) of a matrix $X$ and $\bar \lam$ denotes  the conjugate of $\lam.$ For more on these structured rational matrices, we refer to~\cite{hilscot, kbrr, hrw, genin, volkervoss, ran,  willem, fuhr, bai11} and the references therein.

%

We mention that there is a slight difference in the naming convention between  some of the structured rational matrices and structured matrix polynomials. The Hamiltonian (resp., skew-Hamiltonian) structure for rational matrices is known as $T$-even (resp., $T$-odd) structure for matrix polynomials \cite{mmmm_goodvib}. On the other hand, para-Hermitian (resp., para-skew-Hermitian) structure for rational matrices is known as $*$-even (rep., $*$-odd) structure for matrix polynomials \cite{mmmm_goodvib}. We follow both the naming conventions in the rest of the paper without any bias.

Linearization of rational matrices is a relatively new concept and has been studied in~\cite{rafinami1, rafinami3, admz, behera, rafiran1, bai11}. However, barring symmetric linearizations~\cite{rafiran1,fmq}, structure-preserving linearizations of structured rational matrices have not been constructed in the literature. The frameworks of Fielder pencils, generalized Fiedler pencils, and affine spaces of pencils for rational matrices presented  in \cite{rafinami1, rafinami3, rafiran1} are not adequate for construction of structure-preserving linearizations of structured rational matrices. 
 
The main aim of this paper is to present a framework  for construction of  structure-preserving strong linearizations of structured rational matrices considered in~(\ref{struc}).
For this purpose, we propose a new family of Fiedler-like pencils of  $G(\lam)$ which we refer to as generalized Fiedler pencils with repetition (GFPRs) of $G(\lam).$ We show that the GFPRs of $G(\lam)$ are Rosenbrock strong linearizations of $G(\lam)$ and describe the recovery of eigenvectors, minimal bases and minimal indices of $G(\lam)$ from those of the GFPRs of $G(\lam).$ In fact, we show that the  eigenvectors and minimal bases can be recovered without performing any arithmetic operations. Next, we show that the family of GFPRs of $G(\lam)$ is a rich source of structure-preserving linearizations of $G(\lam)$ and utilize these pencils to construct structure-preserving Rosenbrock strong linearizations of $G(\lam).$ In particular, when $G(\lam)$ is real symmetric, we construct  real symmetric linearizations of $G(\lam)$ whose transfer functions  preserve the Cauchy-Maslov index of $G(\lam).$
We also show that Fiedler pencils (FPs) and generalized Fiedler pencils (GFPs) of $G(\lam)$ constructed in~\cite{rafinami1, rafinami3} are in fact Rosenbrock strong linearizations of $G(\lam).$

The rest of the paper is organized as follows. We collect some basic results in Section~2. We introduce GFPRs of $G(\lam)$ in Section~3 and show that the FPs, GFPs and GFPRs of $G(\lam)$ are Rosenbrock strong linearizations. We construct structure-preserving Rosenbrock strong linearizations of structured rational matrices in Section~4. Finally, we describe the recovery of eigenvectors, minimal bases and minimal indices of $G(\lam)$ from those of the Rosenbrock strong linearizations of $G(\lam)$ in Section~5.

\textbf{Notation.} We denote  by $\C[\lam]$  the ring (over $\C$) of scalar polynomials and  by  $\C(\lam)$ the field of rational functions of the form  $p(\lambda)/ q(\lambda),$ where $p(\lambda)$ and $q(\lambda)$ are  polynomials in $\mathbb{C}[\lambda].$  We denote by $\mathbb{C}[\lambda]^{m\times n}$ (resp., $\mathbb{C}(\lambda)^{m\times n}$) the vector space of $m\times n$ matrix polynomials (resp., rational matrices) over $\C$ (resp., over $\C(\lam)$). The spaces $\C[\lam]^m$ and $\C(\lam)^{m},$ respectively, denote $\mathbb{C}[\lambda]^{m\times n}$ and  $\mathbb{C}(\lambda)^{m\times n}$ when $n=1.$ We denote the \textit{j}-th column of the $n\times n $ identity matrix $I_n$ by $e_j$ and the transpose (resp., conjugate transpose) of an $m\times n$  matrix $A$ by $A^T$ (resp., $A^*$). The right and left null spaces of $A$ are given by $ \mathcal{N}_r(A) := \{ x \in \C^n : Ax = 0\}$ and $\mathcal{N}_l(A) := \{ x \in \C^m : x^T A = 0 \},$ respectively. We denote by $A\otimes B$ the Kronecker product of the matrices $A$ and $B.$

\section{Basic results} Let $ G(\lam) \in \C(\lam)^{m\times n}$. The rank of  $G(\lam)$ over the field $\C(\lam)$ is called the \textit{normal rank} of $G(\lam)$ and is denoted by $\nrk(G).$  If $\nrk(G) = n =m$ then $G(\lam)$  is said to be \textit{regular}, otherwise $G(\lam)$ is said to be \textit{singular}. A complex number $\mu \in \C$ is said to be an \textit{eigenvalue} of $G(\lambda)$ if $\rank(G(\mu)) < \nrk(G).$  We denote the set of eigenvalues of $G$ by $\eig(G).$ Let
$$
D(\lam) := \diag\left( \frac{\phi_{1}(\lam)}{\psi_{1}(\lam)}, \ldots, \frac{\phi_{k}(\lam)}{\psi_{k}(\lam)}, 0_{m-k, n-k}\right)$$
be the Smith-McMillan form~\cite{kailath, rosenbrock70} of  $G(\lam),$ where $ k:= \nrk(G)$ and the scalar polynomials $\phi_{i}(\lam)$ and $ \psi_{i}(\lam)$ are monic and pairwise coprime and  that $\phi_{i}(\lam)$ divides $\phi_{i+1}(\lam)$ and $\psi_{i+1}(\lam)$ divides $\psi_{i}(\lam),$ for $i= 1, 2, \ldots, k-1$. Set $\phi_{G}(\lam) := \prod _{j=1}^{k} \phi_{j}(\lam) \,\,\, \mbox{ and } \,\,\, \psi_{G}(\lam) := \prod _{j=1}^{k} \psi_{j}(\lam).$ Then $ \mu\in \C$ is a pole of $G(\lam) $ if $\psi_G(\mu) =0.$ A complex number $ \mu$ is said to be a {\em zero} of $G(\lam)$ if $ \phi_G(\mu) =0.$ The {\em spectrum} of $G(\lam)$ is given by $ \mathrm{Sp}(G) := \{ \lam \in \C : \phi_G(\lam) = 0\}$ and consists of the finite zeros of $G(\lam).$
 Note that $ \eig(G) \subset \mathrm{Sp}(G).$ See~\cite{rafinami1, kailath} for more on eigenvalues and zeros of $G(\lam).$

When $G(\lam)$ is singular, the {\em{right null space}} $\mathcal{N}_r(G) $ and the {\em{left null space}} $\mathcal{N}_l(G)$ of $G(\lam)$ are given by
\beano \mathcal{N}_r(G) &:=& \lbrace x(\lambda)\in\mathbb{C}(\lambda)^n :G(\lambda)x(\lambda) = 0 \rbrace   \subset \C(\lam)^n, \\
\mathcal{N}_l(G) &:=&\lbrace y(\lambda)\in\mathbb{C}(\lambda)^m:y(\lambda)^T G(\lambda) = 0\rbrace \subset \C(\lam)^m.\eeano
Let $ \mathcal{B} := \big( x_1(\lam), \ldots, x_p(\lam)\big)$ be a polynomial basis~\cite{kailath, forney}  of $\mathcal{N}_r(G)$ ordered so that $\deg(x_1) \leq \cdots \leq \deg(x_p),$ where $x_1(\lam), \ldots, x_p(\lam)$ are vector polynomials, that is, are elements of $\C[\lam]^n.$ Then $\mathrm{Ord}(\mathcal{B}) := \deg(x_1) + \cdots + \deg(x_p)$ is called the \textit{order} of the basis $\mathcal{B}.$  A basis $\mathcal{B}$ is said to be a minimal polynomial basis~\cite{kailath} of $\mathcal{N}_r(G)$ if $\mathcal{E}$ is any polynomial basis of $\mathcal{N}_r(G)$ then $ \mathrm{Ord}(\mathcal{E}) \geq \mathrm{Ord}(\mathcal{B}).$ A minimal polynomial basis $ \mathcal{B} := \big( x_1(\lam), \ldots, x_p(\lam)\big)$ of $\mathcal{N}_r(G)$ with $\deg(x_1) \leq \cdots \leq \deg(x_p)$ is called a {\em right minimal basis} of $G(\lam)$ and $ \deg(x_1) \leq \cdots \leq \deg(x_p)$ are called the {\em{right minimal indices}} of $G(\lam).$ A {\em{left minimal basis}} and the {\em{left minimal indices}} of $G(\lam)$ are defined similarly. See~\cite{kailath, forney} for further details.

 We say that a $k\times p$ matrix polynomial $ Z(\lam)$ is a minimal basis if the columns of $Z(\lam)$ form a minimal basis of the subspace of $\C(\lam)^k$ spanned (over the field $\C(\lam)$) by the columns of $Z(\lam).$

Let $ G(\lam) \in \C(\lam)^{n\times n}$. We consider a realization of $G(\lam)$ of the form
\begin{equation}\label{minrel2_RCh4}
 G(\lambda)= \sum\nolimits^m_{j=0} A_j \lam^j+C(\lambda E-A)^{-1}B =:   P(\lambda)+C(\lambda E-A)^{-1}B,
\end{equation}
where $ \lam E- A$ is an $r\times r$ matrix pencil with $E$ being nonsingular,  $C\in\mathbb{C}^{n\times r}$ and $B\in\mathbb{C}^{r\times n}.$
The realization (\ref{minrel2_RCh4}) is said to be {\em minimal} if the size of the pencil $\lam E- A$ is the smallest among all the realizations of $G(\lambda),$ see~\cite{kailath}. The matrix polynomial
\begin{equation} \label{slamsystemmatrix_RCh4}
{\mathcal{S}}(\lambda ) : = \left[
                            \begin{array}{c|c}
                              P(\lambda) & C \\
                              \hline
                              B & A-\lambda E
                            \end{array}
                          \right]
\end{equation} is called the  {\em{system matrix}} (or the Rosenbrock system matrix) of $G(\lam)$ associated with the realization (\ref{minrel2_RCh4}). The system matrix $\mathcal{S}(\lam)$ is said to be {\em irreducible } if the realization (\ref{minrel2_RCh4}) is minimal. The system matrix  $\mathcal{S}(\lam)$ is irreducible  if and only if $\rank\Big(\left[\begin{array}{cc} B & A - \lam E\end{array}\right]\Big) = r = \rank\Big(\left[\begin{array}{cc} C^T & (A - \lam E)^T\end{array}\right]^T\Big), $ see~\cite{kailath, rosenbrock70}. Observe that  $\eig(G) \subset \eig(\mathcal{S})$ and we have  $ \eig(\mathcal{S}) = \mathrm{Sp}(G)$ when $\mathcal{S}(\lam)$ is irreducible, see~\cite{rafinami1, rosenbrock70}.

An $n\times n$ matrix polynomial  $U(\lam)$ is said to be {\em unimodular} if $ \det(U(\lam))$ is a nonzero constant independent of $\lam.$ A rational matrix $G(\lam)$ is said to be {\em proper } if $ G(\lam) \rightarrow D$ as $ \lam \rar \infty,$ where $D$ is a matrix. An $n\times n$ rational matrix $F(\lam)$ is said to be {\em  biproper} if $F(\lam)$ is proper and $F(\infty)$ is a nonsingular matrix~\cite{vardulakis}.

\begin{definition}[\cite{rafiran1}] \label{stln} Let  $\mathbb{L}(\lam)$ be an $(mn+r)\times (mn+r)$ irreducible system matrix of the form
\be\label{lbb} \mathbb{L}(\lam) := \left[
                                           \begin{array}{c|c}
                                             \mathcal{X}- \lam \mathcal{Y} & \mathcal{C} \\
                                             \hline
                                             \mathcal{B} &  H - \lam K \\
                                           \end{array}
                                         \right], \ee where $ H- \lam K$ is an $r\times r$  pencil with $K$ being nonsingular. Then  $\mathbb{L}(\lam)$ is said to be a Rosenbrock strong linearization of $G(\lam)$ if the following                                       conditions hold.

\begin{itemize}
\item[(a)] There exist $mn\times mn$ unimodular matrix polynomials $U(\lam)$ and $V(\lam)$, and $r\times r$ nonsingular matrices $U_0$ and $V_0$  such that
	\be
	\left[
	\begin{array}{c|c}
		U(\lam) & 0 \\
		\hline
		0 & U_0 \\
	\end{array}
	\right] \mathbb{L} (\lam) \left[
	\begin{array}{c|c}
		V(\lam) & 0 \\
		\hline
		0 & V_0 \\
	\end{array}
	\right] = \left[
	\begin{array}{c|c}
		I_{(m-1)n} & 0 \\
		\hline
		0 &  \mathcal{S}(\lam) \\
	\end{array}
	\right]. \nonumber \ee
	
\item[(b)]	There exist $mn\times mn$  biproper rational  matrices $\mathcal{O}_{\ell} (\lam)  $ and $\mathcal{O}_r (\lam) $ such that
	\begin{equation}
	\mathcal{O}_{\ell} (\lam)  \, \lam^{-1} \mathbb{G} (\lam)  \, \mathcal{O}_r (\lam) =
	\left[ \begin{array}{c|c}
	I_{(m-1)n} & 0 \\
	\hline
	0 & \lam^{-m} G(\lam) \\
	\end{array}
	\right], \nonumber
	\end{equation}
	where $\mathbb{G}(\lam) :=\mathcal{X}- \lam \mathcal{Y} + \mathcal{ C} (\lam K -H)^{-1} \mathcal{ B}  $ is the transfer function of $ \mathbb{L}(\lam)$.
\end{itemize}
The pencil $\mathbb{L}(\lam)$ is also referred to as a Rosenbrock strong linearization of $\mathcal{S}(\lam).$

\end{definition}

We refer to \cite{rafiran1}  for more on Rosenbrock strong linearizations of $G(\lam)$ and the relation between the  structural indices of (finite and infinite) zeros and poles of $G(\lam)$ and $\mathbb{L}(\lam).$
Suffice it to say that the condition (a) ensures~(see, \cite[Theorem~3.4]{rafinami3}) that $ U(\lam) \mathbb{G}(\lam)V(\lam)  = \diag( I_{(m-1)n}, \; G(\lam))$ which in turn ensures that $G(\lam)$ and $\mathbb{G}(\lam)$ have the same finite zeros and poles. The irreducibility of $\mathbb{L}(\lam)$ guarantees that the finite zeros and poles of $\mathbb{G}(\lam)$ are the same as the finite eigenvalues of $\mathbb{L}(\lam)$ and  $H - \lam K,$ respectively; see~\cite{kailath, rafiran1}. On the other hand, the condition (b) ensures that the structural indices of zeros and poles of $G(\lam)$ at infinity can be recovered from the structural indices of eigenvalues and poles of $\mathbb{L}(\lam)$ at infinity (see~\cite{rafiran1}). Thus the zeros and poles of $G(\lam)$ including their structural indices can be obtained by solving the eigenvalue problems $\mathbb{L}(\lam) v = 0$ and $(H-\lam K)u=0$; see~\cite{rafinami1, rafinami2, rafinami3, rafiran1}. As mentioned in~\cite{rafiran1}, Definition~\ref{stln} is equivalent to the definition of strong linearization of rational matrices presented in~\cite{admz}.


\subsection{Fiedler matrices}\label{sectionflp_RCh4}
 For $k, \ell \in \mathbb{Z},$ we use the following notation
	$$k: \ell := \left\{ \begin{array}{cc}
	k,k+1,\ldots, \ell & \mbox{if}~ k \leq \ell, \\
	\emptyset & \mbox{if}~ k > \ell.
	\end{array} \right.$$
When $ k \leq \ell $, $(k:\ell)$ is called a {\em string} of integers from $k$ to $\ell$.

 {\bf Assumption:} {\em  For the rest of the paper, we assume that $ P(\lam) : = \sum_{i=0}^{m} \lam^i A_i$ with $A_m \neq 0$ and the realization  $G(\lam)  = P(\lam) + C (\lam E - A)^{-1} B$ of $G(\lam)$  given by (\ref{minrel2_RCh4}) is minimal. The system matrix $\mathcal{S} (\lam)$ associated with $G(\lam)$ is given by (\ref{slamsystemmatrix_RCh4}).}

 For an arbitrary matrix $X \in \mathbb{C}^{n \times n}$, we define the elementary matrices by~\cite{BDFR15}
$$ M_{0} (X) := \left[
                   \begin{array}{@{}cc@{}}
                     I_{(m-1)n} &  \\
                      & X \\
                   \end{array}
                 \right],~   M_{i} (X) := \left[
                 \begin{array}{@{}cccc@{}}
                 I_{(m-i-1)n} &  & & \\
                 & X & I_{n} & \\
                 & I_{n} & 0  & \\
                 &   &   & I_{(i-1)n}\\
                 \end{array}
                 \right] ~\mbox{for}~ i= 1: m-1,$$
$$ M_{-m}(X) := \left[
\begin{array}{@{}cc@{}}
X &  \\
& I_{(m-1)n} \\
\end{array}
\right],  ~ M_{-i}(X) := \left[
  \begin{array}{@{}c@{\;}ccc@{}}
    I_{(m-i-1)n} &  & & \\
     & 0 & I_{n} & \\
     & I_{n} & X  &  \\
     &   &   & I_{(i-1)n}\\
  \end{array}
\right] ~\mbox{for}~ i= 1: m-1.$$
Note that, for $i=1:m-1$, $M_i(X)$ and $M_{-i}(X)$ are invertible and $(M_i(X))^{-1} = M_{-i}(-X)$ for any arbitrary matrix $X \in \mathbb{C}^{n \times n}$. On the other hand, the matrices $M_{0} (X)$ and $M_{-m} (X)$ are invertible if and only if $X$ is invertible. Further, $M_{i} (X) M_{j} (Y) = M_{j} (Y) M_{i} (X)$ holds for any matrices $X, Y \in \mathbb{C}^{n \times n}$ if $||i|-|j|| > 1$,  see \cite{BDFR15}. For $i \in \{-m:m-1\}$, we define \cite{BDFR15}
$$ M_i^P  : =
     \left\{ \begin{array}{ll}
              M_i(-A_i) & \mbox{if}~ i \geq 0,\\
              M_i(A_{-i}) & \mbox{if}~ i <  0.
              \end{array}
        \right.$$
Then $ M_i^P $, $i \in \{-m:m-1\}$, are the Fiedler matrices of $P(\lam) $ (see \cite{tdm}).

For an arbitrary matrix $X \in \mathbb{C}^{n \times n} $, we define $(mn+r)\times (mn+r)$ elementary matrices  $ \mathbb{M}_{i}(X)$ by
$$\mathbb{M}_{i}(X) := \left[
                       \begin{array}{c|c}
                         M_{i}(X) &  \\
                         \hline
                          & I_{r} \\
                       \end{array}
                     \right] \; \mbox{ for } \; i\in \{ -m:m-1\}.$$
Note that $\mathbb{M}_i(X)$ and $\mathbb{M}_{-i}(X)$ are invertible and $(\mathbb{M}_i(X))^{-1} = \mathbb{M}_{-i}(-X) $  for $i=1:m-1$. On the other hand, the matrices $\mathbb{M}_{0} (X)$ and $\mathbb{M}_{-m} (X)$ are invertible if and only if $X$ is invertible.  For any arbitrary matrices $X,Y \in \mathbb{C}^{n \times n}$, we have  $\mathbb{M}_{i} (X) \mathbb{M}_{j} (Y) = \mathbb{M}_{j} (Y) \mathbb{M}_{i} (X)$ if $||i|-|j|| > 1.$

The $(mn+r)\times (mn+r)$ Fiedler matrices $ \mathbb{M}^\mathcal{S}_{i},$ $ i \in \{ -m:m-1\}$, associated with the system matrix (\ref{slamsystemmatrix_RCh4}) are defined by~\cite{rafinami1, rafinami3}
\begin{equation*}
\mathbb{M}^\mathcal{S}_{0} := \left[
                       \begin{array}{c|c}
                         M^P_{0} & -e_{m} \otimes C \\
                         \hline \\[-1em]
                         -e_{m}^{T}\otimes B & -A \\
                       \end{array}
                     \right], ~ \mathbb{M}^\mathcal{S}_{-m} := \left[
                                        \begin{array}{c|c}
                                          M^P_{-m} & 0 \\
                                          \hline
                                          0 & -E \\
                                        \end{array}
                                      \right], ~ \mathbb{M}^\mathcal{S}_{i} := \left[
                       \begin{array}{c|c}
                         M^P_{i} & 0 \\
                         \hline
                         0 & I_{r} \\
                       \end{array}
                     \right],\label{0mmfr_RCh4}
\end{equation*}
for $i =1:m-1$, and $ \mathbb{M}^\mathcal{S}_{-i}: = (\mathbb{M}^\mathcal{S}_{i})^{-1}$ for $i=1:m-1$.
The matrices $\mathbb{M}^\mathcal{S}_{i} $ are also referred to as Fiedler matrices of $G(\lam)$. We have $\mathbb{M}^\mathcal{S}_{i} \mathbb{M}^\mathcal{S}_{j} = \mathbb{M}^\mathcal{S}_{j} \mathbb{M}^\mathcal{S}_{i}$ for $ ||i|-|j|| > 1$, except for $ ||i|-|j|| =m.$ For convenience in defining Fiedler-like pencils, we define
\begin{equation}
\mathbb{M}^P_{i} := \left[
                       \begin{array}{c|c}
                         M^P_{i} & \\
                         \hline
                           & I_r \\
                       \end{array}
                     \right]~\text{for}~i \in \{ -m:m-1\}.
\end{equation}

\begin{remark}
	Note that $\mathbb{M}^\mathcal{S}_i = \mathbb{M}^P_{i}$, for $ i= \pm 1,\ldots ,\pm (m-1)$, and  $\mathbb{M}^\mathcal{S}_0 \neq  \mathbb{M}^P_{0} $ and $\mathbb{M}^\mathcal{S}_{-m} \neq  \mathbb{M}^P_{-m}$. The utility of the notation $\mathbb{M}^P_{i}$ will be clear when we analyze Fiedler-like pencils.
\end{remark}

\subsection{Index tuple} Permutations and sub-permutations are defined as follows.

\begin{definition}\cite{rafinami3} Let $N$ be a finite set. A bijection $\omega : N \rightarrow N$ is called a permutation of $N.$ $\tau$ is said to be a sub-permutation of $N$ if $\tau$ is a permutation of a subset of $N.$
\end{definition}

\begin{definition} \cite{rafinami3}  An ordered tuple $\textbf{t} : =(t_1 , t_2 ,\ldots , t_p)$ is said to be an index tuple containing indices from $\mathbb{Z}$ if $t_i \in \mathbb{Z}$ for $i=1:p.$ We define $- \textbf{t} : = ( -t_1, -t_2, \ldots , -t_p)$, $rev(\textbf{t}):=( t_p ,  \ldots, t_2, t_1)$ and $\textbf{t} + k:= ( t_1 +k , t_2+k, \ldots , t_p +k)$ for $k \in \mathbb{Z}. $ For any index tuples $\textbf{t} := ( t_1  , \ldots , t_p )$ and $\textbf{s} := ( s_1  , \ldots , s_q)$, we define  $\textbf{t} \cup \textbf{s} := (\textbf{t} ,\textbf{s}) = ( t_1  , \ldots , t_p , s_1  , \ldots , s_q)$.
\end{definition}

Next, we define SIP, rsf and csf of an index tuple which will  be used extensively.

\begin{definition} \cite{bueno, Vologiannidis} Let  $\sigma := (i_1, i_2, \ldots , i_t)$ be an index tuple containing indices from $\{0,1,\ldots, h\}$ for some non-negative integer $h$. Then:
\begin{enumerate}
\item[(a)] $\sigma$ is said to satisfy the Successor Infix Property (SIP) if for every pair of indices $i_a, i_b \in \sigma$ with $1 \leq a < b \leq t$ satisfying $i_a = i_b,$ there exists at least one index $i_c = i_a +1$ such that $a < c <b.$ Let $\alpha$  be an index tuple containing indices from $\{-h,-h+1,\ldots ,-1 \}.$ Then $\alpha$ is said to satisfy the SIP if $\alpha + h $ satisfies the SIP.

\item[(b)]$\sigma$ is said to be in  column standard form if
$$\sigma= ( a_s: b_s, a_{s-1}: b_{s-1}, \ldots , a_2:b_2, a_1 : b_1),$$
with $0 \leq b_1 < \cdots < b_{s-1} < b_s \leq h $ and $0 \leq a_j \leq b_j ,$ for all $j =1 , \ldots ,s. $ We denote the column standard form of $\sig$  by $csf(\sigma).$ Let $\beta$  be an index tuple containing indices from $\{-h,-h+1,\ldots ,-1 \}.$ Then $\beta$ is said to be in column standard form if $\beta + h$ is in column standard form.
\end{enumerate}
\end{definition}

\begin{definition} \cite{BDFR15}  Let $\alpha$ and $\beta$ be two index tuples. Then $\alpha$ is said to be a subtuple of $\beta$ if $\alpha = \beta$ or if $\alpha$ can be obtained from $\beta$ by deleting some indices in $\beta.$
\end{definition}

\begin{example} Let $\alpha = ( 1,2,0,3,0,2)$ be an index tuple. Then  $(2,3,2)$ is a subtuple of $\alpha$ but $ ( 2,2,3)$ is not a subtuple of $\alpha.$
\end{example}

We now present the concept of consecutive consecutions and consecutive inversions of an index tuple which we will  use extensively in the paper.

\begin{definition}[\cite{rafiran3}, Consecutions and inversions] \label{coninvoftuple_RCh4} Let $\alpha$ be an index tuple containing indices from $\{ 0: m\}$. Suppose that $t \in \alpha$. Then we say that $\alpha$ has $p$ consecutive consecutions at $t$ if $ ( t , t+1 , \ldots, t+p)$ is a subtuple of $\alpha$ and $ ( t , t+1 , \ldots, t+p, t+p+1)$ is not a subtuple of $\alpha$. We denote the number of consecutive consecutions of $\alpha$ at $t$ by $ c_{t}(\alpha)$. Similarly, we say that $\alpha$ has $s$ consecutive inversions at $t$ if $ (t+s, \ldots, t+1, t)$ is a subtuple of $\alpha$ and $ (t+s+1, t+s, \ldots, t+1, t)$ is not a subtuple of $\alpha$. We denote the number of consecutive inversions of $\alpha$ at $t$ by $ i_{t}(\alpha)$. For any index $k \in \{ 0: m\} $, if $ k \notin \alpha$, we define $ c_k(\alpha) : =-1$ and $i_k(\alpha) := -1.$
\end{definition}

\begin{example} \label{exampleofnewconse_RCh4} Let $\alpha := ( 1,0,2,1,3,2,4,1,3,2,1)$ be an index tuple containing indices from $\{0:6\}$.  Then $ c_0(\alpha) = 3$ as $(0,1,2,3)$ is a subtuple of $\alpha$ and $(0,1,2,3,4)$ is not a subtuple of $\alpha$.
\end{example}

\begin{remark} \cite{rafiran3} Let $\alpha$ be a permutation of $\{0 :m-1\}$. We denote the total number of  consecutions and inversions of $\alpha$  by $c(\alpha)$ and $i(\alpha),$ respectively. Note that $ c(\alpha) + i(\alpha) = m-1.$
\end{remark}

\section{Generalized Fiedler pencils with repetition} We now introduce a new family of Fiedler-like pencils for rational matrices which we refer to as generalized Fiedler pencils with repetition (GFPRs). We proceed as follows.

\begin{definition}[\cite{BDFR15}, Matrix assignments] Let $\textbf{t} := (t_1 ,t_2, \ldots , t_k)$ be an index tuple containing indices from $\{ -m : m-1 \}$ and $X:=( X_1 ,X_2, \ldots , X_k)$ be a tuple of $n \times n$ matrices. We define
	$ M_{\textbf{t}} (X) : = M_{t_1} (X_1) M_{t_2} (X_2) \cdots M_{t_k} (X_k)$ and say that $X$ is a matrix assignment for $\textbf{t}$. Further, we say that the matrix $X_j$ is assigned to the position $j$ in $\textbf{t}$. The matrix assignment $X$ for $\textbf{t}$ is said to be nonsingular if the matrices assigned by $X$ to the positions in $\textbf{t}$ occupied by the $ 0$ and $- m$ indices are nonsingular. Further, we define $rev(X) : = ( X_k, \ldots, X_2,X_1)$.
\end{definition}

Let $\textbf{t} := (t_1 , \ldots, t_k)$ be an index tuple containing indices from $ \{ -m : m-1\}$ and $X := (X_1, \ldots , X_k)$ be  a matrix assignment for $\textbf{t}$.
Then  we say that $X $ is the trivial matrix assignment for the index tuple $\textbf{t}$ associated with the matrix polynomial $P(\lambda)$  if $ M_{t_j} (X_j) = M_{t_j}^P$ for $j=1:k.$ Further, we define $ M^P_{\textbf{t}} : = M^P_{t_1}  \cdots M^P_{t_k}$. Similarly, we define
$\mathbb{M}_{\textbf{t}} (X):=\mathbb{M}_{t_1} (X_1)  \cdots \mathbb{M}_{t_k}(X_k)$,
$\mathbb{M}^{\mathcal{S}}_{\textbf{t}} :=\mathbb{M}^{\mathcal{S}}_{t_1} \cdots \mathbb{M}^{\mathcal{S}}_{t_k}$, and  $\mathbb{M}^{P}_{\textbf{t}} :=\mathbb{M}^{P}_{t_1} \cdots \mathbb{M}^{P}_{t_k}$.

\begin{definition}[GFPR of $G(\lam)$] \label{p02gfprglam2m181a_RCh4} Let $0 \leq h \leq m-1,$ and let $\sigma$ and $\tau$ be permutations of $\{ 0 :h \}$ and $\{ -m :-h -1\},$ respectively. Let $\sigma_1$ and $\sigma_2$ be index tuples containing indices from $\{0 :h-1 \}$ such that $(\sigma_1,\sigma , \sigma_2)$ satisfies the SIP. Similarly, let $\tau_1$ and $\tau_2$ be index tuples containing indices from $\{-m :-h -2 \}$ such that $( \tau_1 ,\tau , \tau_2)$ satisfies the SIP. Let $X_1 , X_2, Y_1$ and $Y_2$ be any arbitrary matrix assignments for $\sigma_1, \sigma_2, \tau_1$ and $ \tau_2$, respectively. Then the pencil
	\begin{equation} \label{gfprunknowncoeff13m181215_RCh4}
	\mathbb{L}(\lambda):=  \mathbb{M}_{\tau_1} (Y_1) \, \mathbb{M}_{\sigma_1} (X_1) \, (\lambda \mathbb{M}^\mathcal{S}_{\tau}  -  \mathbb{M}^\mathcal{S}_{\sigma})\, \mathbb{M}_{\sigma_2} (X_2)\, \mathbb{M}_{\tau_2}(Y_2)
	\end{equation}
	is said to be a generalized Fiedler pencil with repetition (GFPR) of $G(\lambda)$.  We also refer to $\mathbb{L}(\lam)$ as a GFPR of $\mathcal{S}(\lambda)$.
\end{definition}

Note that if all the matrix assignments $X_1 , X_2, Y_1$ and $Y_2$ in Definition~\ref{p02gfprglam2m181a_RCh4} are the trivial matrix assignments  then $	 \mathbb{L}(\lambda)=  \mathbb{M}^P_{\tau_1} \mathbb{M}^P_{\sigma_1} (\lambda \mathbb{M}^\mathcal{S}_{\tau}  -  \mathbb{M}^\mathcal{S}_{\sigma}) \mathbb{M}^P_{\sigma_2} \mathbb{M}^P_{\tau_2} $ is called a Fiedler pencil with repetition (FPR) of $G(\lam)$ \cite{behera,rafiran2}.  Hence the family of FPRs  of $G(\lam)$ is a subclass of the family of GFPRs of $G(\lam)$.

\begin{example} \label{examplegfprma_RCh4} Let  $G(\lambda) := \sum_{i=0}^4 \lambda^i A_i + C(\lambda E- A)^{-1}B.$ Consider  $\sigma:=(1,2,3,0),$ $\tau:=(-4),$  $\sigma_2:=(2,1)$ and $\sigma_1 = \tau_1=\tau_2= \emptyset$.  Then
	$$  \big( \lambda \mathbb{M}^\mathcal{S}_{-4} - \mathbb{M}^\mathcal{S}_{(1,2,3,0)} \big ) \mathbb{M}_{(2,1)} (X,Y) =  \left[
	\begin{array}{@{}cccc|c@{}}
	\lambda A_4 + A_3 & -X & -Y & -I_n &0\\
	A_2& \lambda X -I_n  & \lambda Y &\lambda I_n &0\\
	A_1& \lambda I_n &  A_0 &0 &C\\
	-I_n & 0& \lambda I_n &0 &0\\ \hline
	0& 0 &B &0 & A- \lambda E
	\end{array}\right]$$
	is a GFPR of $G(\lam)$, where $(X,Y)$ is an arbitrary  matrix assignment for $\sig_2$.
\end{example}

\begin{remark} 	The pencil $L(\lambda):=  M_{\tau_1} (Y_1) M_{\sigma_1} (X_1) (\lambda M^P_{\tau}  -  M^P_{\sigma}) M_{\sigma_2} (X_2) M_{\tau_2} (Y_2)$
	is called a generalized Fiedler pencil with repetition (GFPR) of $P(\lambda)$ \cite{BDFR15}, where
	$\sig, \tau, \sig_j$ and $\tau_j$, $j=1,2$, are as given in Definition~\ref{p02gfprglam2m181a_RCh4}. 	In particular, if $X_1 , X_2, Y_1$ and $Y_2$ are the trivial matrix assignments then $	L(\lambda) := M^P_{\tau_1} M^P_{\sigma_1} (\lambda M^P_{\tau}  -  M^P_{\sigma}) M^P_{\sigma_2} M^P_{\tau_2} $ is called a Fiedler pencil with repetition (FPR) of $P(\lam)$ \cite{Vologiannidis, bueno}.
\end{remark}

We now show that a GFPR of $G(\lam)$ can be constructed directly from a GFPR of $P(\lam)$ without performing any arithmetic operations. For this purpose we need the following result which is given in \cite[Lemma~3.10]{rafiran3}.

\begin{lemma} \cite{rafiran3} \label{lemmblockrowncolp0221d17ee_RCh4} Let $ L(\lambda):= M_{(\tau_1,\sigma_1)} (Y_1,X_1) (\lambda M^P_{\tau} -  M^P_{\sigma}) $ $ M_{(\sigma_2,\tau_2)} (X_2,Y_2) $ be a GFPR of $P(\lambda)$. Then we have $(e_{m-c_0(\sigma)}^{T}\otimes I_n) M_{(\sigma_{2},\tau_{2})} (X_2,Y_2) =  e_{m-c_0(\sigma, \sigma_2)}^{T}\otimes I_n$ and $
	M_{(\tau_{1},\sigma_{1})} (Y_1,X_1) \,(e_{m-i_0(\sigma)}\otimes I_n) = e_{m-i_0(\sigma_1, \sigma)}\otimes I_n.$
\end{lemma}

\begin{theorem}\label{gfprptoge_RCh4} Let $ \mathbb{L}(\lambda):=  \mathbb{M}_{(\tau_1,\sigma_1)} (Y_1, X_1) \, (\lambda \mathbb{M}^\mathcal{S}_{\tau} -  \mathbb{M}^\mathcal{S}_{\sigma}) \, \mathbb{M}_{(\sigma_2,\tau_2)} (X_2,Y_2) $ and $ L(\lambda):= M_{(\tau_1,\sigma_1)} (Y_1,X_1) \,  (\lambda M^P_{\tau} -  M^P_{\sigma}) \, M_{(\sigma_2,\tau_2)} (X_2,Y_2) $ be  GFPRs of $G(\lambda)$ and $P(\lambda)$, respectively. Then
	$$ \mathbb{L}(\lambda)  =
	\left[
	\begin{array}{c|c}
	L(\lam) &  e_{m-i_0(\sigma_1,\sigma)} \otimes C \\[.1em] \hline \\[-1em]
	e^T_{m-c_0( \sigma, \sigma_2)} \otimes B  & A-\lam E \\
	\end{array}
	\right].$$
	Thus, the map  $ \mathrm{GFPR}(P) \rightarrow \mathrm{GFPR}(G), \, L(\lam) \mapsto  \left[
	\begin{array}{@{}c|c@{}}
	L(\lam) &  e_{m-i_0 ( \sigma_1, \sigma)} \otimes C \\[.1em] \hline \\[-1em]
	e^T_{m-c_0(\sigma, \sigma_2)} \otimes B  & A-\lam E
	\end{array}
	\right] $ is a bijection, where GFPR(P) and GFPR(G) denote the set of GFPRs of $P(\lambda)$ and $G(\lambda),$ respectively.
	
\end{theorem}

\begin{proof} Let  $\sigma$ be given by $ \sigma =( \delta_1, 0, \delta_2).$ A straight forward calculation shows that
	{\small
		\begin{eqnarray}
		\mathbb{L}(\lam) &= & \mathbb{M}_{(\tau_1,\sigma_1)} (Y_1,X_1)  \left ( \lam \left[
		\begin{array}{@{}c@{\,}|@{\,}c@{}}
		M^P_\tau & 0 \\
		\hline
		0 & -E \\
		\end{array}
		\right]  - \left[
		\begin{array}{@{}c@{\,}|@{\,}c@{}}
		M^P_{\delta_1} M^P_{0} M^P_{\delta_2} &  M^P_{\delta_1}(-e_{m} \otimes C) \\
		\hline \\[-1em]
		(-e_{m}^{T} \otimes B) M^P_{\delta_2} & -A \\
		\end{array}
		\right] \right )  \mathbb{M}_{(\sigma_2,\tau_2)} (X_2,Y_2)  \nonumber  \\
		&=& \left[
		\begin{array}{@{}c@{\,}|@{\,}c@{}}
		L(\lambda) & M_{(\tau_{1},\sigma_{1})}(Y_1,X_1) M^P_{\delta_{1}}(e_{m}\otimes C) \\[.1em] \hline \\[-1em]
		(e_{m}^{T}\otimes B) M^P_{\delta_{2}}M_{(\sigma_{2},\tau_{2})}(X_2,Y_2) & A-\lambda E \\
		\end{array}
		\right].  \label{LlamP2LlamG}
		\end{eqnarray}}
	It is shown in the proof of \cite[Theorem~5.12]{rafiran2} that  $ M^P_{\delta_{1}}(e_{m}\otimes I_n) = e_{m-i_0(\sigma)}\otimes I_n$ and $ (e^T_{m}\otimes I_n) M^P_{\delta_{2}} = e^T_{m-c_0(\sigma)}\otimes I_n$. Consequently, by Lemma \ref{lemmblockrowncolp0221d17ee_RCh4}, we have
	$(e^T_{m} \otimes I_n)M^P_{\delta_2}M_{(\sigma_2,\tau_2)}(X_2,Y_2) =e^T_{m-c_0(\sigma, \sigma_2)} \otimes I_n$ and $M_{(\tau_1,\sigma_1)} (Y_1,X_1) M^P_{\delta_1}(e_{m} \otimes I_n) = e_{m-i_0(\sigma_1,\sigma)} \otimes I_n$. Hence the desired form of $\mathbb{L}(\lam)$ follows from (\ref{LlamP2LlamG}).
\end{proof}

{\begin{remark}
We mention that FPRs and GFPRs of matrix polynomials can be generated by  automatic algorithms without performing any arithmetic operations (see,  Algorithms 1,2,3 and 4, in \cite[Pages 49-52]{RanjanThesis}). Thus, in view of Theorem~\ref{gfprptoge_RCh4}, GFPRs of rational matrices can be generated by an operation-free automatic algorithm.
\end{remark}

\subsection{Fiedler-like pencils are Rosenbrock strong linearizations} We now show that Fiedler pencils (FPs), generalized Fiedler pencils (GFPs) and  GFPRs of  $G(\lam)$ are Rosenbrock strong linearizations of $G(\lam)$. First, we show that the FPs of $G(\lam)$  introduced in \cite{rafinami1} are Rosenbrock strong linearizations of $G(\lam)$.

\begin{definition}[\cite{rafinami1}, Fiedler pencil] Let $\sigma $ be a permutation of $\{0: m-1\}.$ Then
	$\mathbb{L}_{\sigma}(\lambda):= \lambda \mathbb{M}^\mathcal{S}_{-m} - \mathbb{M}^\mathcal{S}_{\sigma} $ is called a  Fiedler pencil (FP) of $G(\lambda)$ associated with $\sigma.$  The pencil $\mathbb{L}_\sigma(\lambda)$ is also referred to as a Fiedler pencil of $\mathcal{S}(\lambda).$
\end{definition}

We now define the reverse consecution-inversion structure sequence of a permutation which we need in order to prove that a Fiedler  pencil is a Rosenbrock strong linearization of $G(\lam).$

\begin{definition} Let $\alpha$ be a permutation of $ \{0:m-1\}$. Then the tuple $\RCISS (\alpha) : =(c_1,i_1,c_2,i_2, \ldots, c_\ell, i_\ell)$ is called the reverse consecution-inversion structure sequence of $\alpha$ when  $\alpha$ has consecutions at $ m-c_1-1, m-c_1 ,\ldots ,m-2$; inversions at $m-c_1-i_1-1, m-c_1-i_1,\ldots,  m-c_1-2  $ and so on, consecutions at $i_\ell, i_\ell +1,  \ldots, i_\ell +c_\ell-1  $; inversions at $0, 1, \ldots,  i_\ell-1 $.
\end{definition}

\begin{remark} It is easy to see that  $\RCISS (\alpha) = rev \big(\CISS( rev(\alpha)) \big ),$ where   $\CISS(\alpha)$ is the consecution-inversion structure sequence of $\alpha$ defined in~\cite{tdm}.
\end{remark}

\begin{example} Let $ m=11$, and let $\alpha$ and $\beta$ be  permutations of $ \{0:10\}$ given by  $ \alpha = ( 8:10, 7,6,5, 2:4, 1,0)$ and $ \beta = (10,9,5:8, 3:4, 2,0:1)$. Then $ \RCISS (\alpha) = (2,4,2,2)$ since $\alpha$ has consecutions at $8,9$; inversions at $ 4,5,6,7$; consecutions at $2,3$; inversions at $0,1$. Similarly, we have $\RCISS (\beta) = (0,2,3,1,1,2,1,0)$.
\end{example}

Let $\alpha$ be a permutation of $ \{0:m-1\}$ with RCISS$(\alpha) = (c_1,i_1,c_2,i_2, \ldots, c_\ell, i_\ell)$. We define
\begin{equation} \label{mini}
 m_0:=0, ~n_0 : =0, \text{ and } m_p:=\sum\nolimits_{j=1}^{p} c_j  \text{ and } n_p:=\sum\nolimits_{j=1}^{p} i_j \text{ for } p=1:\ell.
\end{equation}

Observe that $m_\ell= c(\sigma)$  and $n_\ell= i(\sigma)$, that is, $m_\ell$ is the total number of  consecutions of $\alpha$ and $n_\ell$ is the total number of inversions of $\alpha$. Thus $m_\ell + n_\ell = m-1$. Further, we define
\begin{equation} \label{si}
s_0 :=0 \text{ and } s_p := \sum\nolimits_{j=1}^p (c_j + i_j) \text{ for } p=1:\ell.
\end{equation}  Observe that $s_\ell = m_\ell + n_\ell = m-1$.

For $i\geq 0$ and $j\geq 0$, we define $\widehat{\Lambda}_{i,j} (\lam),$ $ \Lambda_{i,j} (\lam),$ $  \widehat{\Omega}_{i,j} (\lam) ,$ and $\Omega_{i,j} (\lam)$ as follows:
\begin{equation} \label{eqn-bc}
\widehat{\Lambda}_{i,j} (\lam):=
\left[
\begin{matrix}
I_n\\
\lambda I_n\\
\lambda^2 I_n\\
\vdots\\		
\lambda^{i-1} I_n\\
0_{jn \times n}\\
\end{matrix}\right] \in \mathbb{C}[\lam]^{(i+j)n \times n},  ~
\Lambda_{i,j}(\lam) :=
\left[
\begin{matrix}
I_n\\
\lambda I_n\\
\lambda^2 I_n\\
\vdots\\		
\lambda^{i-1} I_n\\
0_{jn \times n}\\
\lam^{i} I_n
\end{matrix}
\right] \in \mathbb{C}[\lam]^{(i+1+j)n \times n} 	,	
\end{equation}
\begin{equation} \label{eqn-br}
\widehat{\Omega}_{i,j} (\lam) :=\left[ \begin{matrix}
0_{in \times n}\\
I_n\\
\lambda I_n\\
\lambda^2 I_n\\
\vdots\\		
\lambda^{j-1} I_n\\
\end{matrix}
\right] \in \mathbb{C}[\lam]^{(i+j)n \times n}
\text{ and }
\Omega_{i,j} (\lam) :=
\left[
\begin{matrix}
0_{in \times n}\\
I_n\\
\lambda I_n\\
\lambda^2 I_n\\
\vdots\\		
\lambda^{j-1} I_n\\
\lambda^{j} I_n\\
\end{matrix}
\right] \in \mathbb{C}[\lam]^{(i+j+1)n \times n} .
\end{equation}
Note that $
\Lambda_{i,j} (\lam) =
\left[
\begin{array}{@{}c@{}}
\widehat{\Lambda}_{i,j} (\lam)\\
\lam^i I_n
\end{array}\right]$
and
$\Omega_{i,j} (\lam)=
\left[
\begin{array}{@{}c@{}}
\widehat{\Omega}_{i,j} (\lam)\\
\lam^j I_n
\end{array}\right]$. Further, $
\Lambda_{0,j} (\lam) = \left[
\begin{array}{@{}c@{}}
0_{jn \times n}\\
I_n
\end{array}\right] $  and $ \Omega_{i,0} (\lam) =\left[
\begin{array}{@{}c@{}}
0_{in \times n}\\
I_n
\end{array}\right] $.

Let $\mathcal{H} := (\mathcal{H}_{ij})$ be a block $k \times \ell$ matrix, where each block $\mathcal{H}_{ij}$ is a $p \times q$ matrix. Then the {\em block transpose} of $\mathcal{H}$ is the block $\ell \times k$ matrix $\mathcal{H}^{\mathcal{B}}$ given by
$(\mathcal{H}^{\mathcal{B}})_{ij} = \mathcal{H}_{ji}$, see~\cite{tdm}.

\begin{remark} \label{buidng_blocks}  It follows from (\ref{eqn-bc}) and (\ref{eqn-br}) that $( e_k^T \otimes I_n) \widehat{\Lambda}_{i,j} (\lam) = 0 \iff  \widehat{\Omega}_{i,j} (\lam)^{\mathcal{B}} ( e_k \otimes I_n) \neq 0$  for any $i\geq 0$, $j \geq 0$ and $1 \leq k \leq i+j$.
\end{remark}

%

\begin{definition} \label{def_1stcol_1strow} Let $\alpha$ be a permutation of $ \{0:m-1\}$ with $ \RCISS (\alpha) = (c_1,i_1,c_2,i_2, \ldots, c_\ell, i_\ell)$. We define  $\Lambda_{\alpha} (\lam) \in \mathbb{C}[\lam]^{mn \times n}$ and $\Omega_{\alpha}  (\lam) \in \mathbb{C}[\lam]^{n \times mn}$  as follows:
	\begin{equation}\label{Esigmap}
	\Lambda_{\alpha} (\lam) : =
	\left[
	\begin{matrix}
	\widehat{\Lambda}_{c_1,i_1} (\lam)\\
	\lambda^{m_1}\widehat{\Lambda}_{c_2,i_2} (\lam)\\
	\vdots\\
	\lambda^{m_{\ell-2}}\widehat{\Lambda}_{c_{\ell - 1},i_{\ell-1}} (\lam)\\
	\lambda^{m_{\ell-1}} \Lambda_{c_\ell, i_\ell} (\lam)
	\end{matrix}
	\right]~ \text{if}  ~\ell > 1,
	\end{equation}
	and $	\Lambda_{\alpha} (\lam) :  =\Lambda_{c_1,i_1} (\lam)$ if $\ell=1,$
	\begin{equation}
	\Omega_{\alpha} (\lam) : =
	\left[
	\begin{matrix}
	\widehat{\Omega}_{c_1,i_1} (\lam) \\
	\lambda^{n_1} \widehat{\Omega}_{c_2,i_2} (\lam) \\
	\vdots\\
	\lambda^{n_{\ell-2}} \widehat{\Omega}_{c_{\ell-1}, i_{\ell-1}} (\lam) \\
	\lambda^{n_{\ell-1}} \Omega_{c_\ell, i_\ell} (\lam)
	\end{matrix}
	\right]^{\mathcal{B}}~ \text{if} ~\ell > 1,
	\end{equation}
	and $  \Omega_{\alpha} (\lam) := \big( \Omega_{c_1,i_1} (\lam) \big)^{\mathcal{B}} $ if $\ell=1.$
\end{definition}

\begin{remark} \label{special_form_lam} Let $\alpha$ be a permutation of $ \{0:m-1\}$ with $\RCISS (\alpha) = (c_1,i_1, \ldots, c_\ell, i_\ell)$. Since $	 \widehat{\Lambda}_{c_j,i_j} (\lam)$ and $\widehat{\Omega}_{c_j,i_j} (\lam)$ are the basic building blocks of $\Lambda_{\alpha} (\lam)$ and $\Omega_{\alpha} (\lam) $, respectively,  it follows from Remark~\ref{buidng_blocks} that $( e_k^T \otimes I_n) \Lambda_{\alpha} (\lam)  =0 \iff \Omega_{\alpha} (\lam) ( e_k \otimes I_n) \neq 0$  for any $ k \in \{1:m -1\}$. Further, note that
	$( e_m^T \otimes I_n) \Lambda_{\alpha} (\lam) = \lam^{m_\ell} I_n$ and $\Omega_{\alpha} (\lam) ( e_m \otimes I_n) = \lam^{n_\ell} I_n $.
\end{remark}

\begin{definition}[\cite{tdm}, Horner shift] \label{defhornershift}
	Let $ P(\lam) = \sum\nolimits_{i=0}^m \lam^{i}  A_i$. For $k = 0:m$, the matrix polynomial  $P_{k}(\lambda) := A_{m-k} + \lambda A_{m-k+1} + \cdots + \lambda^{k}A_{m}$ is called the  Horner shift of $P(\lambda)$ of degree $k$.
\end{definition}

 For $ 1 \leq i \leq m-1$, we consider the following $mn \times mn$ unimodular matrix polynomials \cite{tdm}
$$  Q_i(\lam) : =
\left[
\begin{array}{cccc}
I_{(i-1)n} & &&\\
& I_n & \lam I_n & \\
& 0_n & I_n &\\
&&& I_{(m-i-1)n}
\end{array}
\right]   $$ and
$$ R_i(\lam) : =
\left[
\begin{array}{cccc}
I_{(i-1)n} & &&\\
& 0_n &  I_n & \\
& I_n & P_i(\lam) &\\
&&& I_{(m-i-1)n}
\end{array}
\right] = R_i^{\mathcal{B}} (\lam). $$
Observe that $ R_i(\lam)$ depends on the Horner shifts of $P(\lam)$ whereas $ Q_i(\lam)$ does not.
For simplicity, we write $ Q_i$ and $R_i$ for $ Q_i (\lam)$ and $R_i (\lam)$, respectively.

 We need the following results in order to prove that Fiedler  pencils are Rosenbrock strong linearizations of $G(\lam).$


\begin{lemma} \label{lem_full_expre} Let $ P(\lam)$ be a matrix polynomial of degree $m$ and $\alpha$ be a permutation of $ \{0:m-1\}$. Suppose that $\RCISS (\alpha) = (c_1,i_1,c_2,i_2, \ldots, c_\ell, i_\ell)$. For $j =1:\ell,$  set
\begin{align}
 U_{(c_j,i_j)} : =  R^{\mathcal{B}}_{s_{j-1}+c_j+i_j} \cdots R^{\mathcal{B}}_{s_{j-1}+c_j+1} Q^{\mathcal{B}}_{s_{j-1}+c_j} \cdots Q^{\mathcal{B}}_{s_{j-1}+1} \label{uBij} \\
\text{ and }  V_{(c_j,i_j)} :=   R_{s_{j-1}+1} \cdots R_{s_{j-1}+c_j} Q_{s_{j-1}+c_j +1} \cdots Q_{s_{j-1}+c_j+i_j} . \label{uij}
\end{align}	
Let $U(\lam)$ and $V(\lam)$ be given by		
	$U(\lam) : = U_{(c_\ell,i_\ell)} U_{(c_{\ell -1},i_{\ell -1})} \cdots U_{(c_2,i_2)} U_{(c_1,i_1)} \text{ and } $
	$ V(\lam) : = V_{(c_1,i_1)}   V_{(c_2,i_2)}\cdots   V_{(c_{\ell -1},i_{\ell -1})} V_{(c_\ell,i_\ell)}. $
	Then $$U(\lam) (e_1 \otimes I_n) = \Lambda_\alpha (\lam) \mbox{ and } (e^T_1 \otimes I_n) V(\lam) = \Omega_\alpha (\lam),$$ where $\Lambda_\alpha (\lam)$ and $\Omega_\alpha (\lam)$ are as given in Definition~\ref{def_1stcol_1strow}.
\end{lemma}

We prove Lemma~\ref{lem_full_expre} in Appendix~\ref{appendix1}.

\begin{proposition} \label{prop_elimination} Let $  X (\lam) := [\mathbf{x}_1~ ~\mathbf{x}_2~ \cdots ~  \mathbf{x}_m]^\mathcal{B}$  and $ Y(\lam) := [\mathbf{y}_1 ~~\mathbf{y}_2~ \cdots ~  \mathbf{y}_m]$, where $\mathbf{x}_i = 0$ or $ \mathbf{x}_i = \lam^{p_i} I_n $, and $\mathbf{y}_i = 0$ or $\mathbf{y}_i = \lam^{q_i} I_n $, for some $p_i\geq 0 $ and $ q_i \geq 0 $, $i=1:m$. Suppose that $\mathbf{x}_i \mathbf{y}_i = 0$ for $i = 1:m-1$. Then there exist an $m\times m$ lower block-triangular matrix polynomial $L(\lam)$ with diagonal blocks $I_n$ and an $m\times m$ upper block-triangular matrix polynomial $U(\lam)$ with diagonal blocks $I_n$ such that
	\begin{equation} \label{gauss_el_type}
	L(\lam)
	\Big(
	\left[
	\begin{array}{c|c}
	I_{(m-1)n} & 0 \\
	\hline
	0 &0\\
	\end{array}
	\right]
	+ X(\lam)  Y (\lam)  \Big) U(\lam)  =
	\left[ \begin{array}{c|c}
	I_{(m-1)n} & 0 \\
	\hline
	0 & \mathbf{x}_m \mathbf{y}_m  \\
	\end{array}
	\right].
	\end{equation}
\end{proposition}

We prove Proposition~\ref{prop_elimination} in Appendix~\ref{appendix2}. As an immediate corollary we have the following result.

\begin{corollary} \label{coro:biprt1t2} Let $\alpha$ be a permutation of $ \{0:m-1\}$ with RCISS$(\alpha) = (c_1,i_1,c_2,i_2, \ldots, c_\ell, i_\ell)$. Consider $\Lambda_\alpha (\lam)$ and $\Omega_\alpha (\lam)$ associated with RCISS$(\alpha)$ as given in Definition~\ref{def_1stcol_1strow}.	Then there exist an $m\times m$ lower block-triangular matrix polynomial $T_1(\lam)$ with diagonal blocks $I_n$ and an $m\times m$ upper block-triangular matrix polynomial $T_2(\lam)$ with diagonal blocks $I_n$  such that
	$$ T_1(\lam) \; \Big (   \left[ \begin{array}{@{}c|c@{}}
	I_{(m-1)n} & 0 \\
	\hline
	0 & 0 \\
	\end{array}
	\right] + \Lambda_\alpha (\lam)  \; \Omega_\alpha (\lam) \;  \Big ) \; T_2(\lam)  =
	\left[ \begin{array}{@{}c|c@{}}
	I_{(m-1)n} & 0 \\
	\hline
	0 & \lam^{m-1} I_n \\
	\end{array}
	\right].$$
\end{corollary}

\begin{proof} Note that we have
	$( e_m^T \otimes I_n) \; \Lambda_{\alpha} (\lam)= \lam^{m_\ell} I_n$ and $\Omega_{\alpha}  (\lam)\; ( e_m \otimes I_n) = \lam^{n_\ell} I_n $  and  that $m_\ell + n_\ell = m-1$. By
Remark~\ref{special_form_lam}, it follows that   $\Lambda_\alpha (\lam)$ and $ \Omega_\alpha (\lam)$ satisfy the conditions of Proposition~\ref{prop_elimination}.    Hence the result follows from  Proposition~\ref{prop_elimination}.
\end{proof}

We now prove that Fiedler pencils are Rosenbrock strong linearizations of $G(\lam)$. For any index tuples $\alpha$ and $\beta$ containing indices from $\{0:m-1\}$, we write $\alpha \sim \beta$ if $\mathbb{M}^{\mathcal{S}}_\alpha = \mathbb{M}^{\mathcal{S}}_\beta$. Let $\sig$ be a permutation of $\{0:m-1\}$ with $c_0(\sig) >0$. Since $\mathbb{M}^\mathcal{S}_{i} \mathbb{M}^\mathcal{S}_{j} = \mathbb{M}^\mathcal{S}_{j} \mathbb{M}^\mathcal{S}_{i}$ for $ |i-j| > 1$, $i,j \in \{0:m-1\}$, we have  $ \sig \sim (\sig^L, 0, 1, \ldots, c_0(\sig))$, where $\sig^L$ is a sub-permutation of $\{0:m-1\} \setminus \{0,1,\ldots,c_0(\sig)\}$. Similarly, if $\sig$ is a permutation of $\{0:m-1\}$ with $i_0(\sig) >0$ then $ \sig \sim ( i_0(\sig), \ldots, 1, 0, \sig^R )$, where $\sig^R$ is a sub-permutation of $\{0:m-1\} \setminus \{0,1,\ldots,i_0(\sig)\}$. The reversal of a matrix polynomial $P(\lambda) := \sum_{j=0}^m \lambda^j A_j$ is defined by $rev P(\lam) := \sum \nolimits_{j=0}^{m} \lam^{j}A_{m-j}$.

\begin{theorem} \label{FP:Rosen:stro:lin} Let $\mathbb{L}_{\sigma}(\lambda):= \lambda \mathbb{M}^\mathcal{S}_{-m} - \mathbb{M}^\mathcal{S}_{\sigma} $ be the  Fiedler pencil of $G(\lambda)$ associated with a permutation $\sig$ of  $\{0: m-1\}$. Then  $\mathbb{L}_{\sigma}(\lambda)$ is a Rosenbrock strong linearization of $G(\lam)$. More precisely, we have the following.
	
	(a) There exist $mn\times mn$ unimodular matrix polynomials $U(\lam)$ and $V(\lam)$, and $r\times r$ nonsingular matrices $U_0$ and $V_0$  such that
	\be
	\left[
	\begin{array}{c|c}
		U(\lam) & 0 \\
		\hline
		0 & U_0 \\
	\end{array}
	\right] \mathbb{L}_\sig (\lam) \left[
	\begin{array}{c|c}
		V(\lam) & 0 \\
		\hline
		0 & V_0 \\
	\end{array}
	\right] = \left[
	\begin{array}{c|c}
		I_{(m-1)n} & 0 \\
		\hline
		0 &  \mathcal{S}(\lam) \\
	\end{array}
	\right] \, \mbox{ for all }\,  \lam \in \C. \nonumber \ee
	
(b)	There exist biproper rational  matrices $\mathcal{O}_{\ell} (\lam)  $ and $\mathcal{O}_r (\lam) $ such that
	\begin{equation} \label{eqv_inf_biproe}
	\mathcal{O}_{\ell} (\lam)  \, \lam^{-1} \mathbb{G} (\lam)  \, \mathcal{O}_r (\lam) =
	\left[ \begin{array}{c|c}
	I_{(m-1)n} & 0 \\
	\hline
	0 & \lam^{-m} G(\lam) \\
	\end{array}
	\right]   ,
	\end{equation}
	where $\mathbb{G}(\lam) = L_{\sig}(\lam) + (e_{m-i_0(\sig)} \otimes C) (\lam E -A)^{-1} (e^T_{m-c_0(\sig)} \otimes B)  $ is the transfer function of $ \mathbb{L}_{\sig}(\lam)$ and $	 L_{\sig}(\lam)  := \lam M^P_{-m} - M^P_{\sig}$ is the Fiedler pencil of $P(\lam)$ associated with $\sig.$
\end{theorem}

\begin{proof} Part (a) is proved in \cite[Theorem~4.13]{rafinami1}.  Hence we only prove (b).
	
By Theorem~\ref{gfprptoge_RCh4}, we have 	
$\mathbb{L}_{\sigma}(\lambda) =
	\left[
	\begin{array}{@{}c|c@{}}
	L_{\sig}(\lam) &  e_{m-i_0(\sigma)} \otimes C \\ \hline \\[-1em]
	e^T_{m-c_0( \sigma)} \otimes B  & A-\lam E \\
	\end{array}
	\right].$ Hence $\mathbb{G}(\lam)$ is the transfer function of $\mathbb{L}_{\sigma}(\lambda).$   Let $\sig$ be given by  $\sig = (\delta_1, 0 , \delta_2)$. Then we have $L_{\sig} (\lam) = \lam M^P_{-m} - M^P_{\sig} =  \lam M^P_{-m} - M^P_{\delta_1} M^P_0 M^P_{\delta_2}.  $ It is shown in the proof of \cite[Theorem~4.6]{tdm} that $- rev L_{\sig} (\lam)$ is strictly equivalent to $- rev P (\lam)$. More precisely,  $J (M^P_{\delta_1})^{-1} \big(- rev L_{\sig} (\lam) \big)  (M^P_{\delta_2})^{-1} J  = : L_{\alpha} (\lam) $ is a Fiedler pencil of $- rev P (\lam)$, where $ \alpha = \big ( m-rev(\delta_1), 0 , m-rev(\delta_2) \big)$ and
	$J := \left[\begin{array}{@{}ccc@{}}
	& & I_n\\
	& \rotatebox{45}{$\cdots$} &\\
	I_n & &
	\end{array} \right] \in \mathbb{C}^{mn \times mn}.$ Hence $L_{\alpha} (\lam) $ is a linearization of $- rev P (\lam)$. Thus there exist unimodular matrix polynomials $\widehat{U} (\lam)$ and   $\widehat{V} (\lam)$ such that
	$$  \widehat{U} (\lam) \,  L_{\alpha} (\lam) \,  \widehat{V} (\lam)  =
	\left[ \begin{array}{c|c}
	-I_{(m-1)n} & 0 \\
	\hline
	0 & - rev P(\lam)\\
	\end{array}
	\right]  , $$
	where $ \widehat{U} (\lam)$ and $ \widehat{V} (\lam)$ are given by \cite{tdm}
	\begin{equation} \label{U_cap}
	\widehat{U} (\lam): = U_0 U_1 \cdots U_{m-3} U_{m-2}, \text{ with } U_j = \left\{ \begin{array}{ll}
	Q^{\mathcal{B}}_{m-1-j} & \text{if } \alpha \text{ has a consecution at } j, \\
	R^{\mathcal{B}}_{m-1-j} & \text{if } \alpha \text{ has an inversion at } j,
	\end{array} \right.
	\end{equation}
	\begin{equation}\label{V_cap}
	\widehat{V} (\lam) := V_{m-2} V_{m-3} \cdots V_{1} V_{0}, \text{ with } V_j = \left\{ \begin{array}{ll}
	R_{m-1-j} & \text{if } \alpha \text{ has a consecution at } j, \\
	Q_{m-1-j} & \text{if } \alpha \text{ has an inversion at } j.
	\end{array} \right.
	\end{equation}
	Note that $R_i^{\text{'}}s$ in (\ref{U_cap}) and (\ref{V_cap}) are associated with  the matrix polynomial $- rev P(\lam)$.	Thus we have	
{\small \begin{align}
& \left[ \begin{array}{@{}c@{}|c@{}}
-I_{(m-1)n} & 0 \\
\hline
0 & - rev P(\lam)\\
\end{array}
\right]  = 	\widehat{U} (\lam)  L_{\alpha} (\lam)  \widehat{V} (\lam) = \widehat{U} (\lam)    J (M^P_{\delta_1})^{-1} \big(- rev L_{\sig} (\lam) \big)  (M^P_{\delta_2})^{-1} J    \widehat{V} (\lam) \nonumber \\
& \Rightarrow
\left[ \begin{array}{@{}c@{}|c@{}}
I_{(m-1)n} & 0 \\
\hline
0 & \lam^{-m} P(\lam)\\
\end{array}
\right] =  \widehat{U} (1/\lam) \,  J (M^P_{\delta_1})^{-1} \Big(\lam^{-1} L_{\sig} (\lam) \Big)  (M^P_{\delta_2})^{-1} J  \,  \widehat{V} (1/\lam) . \label{eqn:1stpart}
\end{align}}

Next we evaluate
	\begin{equation} \label{eqn:ujdel}
	\widehat{U} (\lam)  J (M^P_{\delta_1})^{-1} (e_{m-i_0(\sig)} \otimes I_n) \text{ and } (e^T_{m-c_0(\sig)} \otimes I_n)  (M^P_{\delta_2})^{-1} J \widehat{V} (\lam).
	\end{equation}
Recall that $ M^P_{-t}  = \left[
	\begin{array}{@{}cccc@{}}
	I_{(m-t-1)n} &  & & \\
	& 0 & I_n &\\
	& I_n & A_{t} &\\
	
	& & & I_{(t-1)n} \end{array}
\right]$ for $t=1:m-1$. Hence we have
\begin{equation}\label{-veBrow}
(e^T_{m-q} \otimes I_n) M^P_{-t}=
\left\{
\begin{array}{ll}
e^T_{m-(q-1)} \otimes I_n & \mbox{for}~t= q~ \mbox{and}~ q=1:m-1, \\
e^T_{m-q} \otimes I_n &    \mbox{for} ~t \notin \{q,q+1\},~ q=0:m-1,
\end{array}
\right.
\end{equation} and
\begin{equation}\label{-veBcol}
M^P_{-t} (e_{m-q} \otimes I_n) =
\left\{
\begin{array}{ll}
e_{m-(q-1)} \otimes I_n & \mbox{for}~t = q~ \mbox{and}~ q=1:m-1, \\
e_{m-q} \otimes I_n & \mbox{for}~t \notin \{q,q+1\},~ q=0:m-1.
\end{array}
\right.
\end{equation}

	Case-I: Suppose that $c_0(\sig) >0.$ Then $ i_0(\sig) = 0.$ Since $\sig $ has $c_0(\sig)$ consecutions at $0$, we have  $ \sig \sim (\sig^L, 0, 1,2, \ldots, c_0(\sig))$. Without loss of generality, we assume that  $ \sig = (\delta_1 , 0, \delta_2) = ( \delta_1, 0, 1,2, \ldots, c_0(\sig))$, that is, $\delta_2 = (1,2, \ldots, c_0(\sig)). $  Then by repeated application of (\ref{-veBrow}) we have
	$$ (e^T_{m-c_0(\sig)} \otimes I_n)  (M^P_{\delta_2})^{-1} = (e^T_{m-c_0(\sig)} \otimes I_n)  M^P_{-c_0(\sig)} M^P_{-(c_0(\sig)-1)} \cdots M^P_{-2} M^P_{-1}  = e^T_{m} \otimes I_n .$$  Hence $  (e^T_{m-c_0(\sig)} \otimes I_n)  (M^P_{\delta_2})^{-1} J  = (e^T_{m} \otimes I_n)  J =  e^T_{1} \otimes I_n.$
	
	Further, since $i_0(\sig) = 0$ and $0,1 \notin \delta_1$, by (\ref{-veBcol}) we have $ (M^P_{\delta_1})^{-1} (e_{m-i_0(\sig)} \otimes I_n) =  (M^P_{\delta_1})^{-1} (e_{m} \otimes I_n) =   e_{m} \otimes I_n.$ Hence $J (M^P_{\delta_1})^{-1} ( e_{m-i_0(\sig)} \otimes I_n) = J ( e_{m} \otimes I_n) = e_{1} \otimes I_n $.
	
	Case-II: Suppose that $i_0(\sig) >0.$ Then $ c_0(\sig) = 0.$ Since $\sig $ has $i_0(\sig)$ inversions at $0$, we have $ \sig \sim ( i_0(\sig), \ldots,  2,1, 0, \sig^R )$. Without loss of generality, we assume that $ \sig  = (\delta_1 , 0, \delta_2) =  ( i_0(\sig), \ldots,  2,1, 0, \delta_2 )$, that is, $\delta_1 = (i_0(\sig), \ldots ,2, 1). $ Then by repeated application of (\ref{-veBcol}) we have
	$$ (M^P_{\delta_1})^{-1} (e_{m-i_0(\sig)} \otimes I_n) = M^P_{-1} M^P_{-2}   \cdots M^P_{-(i_0(\sig)-1)}  M^P_{-i_0(\sig)} (e_{m-i_0(\sig)} \otimes I_n)  = e_{m} \otimes I_n.$$  Hence we have  $ J (M^P_{\delta_1})^{-1} (e_{m-i_0(\sig)} \otimes I_n) = J(  e_{m} \otimes I_n) =  e_{1} \otimes I_n . $
	
	Further, since $c_0(\sig) = 0$ and $0,1 \notin \delta_2$, by (\ref{-veBrow}) we have $ (e^T_{m-c_0(\sig)} \otimes I_n)  (M^P_{\delta_2})^{-1}  =  (e^T_{m} \otimes I_n) (M^P_{\delta_2})^{-1} =  e^T_{m} \otimes I_n $. Hence $ (e^T_{m-c_0(\sig)} \otimes I_n) (M^P_{\delta_2})^{-1} J =  (e^T_{m} \otimes I_n) J =  e^T_{1} \otimes I_n $.
	
	Thus in both the cases, we have
	\begin{equation} \label{eqn_jdel}
	\begin{array}{l}
	\widehat{U} (\lam) J (M^P_{\delta_1})^{-1} (e_{m-i_0(\sig)} \otimes I_n) = \widehat{U} (\lam)(e_{1} \otimes I_n) \\ [.5em]
	(e^T_{m-c_0(\sig)} \otimes I_n)  (M^P_{\delta_2})^{-1} J \widehat{V} (\lam) = (e^T_{1} \otimes I_n) \widehat{V} (\lam).
	\end{array}
	\end{equation}
	
Next, we calculate $\widehat{U} (\lam)(e_{1} \otimes I_n)$ and $(e^T_{1} \otimes I_n) \widehat{V} (\lam)$. Note that $\alpha$ is a permutation of $\{0:m-1\}$. Let $\RCISS (\alpha)$ be given by  $\RCISS (\alpha)=(c_1,i_1,c_2,i_2, \ldots ,c_\ell,i_\ell)$. Recall from (\ref{mini}) and (\ref{si}) the definitions of $m_j,n_j$ and $s_j$, for $j= 0:\ell$, associated with $\RCISS (\alpha)$. By (\ref{U_cap}) and (\ref{V_cap}), we have
	$$\widehat{U}(\lam)  = \widehat{U}_{(c_\ell,i_\ell)} \widehat{U}_{(c_{\ell -1},i_{\ell -1})} \cdots \widehat{U}_{(c_2,i_2)} \widehat{U}_{(c_1,i_1)} \text{ and } $$
	$$ \widehat{V}(\lam)  = \widehat{V}_{(c_1,i_1)}   \widehat{V}_{(c_2,i_2)}\cdots   \widehat{V}_{(c_{\ell -1},i_{\ell -1})} \widehat{V}_{(c_\ell,i_\ell)}, $$
	where   $ \widehat{U}_{(c_j,i_j)} =  R^{\mathcal{B}}_{s_{j-1}+c_j+i_j} \cdots R^{\mathcal{B}}_{s_{j-1}+c_j+1} Q^{\mathcal{B}}_{s_{j-1}+c_j} \cdots Q^{\mathcal{B}}_{s_{j-1}+1} \text{ and } $
	$$  \widehat{V}_{(c_j,i_j)} =   R_{s_{j-1}+1} \cdots R_{s_{j-1}+c_j} Q_{s_{j-1}+c_j +1} \cdots Q_{s_{j-1}+c_j+i_j} .$$
	Hence by Lemma~\ref{lem_full_expre}, we have
	$\widehat{U}(\lam) (e_1 \otimes I_n) = \Lambda_\alpha (\lam)$ and $ (e^T_1 \otimes I_n) \widehat{V}(\lam) = \Omega_\alpha (\lam)$, where $\Lambda_\alpha (\lam)$ and $\Omega_\alpha (\lam)$ are as given in Definition~\ref{def_1stcol_1strow}. Now  by (\ref{eqn_jdel}) we have
	\begin{eqnarray}  \label{eqn:UV}
	\widehat{U} (\lam) J (M^P_{\delta_1})^{-1} (e_{m-i_0(\sig)} \otimes I_n) = \Lambda_\alpha (\lam) \text{ and } (e^T_{m-c_0(\sig)} \otimes I_n)  (M^P_{\delta_2})^{-1} J \widehat{V} (\lam) = \Omega_\alpha (\lam). \nonumber \\
	\end{eqnarray}

 Define  $ \widehat{\mathcal{O}}_{\ell} (\lam) : = \widehat{U} (1/\lam) J (M^P_{\delta_1})^{-1}  $ and $\widehat{\mathcal{O}}_r (\lam)  : =  (M^P_{\delta_2})^{-1} J \widehat{V} (1/\lam)$. Since $\widehat{U} (\lam) $ and $\widehat{V} (\lam) $ are unimodular, $ \widehat{U} (1/\lam)$ and  $\widehat{V} (1/\lam)$ are biproper. Hence it follows that $\widehat{\mathcal{O}}_{\ell} (\lam) $ and $\widehat{\mathcal{O}}_r (\lam)$ are biproper matrices.  Set $ G_{sp}(\lam) := C(\lam E -A)^{-1}B$. Then we have
\begin{align}
	 & \widehat{\mathcal{O}}_{\ell}  (\lam) \lam^{-1}  \mathbb{G} (\lam) \widehat{\mathcal{O}}_{r} (\lam)  \nonumber\\
	& = \widehat{\mathcal{O}}_{\ell}  (\lam) \lam^{-1}  L_\sig (\lam) \widehat{\mathcal{O}}_{r} (\lam) + \widehat{\mathcal{O}}_{\ell}  (\lam) \Big ( (e_{m-i_0(\sig)} \otimes I_n) \lam^{-1} G_{sp}(\lam) (e^T_{m-c_0(\sig)} \otimes I_n) \Big) \widehat{\mathcal{O}}_{r} (\lam)  \nonumber \\
	& =\left[ \begin{array}{@{}c|c@{}}
	I_{(m-1)n} & 0 \\
	\hline
	0 & \lam^{-m} P(\lam) \\
	\end{array}
	\right]  +   \Lam_{\alpha} (1/\lam) \;  \lam^{-1} G_{sp}(\lam)  \; \Omega_{\alpha} (1/\lam) \nonumber \\
	& =\left[ \begin{array}{@{}c|c@{}}
	0_{(m-1)n} & 0 \\
	\hline
	0 & \lam^{-m} P(\lam) \\
	\end{array}
	\right]  + \underbrace{\left[ \begin{array}{@{}c|c@{}}
		I_{(m-1)n} & 0 \\
		\hline
		0 & 0 \\
		\end{array}
		\right]  + \Lam_{\alpha} (1/\lam) \;  \lam^{-1} G_{sp}(\lam)  \;  \Omega_{\alpha} (1/\lam) }_{W(\lam)}, \label{eqn_mainthm}\\
	 & \text{where the second equality holds by } (\ref{eqn:1stpart}) \text{ and }  (\ref{eqn:UV}). \nonumber
	\end{align}
	
Let $T_1(\lam)$ and $T_2(\lam)$ be the matrix polynomials given in Corollary~\ref{coro:biprt1t2}. Since $T_1(\lam)$ and $T_2(\lam)$ are block upper triangular with diagonal blocks $I_n$,  $T_1(1/\lam)$ and $T_2(1/\lam)$ are biproper rational matrices.
Let  $\widehat{T}_j(1/\lam)$, $j=1,2,$ denote the matrix obtained by multiplying each off diagonal block of $T_j(1/\lam)$  by $-\lam^{-1} G_{sp}(\lam)$. Then obviously $\widehat{T}_j(1/\lam)$ is biproper for $j=1,2.$
Now by Corollary~\ref{coro:biprt1t2} we have
	\begin{eqnarray}
	&&\widehat{T}_1(1/\lam) \;W(\lam) \; \widehat{T}_2(1/\lam)  =
	\left[ \begin{array}{@{}c|c@{}}
	I_{(m-1)n} & 0 \\
	\hline
	0 & \lam^{-(m-1)} \lam^{-1} G_{sp}(\lam)\\
	\end{array}
	\right] =
	\left[ \begin{array}{@{}c|c@{}}
	I_{(m-1)n} & 0 \\
	\hline
	0 & \lam^{-m} G_{sp}(\lam)\\
	\end{array}
	\right]. \nonumber\\
	\label{eqn:2ndpart}
	\end{eqnarray}
Hence by defining $\mathcal{O}_{\ell} (\lam) : = \widehat{T}_1(1/\lam) \widehat{\mathcal{O}}_{\ell} (\lam) $ and $ \mathcal{O}_{r} (\lam) : = \widehat{\mathcal{O}}_r (\lam)  \widehat{T}_2(1/\lam) $, the equality in (\ref{eqv_inf_biproe}) follows from (\ref{eqn_mainthm}) and (\ref{eqn:2ndpart}).  This completes the proof of (b).
\end{proof}

Next, we show that  GFPRs of $G(\lam)$ are Rosenbrock strong linearizations of  $G(\lam).$ We need the following result.

\begin{proposition} \label{prop:gfprstlin}  Let $ \mathbb{T} (\lam) : =\lambda \mathbb{M}^\mathcal{S}_{-m}  -  \mathbb{M}^\mathcal{S}_{\alpha}$ be the Fiedler pencil of $G(\lam)$ associated with a permutation $\alpha$ of $\{ 0:m-1\}$.   Let $\mathbb{L} (\lam)$ be a pencil given by $\mathbb{L} (\lam) : = \diag (\mathcal{X}, X_0)   \mathbb{T}(\lam)  \diag (\mathcal{Y}, Y_0)  $, where $\mathcal{X}, \mathcal{Y} \in \mathbb{C}^{mn \times mn}$ and  $X_0, Y_0 \in \mathbb{C}^{r \times r}$ are nonsingular matrices. Then  $\mathbb{L} (\lam)$ is a Rosenbrock strong linearization of $G(\lam)$.
\end{proposition}

\begin{proof} Since $\mathbb{T}(\lam)$ is a Fiedler pencil of $G(\lam)$, by Theorem~\ref{FP:Rosen:stro:lin}, $\mathbb{T}(\lam)$ is a Rosenbrock strong linearization of $G(\lam)$. Hence there exist $mn\times mn$ unimodular matrix polynomials $U(\lam)$ and $V(\lam)$, and $r\times r$ nonsingular matrices $U_0$ and $V_0$  such that
	\begin{align}
	  \diag \big( I_{(m-1)n} , \, \mathcal{S}(\lam) \big) &=\diag \big( U(\lam),\,  U_0 \big) \, \mathbb{T}(\lam) \, \diag \big( V(\lam),  \, V_0 \big)   \nonumber\\
	 &= \diag \big(U(\lam) \mathcal{X}^{-1}, \, U_0 X^{-1}_0 \big)  \, \mathbb{L}(\lam)  \, \diag \big( \mathcal{Y}^{-1} V(\lam), \, Y^{-1}_0 V_0 \big) . \label{ts_st_lin_1}
	\end{align}
	
By Theorem~\ref{gfprptoge_RCh4}, we have
	 $\mathbb{T}(\lambda)  =
	\left[
	\begin{array}{@{}c|c@{}}
		L(\lam) &  e_{m-i_0(\alpha)} \otimes C \\[.1em] \hline \\[-1em]
		e^T_{m-c_0( \alpha)} \otimes B  & A-\lam E \\
	\end{array}
	\right],$ where $L(\lam) = \lambda M^P_{-m}  - M^P_{\alpha}$ is the Fiedler pencil of $P(\lam)$ associated with $\alpha.$ Then $$\mathbb{L} (\lam) =
	\left[
	\begin{array}{@{}c|c@{}}
	\mathcal{X} L (\lam) \mathcal{Y} &  \mathcal{X}( e_{m-i_0(\alpha)} \otimes C)Y_0 \\[.1em] \hline \\[-1em]
	X_0(e^T_{m-c_0( \alpha)} \otimes B) \mathcal{Y}  & X_0(A-\lam E)Y_0 \\
	\end{array}
	\right] $$ and $\mathbb{G}_{\mathbb{L}}(\lam) : = \mathcal{X} L (\lam) \mathcal{Y} + \mathcal{X}( e_{m-i_0(\alpha)} \otimes C)  (\lam E - A)^{-1} 	(e^T_{m-c_0( \alpha)} \otimes B) \mathcal{Y} $ is the transfer function of $\mathbb{L}(\lam)$. Since $\mathbb{T}(\lam)$ is a Rosenbrock strong linearization of $G(\lam)$, there exist biproper rational  matrices $\mathcal{O}_{\ell} (\lam)  $ and $\mathcal{O}_r (\lam) $ such that
	\begin{equation}  \label{stlinT}
	\mathcal{O}_{\ell} (\lam)  \, \lam^{-1} \mathbb{G}_{\mathbb{T}} (\lam)  \, \mathcal{O}_r (\lam) =
	\left[ \begin{array}{c|c}
	I_{(m-1)n} & 0 \\
	\hline
	0 & \lam^{-m} G(\lam) \\
	\end{array}
	\right]   ,
	\end{equation}
	where $\mathbb{G}_{\mathbb{T}} (\lam) = L(\lam) + (e_{m-i_0(\alpha)} \otimes C) (\lam E -A)^{-1} (e^T_{m-c_0(\alpha)} \otimes B)  $ is the transfer function of $ \mathbb{T}(\lam)$. Since $\mathcal{X}^{-1} \mathbb{G}_{\mathbb{L}} (\lam) \mathcal{Y}^{-1} =  \mathbb{G}_{\mathbb{T}} (\lam)  $, it follows from (\ref{stlinT}) that
	\begin{eqnarray} \label{ts_st_lin_2}
	\mathcal{O}_{\ell} (\lam)  \mathcal{X}^{-1} \lam^{-1} \mathbb{G}_{\mathbb{L}} (\lam) \mathcal{Y}^{-1}\mathcal{O}_{r} (\lam) = 	\left[ \begin{array}{c|c}
	I_{(m-1)n} & 0 \\
	\hline
	0 & \lam^{-m} G(\lam) \\
	\end{array}
	\right] .
	\end{eqnarray}
Note that 	$\mathcal{O}_{\ell} (\lam)  \mathcal{X}^{-1}$ and $\mathcal{Y}^{-1}\mathcal{O}_{r} (\lam) $ are biproper rational matrices. Hence it follows from (\ref{ts_st_lin_1}) and (\ref{ts_st_lin_2}) that $\mathbb{L}(\lam)$  is a Rosenbrock strong linearization of $G(\lam)$. \end{proof}

Now we prove that the GFPRs of $G(\lam)$ are Rosenbrock strong linearizations.

\begin{theorem} \label{thm:gfpr:RSL} Let $\mathbb{L}(\lambda) :=  \mathbb{M}_{\tau_1} (Y_1) \, \mathbb{M}_{\sigma_1} (X_1) (\lambda \mathbb{M}^\mathcal{S}_{\tau}  -  \mathbb{M}^\mathcal{S}_{\sigma}) \mathbb{M}_{\sigma_2}(X_2) \, \mathbb{M}_{\tau_2} (Y_2)$ be a GFPR of $G(\lambda)$  as given in Definition~\ref{p02gfprglam2m181a_RCh4}, where all the matrix assignments $X_j$ and $Y_j$, $j=1,2,$ are nonsingular. Then $\mathbb{L}(\lambda)$ is a Rosenbrock strong linearization of $G(\lam)$.
\end{theorem}

\begin{proof}  Let $\tau$ be given by $\tau = (\beta, -m , \gamma)$. Define $\alpha : = (- rev (\beta) , \sig, - rev (\gamma) )$ and $\mathbb{T}(\lam) : = \lambda \mathbb{M}^\mathcal{S}_{-m}  -  \mathbb{M}^\mathcal{S}_{\alpha}$. Then $ \mathbb{T}(\lam)$ is a Fiedler pencil of $G(\lam)$ associated with the permutation $\alpha$ of $\{ 0:m-1\}$. It is easily seen that 	 $\mathbb{L}(\lam) =  \mathcal{A} \mathbb{T} (\lam) \mathcal{B}$,
	where $   \mathcal{A} = \mathbb{M}_{(\tau_1,\sigma_1)} (Y_1,X_1) \mathbb{M}^\mathcal{S}_\beta  = \left[ \begin{array}{@{}c@{}|c@{}}
	M_{(\tau_1,\sig_1)} (Y_1,X_1) 	M^P_{\beta}  &\\ \hline & I_r
	\end{array} \right]  $ and
	$     \mathcal{B} = \mathbb{M}^\mathcal{S}_\gamma \mathbb{M}_{(\sigma_2,\tau_2)} (X_2,Y_2) =
	\left[ \begin{array}{@{}c@{}|c@{}}
	M^P_{\gamma}	M_{(\sig_2, \tau_2)} (X_2,Y_2)  &\\ \hline & I_r
	\end{array} \right] $. Since $X_j$ and $Y_j$, $j =1,2,$ are nonsingular matrix assignments, the matrices $M_{(\tau_1,\sigma_1)} (Y_1,X_1) $ and $ M_{(\sigma_2,\tau_2)} (X_2,Y_2)$ are nonsingular. Hence by Proposition~\ref{prop:gfprstlin}, $\mathbb{L}(\lambda)$ is a Rosenbrock strong linearization of $G(\lam)$.
\end{proof}

Finally, we show that the GFPs of $G(\lam)$ are also Rosenbrock strong linearizations.

\begin{definition}[\cite{rafinami3}, GFP]
	Let $ \omega := (\omega_0, \omega_1)$ be a permutation of  $\{0: m\}.$ Then the pencil $\mathbb{T}_\omega(\lambda) := \lambda \mathbb{M}^{\mathcal{S}}_{-\omega_1} - \mathbb{M}^{\mathcal{S}}_{\omega_0}$ is said to be a generalized Fiedler pencil (GFP) of $G(\lambda)$ associated with the permutation $\omega$.
\end{definition}

\begin{theorem} Let $\mathbb{T}_\omega(\lambda) := \lambda \mathbb{M}^\mathcal{S}_{-\omega_1}  -  \mathbb{M}^\mathcal{S}_{\omega_0}$ be a GFP of $G(\lam)$, where $0 \in \omega_0$. Then $\mathbb{T}_\omega(\lambda)$  is a Rosenbrock strong linearization of $G(\lam)$.
\end{theorem}

\begin{proof} It is shown in \cite[Theorem~2.13]{rafinami3} that $\mathbb{T}_\omega(\lambda) = \diag (\mathcal{X},\ X_0)  \ \mathbb{F}(\lam) \ \diag (\mathcal{Y}, \ Y_0)  $ for some  nonsingular matrices $\mathcal{X}, \mathcal{Y} \in \mathbb{C}^{mn \times mn}$ and  $X_0, Y_0 \in \mathbb{C}^{r \times r}$, where $\mathbb{F}(\lam)$ is a Fiedler pencil of $G(\lam)$. Hence by Proposition~\ref{prop:gfprstlin}, $\mathbb{T}(\lambda)$ is a Rosenbrock strong linearization of $G(\lam)$.
\end{proof}

\section{Structure-preserving strong  linearizations} \label{structured_lin_GFPR_Glam} This section is devoted to~the construction of structure-preserving strong linearizations of structured rational matrices. We consider only  symmetric, skew-symmetric, Hamiltonian and skew-Hamiltonian rational matrices and construct their structure-preserving strong linearizations. The construction of structure-preserving strong linearizations of  Hermitian, skew-Hermitian, para-Hermitian and para-skew-Hermitian rational matrices is similar. We show that the family of GFPRs of $G(\lam)$ is a rich source of structure-preserving strong linearizations of $G(\lam).$  Recall that $G(\lam) = P(\lam) + G_{sp}(\lam)$, where  $P(\lam) := \sum^m_{j=0} A_j \lam^j$ with $A_m \neq 0$  and $G_{sp}(\lam)$ is  strictly proper, that is, $G_{sp}(\lam) \rightarrow 0$ as $\lam \rightarrow \infty.$

\subsection{Symmetric GFPRs} Suppose that  $G(\lam)$ is  symmetric, that is,  $G( \lam)^T = G(\lam)$. Since $G(\lam) = P(\lam) +G_{sp}(\lam),$ it follows that both $P(\lam)$ and  $G_{sp}(\lam)$ are symmetric. As $G_{sp}(\lam)$ is strictly proper and symmetric, there exists a minimal symmetric realization of $G(\lam)$ given by $G_{sp}(\lam) =  B^T(\lam I_r -A)^{-1}B$, where  $A$ is a symmetric matrix \cite{fuhr, hrw, hilscot}.  Hence $G(\lam) = P(\lam) + B^T(\lam I_r -A)^{-1}B$ is a minimal symmetric realization of $G(\lam).$ The system matrix
$
\mathcal{S} (\lam) : =
\left[ \begin{array}{@{}c|c@{}} P(\lam) &  B^T\\ \hline  B &   A -  \lam I_r \end{array} \right]
$ is then symmetric and irreducible. Also, there exists a minimal symmetric realization of $G(\lam)$ of the form $G(\lam) = P(\lam) + B^T(\lam E -A)^{-1}B$, where $A$ and $E$ are symmetric matrices with $E$ being nonsingular \cite{fmq}. The system matrix
 \begin{equation} \label{symm:Slam}
 \mathcal{S} (\lam) : =
 \left[ \begin{array}{@{}c|c@{}} P(\lam) &  B^T\\ \hline  B &   A -  \lam E \end{array} \right]
 \end{equation}
is obviously symmetric and irreducible.

 A block matrix $\mathcal{H}$ is said to be {\em block-symmetric} provided that $\mathcal{H}^{\mathcal{B}} = \mathcal{H},$ see~\cite{tdm}. The block-transpose of a system matrix is defined as follows.

\begin{definition} \cite{rafinami3, behera} Let $\mathcal{A}:=\left[\begin{array}{@{}c|c@{}}
          A & u \otimes X \\
                              \hline \\[-1em]
          v^{T} \otimes Y & Z\end{array} \right] \in \mathbb{C}^{(mn+r) \times (mn+r) },$ where $A=[A_{ij}]$ is an $m \times m$ block matrix with $A_{ij}\in \mathbb{C}^{n \times n}, u, v \in \mathbb{C}^{m}, X \in \mathbb{C}^{n \times r}, Y \in \mathbb{C}^{r \times n}$ and $Z \in \mathbb{C}^{r \times r}.$ Define the block transpose of $\mathcal{A}$ by $\mathcal{A}^{\mathcal{B}}:= \left[\begin{array}{@{}c|c@{}}
          A^{\mathcal{B}} & v \otimes X \\
                              \hline \\[-1em]
          u^{T} \otimes Y & Z\end{array} \right].$
\end{definition}

Observe that $\mathcal{A}$ is block-symmetric if and only if $A$ is block-symmetric and $u = v$.

\begin{definition} \cite{BDFR15}  \label{admi_tupDef} (a) Let $h \geq 0$ be an integer. We say that $\textbf{w}$ is an admissible tuple of $\{0:h\}$ if $\textbf{w}$ is a permutation of $\{0:h \}$ and
	\begin{equation} \label{admi_tup}
	 csf(\textbf{w}) = \big( h-1:h, h-3:h-2, \ldots, p+1: p+2, 0: p \big)
	\end{equation}
for some $0 \leq p \leq h$. We call $p$ the index of $\textbf{w}$ and denote it by $Ind(\textbf{w})$.
	
	(b) Let $h \geq 0$ be an integer and let $\textbf{w}$ be an admissible tuple of $\{0:h\}$ with index $p$. Then the symmetric complement of $\textbf{w}$, denoted by $c_{\textbf{w}}$, is defined by
	$$ \textbf{c}_w :=   \left \{ \begin{array}{ll}
	\big( h-1, h-3, \ldots, p+3, p+1, (0:p)_{rev_c} \big) & \text{if }  p \geq 1,\\
	( h-1, h-3, \ldots, 1 )   & \text{if }  p=0 \text{ and }  h  > 0,\\
	~ \emptyset &  \text{if } h=0,
	\end{array}	\right.
	$$	
	where $(0:p)_{rev_c} := (0:p-1,0:p-2, \ldots, 0:1,0). $
	\end{definition}

For simplicity, we always consider an admissible tuple of the form (\ref{admi_tup}). Clearly,~for an integer $h \geq 0$, there exists a unique admissible tuple of $\{0:h\}$ with index $0$ or~$1$ (see~\cite{BDFR15}).

\begin{definition} An admissible tuple $\textbf{w}$ of $\{0:h\}$, $h \geq 0$, is said to be the simple admissible tuple if $ Ind(\textbf{w}) =0$ or $ Ind(\textbf{w}) =1$.
\end{definition}

Note that  for the simple admissible tuple $\textbf{w}$ of $ \{0:h\}$, we have $ Ind(\textbf{w}) =0$ (resp., $ Ind(\textbf{w}) =1$) if $h$ is even (resp., odd).

\begin{remark} \label{oddeven_deg} Let $\textbf{v}$ be an admissible tuple of $\{0:k\}$, $k \geq 0$, and let $ \textbf{c}_{v}$ be the symmetric complement of $\textbf{v}$. Then it follows from Definition~\ref{admi_tupDef} that  $0 \in  \textbf{c}_{v}$ if and only if $Ind(\textbf{v}) \geq 1$. In particular, for the simple admissible tuple $\textbf{w}$ of $ \{0:h\}$, we have $0 \in  \textbf{c}_{w}$ (resp., $0 \notin  \textbf{c}_{w}$) if $h$ is odd (resp., even), where $ \textbf{c}_{w}$ is the symmetric complement of $\textbf{w}$.
\end{remark}

\begin{definition} \cite{BDFR15} Given $ h \geq 0,$ we say that an index tuple $\textbf{t}$ is in canonical form for $h$ if $\textbf{t}$ is of the form
	$$ \Big( a_1:h-2, a_2: h-4,  \ldots , a_{    \lfloor  \frac{h}{2} \rfloor  } : h-2 \lfloor \tfrac{h}{2} \rfloor   \Big)$$ with $a_i \geq 0$, $i=1: \lfloor \tfrac{h}{2} \rfloor$, where  $ \lfloor  \cdot  \rfloor $ stands for the greatest integer function.
\end{definition}

Note that an index tuple in canonical form for $h$ is necessarily empty for $ h = 0,1$.

The following result characterizes all symmetric GFPRs of a matrix polynomial.

\begin{theorem} [\cite{BDFR15}, Theorem~6.11] \label{thmgfprSymmCanonicalUnknowns_RCh4}  Let $ 0 \leq h < m$. Let $\textbf{w}_h$ and $ \textbf{v}_h +m$ be the simple admissible tuples of $\{0:h\}$ and $ \{0:m-h-1\}$, respectively.  Let $ \textbf{t}_{w_h}$ and $  \textbf{t}_{v_h} +m$ be index tuples in canonical form for $ h$ and $ m-h-1$, respectively. Let $ \mathcal{X}$ and $ \mathcal{Y}$ be nonsingular matrix assignments for $ \textbf{t}_{w_h}$ and $ \textbf{t}_{v_h}$, respectively. Then
	\begin{equation} \label{gfprSymmCanonicalUnknowns_RCh4}
	L(\lambda):= M_{(\textbf{t}_{v_h},\textbf{t}_{w_h}) } ( \mathcal{Y}, \mathcal{X})( \lambda M^P_{\textbf{v}_h} - M^P_{\textbf{w}_h}) M^P_{(\textbf{c}_{w_h}, \textbf{c}_{v_h})} M_{(rev(\textbf{t}_{w_h}), rev(\textbf{t}_{v_h}))} ( rev(\mathcal{X}), rev( \mathcal{Y})),
	\end{equation}
	is a block symmetric GFPR of $P(\lambda)$, where $ \textbf{c}_{w_h}$ and $ \textbf{c}_{v_h}+m$ are the symmetric complements of $\textbf{w}_h$ and $ \textbf{v}_h +m$, respectively. Moreover, any block symmetric GFPR of $P(\lambda)$ is of the form (\ref{gfprSymmCanonicalUnknowns_RCh4}). Furthermore, if all the matrices in the matrix assignments $\mathcal{X}$ and $\mathcal{Y}$ are symmetric, then $L(\lam)$ is symmetric when $P(\lam)$ is symmetric.
\end{theorem}

The pencil in (\ref{gfprSymmCanonicalUnknowns_RCh4}) is denoted by $L_P (h, \textbf{t}_{w_h}, \textbf{t}_{v_h}, \mathcal{X},\mathcal{Y})$ and   is uniquely determined by $h, \textbf{t}_{w_h}, \textbf{t}_{v_h},$ $ \mathcal{X}$ and $\mathcal{Y}$, see~\cite{BDFR15}.

\begin{definition} Let $h, \textbf{w}_{h}, \textbf{c}_{w_h}, \textbf{t}_{w_h}, \textbf{v}_{h}, \textbf{c}_{v_h}, \textbf{t}_{v_h},$ $ \mathcal{X}$ and $\mathcal{Y}$ be as given in Theorem~\ref{thmgfprSymmCanonicalUnknowns_RCh4}. Then we define
\begin{equation} \label{L_lam}
\mathbb{L}(\lambda) :=
\mathbb{M}_{(\textbf{t}_{v_h},\textbf{t}_{w_h}) } ( \mathcal{Y}, \mathcal{X})( \lambda \mathbb{M}^\mathcal{S}_{\textbf{v}_h} - \mathbb{M}^\mathcal{S}_{\textbf{w}_h}) \mathbb{M}^P_{(\textbf{c}_{w_h}, \textbf{c}_{v_h})} \mathbb{M}_{(rev(\textbf{t}_{w_h}), rev(\textbf{t}_{v_h}))} ( rev(\mathcal{X}), rev( \mathcal{Y})).
\end{equation}
The pencil $\mathbb{L}(\lam)$ in (\ref{L_lam}) is uniquely determined by  $h, \textbf{t}_{w_h}, \textbf{t}_{v_h},$ $ \mathcal{X}$ and $\mathcal{Y}$. We denote $\mathbb{L}(\lam)$ by $\mathbb{L}_{\mathcal{S}} (h, \textbf{t}_{w_h}, \textbf{t}_{v_h}, \mathcal{X},\mathcal{Y})$.
\end{definition}

The following result characterizes all block-symmetric GFPRs of  $G(\lam)$.

\begin{theorem}  \label{lem:gfpr:blksymm} Let $ \mathcal{S} (\lam)$ be given in (\ref{slamsystemmatrix_RCh4}). Let $ 0 \leq h \leq m-1$ be even. Consider the GFPR $\mathbb{L}(\lam): = \mathbb{L}_\mathcal{S}  (h, \textbf{t}_{w_h}, \textbf{t}_{v_h}, \mathcal{X},\mathcal{Y})$ associated with ${\mathcal{S}}(\lambda )$. Then \begin{equation} \label{blsymmgfpr}
\mathbb{L}(\lambda)  =
\left[
\begin{array}{c|c}
L_P (h, \textbf{t}_{w_h}, \textbf{t}_{v_h}, \mathcal{X},\mathcal{Y}) &  e_{m-i_0(\textbf{t}_{w_h},\textbf{w}_h)} \otimes C \\[.1em] \hline \\[-1em]
e^T_{m-c_0(\textbf{w}_h, \textbf{c}_{w_h}, rev (\textbf{t}_{w_h}) ) } \otimes B  & A-\lam E \\
\end{array}
\right]. \end{equation}
Further, we have the following:
	
(a) $\mathbb{L}(\lambda)$ is a block symmetric GFPR of $\mathcal{S}(\lambda)$. Further, any block symmetric GFPR of $\mathcal{S}(\lambda)$ must be of the form $\mathbb{L}_\mathcal{S}  (h, \textbf{t}_{w_h}, \textbf{t}_{v_h}, \mathcal{X},\mathcal{Y})$ for some even $ 0 \leq h \leq m-1$.
	
(b) If $m$ is odd then $ \mathbb{L}(\lambda)$ is a  Rosenbrock strong linearization of $\mathcal{S}(\lambda).$
If $m$ is even then $ \mathbb{L}(\lambda)$ is a  Rosenbrock strong linearization of $\mathcal{S}(\lambda)$ when the leading coefficient of $P(\lam)$ is nonsingular.
 \end{theorem}

\begin{proof} By substituting $ \sigma=\textbf{w}_h$, $ \sigma_1 = \textbf{t}_{w_h}$, $ \sigma_2= (\textbf{c}_{w_h}, rev(\textbf{t}_{w_h}))$, $ \tau = \textbf{v}_h,   \tau_1 = \textbf{t}_{v_h}$ and $ \tau_2 = (\textbf{c}_{v_h}, rev(\textbf{t}_{v_h}))$ in Theorem~\ref{gfprptoge_RCh4}, we obtain (\ref{blsymmgfpr}).

 By Theorem~\ref{thmgfprSymmCanonicalUnknowns_RCh4},  $L_P (h, \textbf{t}_{w_h}, \textbf{t}_{v_h}, \mathcal{X},\mathcal{Y})$ is a block symmetric pencil. Hence it follows that $\mathbb{L}(\lambda)$ is block symmetric if and only if $c_0(\textbf{w}_h,\textbf{c}_{w_h}, rev(\textbf{t}_{w_h}) ) = i_0( \textbf{t}_{w_h}, \textbf{w}_h )$. Next, we show that $c_0(\textbf{w}_h,\textbf{c}_{w_h}, rev(\textbf{t}_{w_h}) ) = i_0( \textbf{t}_{w_h}, \textbf{w}_h )$.

Case-I: Suppose that $h=0$. Then $ \textbf{w}_h= (0)$ and $\textbf{c}_{w_h} = \emptyset = \textbf{t}_{w_h}$. Hence $i_0( \textbf{t}_{w_h}, \textbf{w}_h ) =0=c_0(\textbf{w}_h,\textbf{c}_{w_h}, rev(\textbf{t}_{w_h}) ) $.

Case-II: Suppose that $ h>0$. Since $h$ is even and $\textbf{w}_h$ is the simple admissible tuple of $\{0:h\}$, we have $\textbf{w}_h =(h-1:h, h-3:h-2, \ldots,1:2,0)$ and $\textbf{c}_{w_h} = ( h-1, h-3, \ldots, 3,1)$. Thus $c_0(\textbf{w}_h,\textbf{c}_{w_h}, rev(\textbf{t}_{w_h}) )=2+c_2(rev(\textbf{t}_{w_h})) $ and $i_0(\textbf{t}_{w_h}, \textbf{w}_h)= 2+ i_2(\textbf{t}_{w_h})$. (Recall that for any index tuple $\beta$ and for any index $t$,  if $t \notin \beta$ then $c_t(\beta) =-1 = i_t(\beta)$).  Hence $\mathbb{L} (\lam)$ is block-symmetric since $i_t(\beta)=c_t(rev(\beta))$ for any index tuple $ \beta$ and any index $t$. This proves the first part of (a).

 Next we prove that, if $h$ is odd, then $c_0(\textbf{w}_h,\textbf{c}_{w_h}, rev(\textbf{t}_{w_h}) ) \neq i_0( \textbf{t}_{w_h}, \textbf{w}_h )$. Then it follows from (\ref{blsymmgfpr}) that $ \mathbb{L} (\lam)$ is not a block symmetric GFPR of $\mathcal{S}(\lam)$. This will prove the second part of (a).

Let $ h \geq 0$ be odd. If $h=1$ then $ \textbf{w}_h= (0,1)$, $\textbf{c}_{w_h} =(0)$ and $ \textbf{t}_{w_h} =\emptyset$. Thus $c_0(\textbf{w}_h,\textbf{c}_{w_h}, rev(\textbf{t}_{w_h}) ) =1 $ and $i_0( \textbf{t}_{w_h}, \textbf{w}_h )=0$. Hence $ \mathbb{L} (\lam)$ is not block symmetric.

Next, suppose that $h>1$. Then $\textbf{w}_h =(h-1:h, h-3:h-2, \ldots,2:3,0:1)$ and $\textbf{c}_{w_h} = ( h-1, h-3, \ldots, 2,0)$. Thus $c_0(\textbf{w}_h,\textbf{c}_{w_h}, rev(\textbf{t}_{w_h}) )= 3+ c_3(rev(\textbf{t}_{w_h})) = 3+ i_3( \textbf{t}_{w_h})$ and $i_0(\textbf{t}_{w_h}, \textbf{w}_h)= 1+ i_1(\textbf{t}_{w_h})$. We show that $3+ i_3( \textbf{t}_{w_h}) \neq 1+ i_1(\textbf{t}_{w_h}).$  Let $i_1(\textbf{t}_{w_h}) =p$.  If $ p=-1$ or $p=0$ then  $1+ i_1(\textbf{t}_{w_h}) < 2 \leq 3+ i_3( \textbf{t}_{w_h})$ and hence the desired result follows. Suppose that $ p \geq 1.$ Note that $\textbf{t}_{w_h}$ is in canonical form for $h$ ($h>1$ is odd), i.e.,
\begin{equation} \label{twh_canonicalform}
\textbf{t}_{w_h}= \big ( a_1:h-2, a_2: h-4,  \ldots ,a_{\frac{h-1}{2} -1} :3, a_{\frac{h-1}{2}} :1   \big ).
\end{equation}
We call $(a_j:h-2j), \, j=1,2,\ldots,\frac{h-1}{2}$, as the strings of $\textbf{t}_{w_h}$ and $h-2j$ as the right end point of the string $(a_j:h-2j).$  Since $i_1(\textbf{t}_{w_h}) =p$, $ ( p+1, p, \ldots,3, 2,1)$ is a subtuple of $\textbf{t}_{w_h}$ and $ (p+2, p+1, p, \ldots, 2,1)$ is not a subtuple of $\textbf{t}_{w_h}$. It is clear from (\ref{twh_canonicalform}) that each index of the subtuple $ ( p+1, p, \ldots, 2,1)$ of $\textbf{t}_{w_h}$ belongs to distinct string of $\textbf{t}_{w_h}$. By collecting all those strings we have a subtuple
$$
\Big ( (p+1:b_{p+1}) , (p:b_p), \ldots,(3:b_3) ,(2 :b_2), (1:b_1) \Big )
$$ of $\textbf{t}_{w_h}$, where $b_j$'s are the right end points of the collected strings. Hence $b_j \in \{ 1,3,5, \ldots, h-4, h-2\}$ for $j=1: p+1$ is such that $ b_{p+1} > b_p > \cdots > b_3 >b_2> b_1.$ This implies that $b_2 \geq 3$ and hence $3 \in (2:b_2)$, $b_3 \geq 5$ and hence $4 \in (3:b_3)$, and so on $p+1 \in (p:b_p)$ and $p+2 \in (p+1: b_{p+2})$. Consequently, $( p+2, p+1, \ldots, 4,3)$ is a subtuple of $\textbf{t}_{w_h}$ and $i_3(\textbf{t}_{w_h})\geq  p-1$. So $ 3+ i_3( \textbf{t}_{w_h}) \geq p+2 > p+1 = 1+ i_1(\textbf{t}_{w_h})$. Hence $c_0(\textbf{w}_h,\textbf{c}_{w_h}, rev(\textbf{t}_{w_h}) ) \neq i_0(\textbf{t}_{w_h}, \textbf{w}_h) $ and  $ \mathbb{L} (\lam)$ is not a block symmetric GFPR of $\mathcal{S}(\lam)$. This completes the proof of the second part of (a).

(b) Since $h$ is even, by Remark~\ref{oddeven_deg} we have  $ 0 \notin \textbf{c}_{w_h}$. This implies that the matrix assignment for $\textbf{c}_{w_h}$ is nonsingular. Further, it is given that $\mathcal{X}$ and $\mathcal{Y}$ are nonsingular matrix assignments for $\textbf{t}_{w_h}$ and $\textbf{t}_{v_h}$, respectively. Consequently, by taking $ \sig : =\textbf{w}_h, \tau : = \textbf{v}_h, \sig_1 := \textbf{t}_{w_h}, \sig_2 := (\textbf{c}_{w_h}, rev (\textbf{t}_{w_h}))$, $ \tau_1 := \textbf{t}_{v_h}$ and $ \tau_2 := (\textbf{c}_{v_h}, rev(\textbf{t}_{v_h}))$, it follows from Theorem~\ref{thm:gfpr:RSL} that  $\mathbb{L}(\lam)$ is a Rosenbrock strong linearization of $\mathcal{S}(\lambda)$  if the matrix assignment for $\textbf{c}_{v_h}$ is nonsingular. Suppose that $m$ is odd. Then $m-h-1$ is even (since $h$ is even) and by Remark~\ref{oddeven_deg}, it follows that $  0 \notin \textbf{c}_{v_h}+m \implies -m \notin \textbf{c}_{v_h}$. Hence the matrix assignment for $\textbf{c}_{v_h}$ is nonsingular. On the other hand, if the leading coefficient of $P(\lam)$ is nonsingular then the matrix assignment for $\textbf{c}_{v_h}$ is nonsingular irrespective of $m$ being even or odd. Hence  $ \mathbb{L}(\lambda)$ is a block symmetric Rosenbrock strong linearization of $\mathcal{S}(\lambda)$.
\end{proof}

\begin{corollary} \label{symm:GFPR} Let $G(\lam)$ be symmetric and  ${\mathcal{S}}(\lambda ) $ be  given  in (\ref{symm:Slam}). Let $ 0 \leq h \leq m-1$ be even. Consider the GFPR $$\mathbb{L}(\lam): = \mathbb{L}_\mathcal{S}  (h, \textbf{t}_{w_h}, \textbf{t}_{v_h}, \mathcal{X},\mathcal{Y}) = \left[
	\begin{array}{c|c}
	L_P(h, \textbf{t}_{w_h}, \textbf{t}_{v_h}, \mathcal{X},\mathcal{Y}) &  e_{m- \alpha} \otimes B^T \\[.1em] \hline \\[-1em]
	e^T_{m-\alpha } \otimes B  & A-\lam E \\
	\end{array}
	\right]$$ associated with ${\mathcal{S}}(\lambda )$, where $\alpha : = i_0(\textbf{t}_{w_h},\textbf{w}_h)$, $\mathcal{X}$ and $\mathcal{Y}$ are nonsingular matrix assignments and all the matrices in $\mathcal{X}$ and $\mathcal{Y}$ are symmetric. If $m$ is odd then $ \mathbb{L}(\lambda)$ is a symmetric Rosenbrock strong linearization of $G(\lambda).$
	If $m$ is even then $ \mathbb{L}(\lambda)$ is a symmetric Rosenbrock strong linearization of $G(\lambda)$ when the leading coefficient of $P(\lam)$ is nonsingular.  Also the transfer function $\mathbb{G}(\lam) : = L(\lam) +  (e_{m-\alpha} \otimes B^T ) ( \lam E - A)^{-1} (e^T_{m-\alpha} \otimes B) $ of $\mathbb{L}(\lam) $ is  symmetric, where $L(\lam) :=L_P(h, \textbf{t}_{w_h}, \textbf{t}_{v_h}, \mathcal{X},\mathcal{Y}).$
\end{corollary}

\begin{proof} By considering $C = B^T$  it follows from the proof of Theorem~\ref{lem:gfpr:blksymm} that
\begin{align}
\mathbb{L}(\lambda) = \left[
\begin{array}{c|c}
L_P(h, \textbf{t}_{w_h}, \textbf{t}_{v_h}, \mathcal{X},\mathcal{Y}) &  e_{m- \alpha} \otimes B^T \\[.1em] \hline \\[-1em]
e^T_{m-\alpha } \otimes B  & A-\lam E \\
\end{array}
\right] \label{BsymmGFPRn}
\end{align}	
is a block symmetric Rosenbrock strong linearization of  $\mathcal{S}(\lambda)$, where $\alpha : = i_0(\textbf{t}_{w_h},\textbf{w}_h)$. Since $P(\lam)$ is symmetric and all the  matrices in the  matrix assignments $\mathcal{X}$ and $\mathcal{Y}$ are symmetric, by Theorem~\ref{thmgfprSymmCanonicalUnknowns_RCh4}, we have $L_P(h, \textbf{t}_{w_h}, \textbf{t}_{v_h}, \mathcal{X},\mathcal{Y})$ is symmetric. Further, since $A$ and $E$ are symmetric, it follows from (\ref{BsymmGFPRn}) that $\mathbb{L}(\lam)$  and $\mathbb{G} (\lam)$ are symmetric.
\end{proof}

\begin{example} Let $G(\lam) = \sum_{i=0}^5 \lambda^i A_i + B^T (\lam E -A)^{-1}B$ be symmetric. Consider  $ h=2$, $\textbf{t}_{w_h}=(0)$ and $\textbf{t}_{v_h}= (-5)$. Let $X$ and $Y$ be any arbitrary nonsingular symmetric matrices. Then the GFPR
	$$\mathbb{L}_{\mathcal{S}} (h, \textbf{t}_{w_h}, \textbf{t}_{v_h}, \mathcal{X},\mathcal{Y}) =
	\left[\begin{array}{@{}ccccc|c@{}}
	0 & -Y & \lam Y & 0 & 0 & 0\\
	-Y & \lam A_{5}-A_{4} & \lam A_{4}  & 0 & 0 & 0\\
	\lam Y  & \lam  A_{4}  & \lam A_{3}  + A_{2} & A_{1} & -X & 0\\
	0 & 0 & A_{1} & - \lam A_{1} + A_{0} & \lam X & B^T\\
	0 & 0 & -X & \lam  X & 0 & 0 \\ \hline
	0 & 0& 0& B& 0& A- \lam E \end{array}\right] $$
	is a symmetric Rosenbrock strong linearization of $\mathcal{S}(\lam)$. Note that $\mathbb{L}_{\mathcal{S}} (h, \textbf{t}_{w_h}, \textbf{t}_{v_h}, \mathcal{X},\mathcal{Y})$ is a block penta-diagonal pencil.

Next, let $G(\lambda) := \sum_{i=0}^6 \lambda^i A_i+  B^T(\lam E -A)^{-1}B$ be symmetric. Consider  $ h=0$, $\textbf{t}_{w_h}= \emptyset$ and $\textbf{t}_{v_h} = (-6:-3, -6:-5).$ Then
the GFPR $ \mathbb{L}_{\mathcal{S}} (h, \textbf{t}_{w_h}, \textbf{t}_{v_h}, \mathcal{X},\mathcal{Y})$
{\small $$= \left[ \begin{array}{@{}cccccc|@{}c@{}}
	0 & 0&0&0&-A_6 & \lam A_6 & 0\\
	0 & 0&0& -A_6 & \lam A_6 - A_5 & \lam A_5& 0\\
	0 & 0 & -A_6 & \lam A_6 -A_5 & \lam A_5 - A_4 & \lam A_4& 0\\
	0 & -A_6 & \lam A_6 -A_5 & \lam A_5 - A_4 & \lam A_4-A_3 & \lam A_3& 0\\
	-A_6 & \lam A_6 -A_5 & \lam A_5 - A_4 & \lam A_4-A_3 & \lam A_3 - A_2 & \lam A_2& 0\\
	\lam A_6 & \lam A_5 & \lam A_4 & \lam A_3 & \lam A_2 & \lam A_1 + A_0& B^T\\ \hline \\[-1.1em]
	0  & 0& 0& 0& 0& B & A- \lam E
	\end{array}\right]$$}is a symmetric Rosenbrock strong linearization of  $G(\lam)$ when $A_6$ is nonsingular, where $\mathcal{X}$ and $\mathcal{Y}$ are the trivial matrix assignments.
\end{example}

We now show that the transfer function of a real symmetric strong linearization preserves the Cauchy-Maslov index of a real symmetric rational matrix.

\begin{definition} \cite{cauchyind} The Cauchy-Maslov index of a real symmetric rational matrix $G(\lam)$  is defined by $\mathbf{Ind}_{\mathrm{CM}}(G):=$ (\# eigenvalues  of  $G(\lam)$ which jump from $-\infty$ to $+\infty )$  $\; -\; $   (\# eigenvalues of  $G(\lam)$ which jump from $+\infty$ to $-\infty$)  as the real parameter $ \lam$ traverses  from $-\infty$ to $+\infty$.
\end{definition}

 The Cauchy-Maslov index of a real symmetric rational matrix plays an important role in many applications such as in networks of linear systems, see~\cite{cauchyind, byrn, byrn2, hug} and the references therein. It is therefore desirable to construct real symmetric linearizations of $G(\lam)$ whose transfer functions preserve the Cauchy-Maslov index of $G(\lam).$

\begin{theorem}  Let $G(\lam)$ be real symmetric and  ${\mathcal{S}}(\lambda ) $ be as given  in (\ref{symm:Slam}).  Let $\mathbb{L}(\lam): = \mathbb{L}_\mathcal{S}  (h, \textbf{t}_{w_h}, \textbf{t}_{v_h}, \mathcal{X},\mathcal{Y})$ be a symmetric Rosenbrock strong linearization of $G(\lam)$ as given in Corollary~\ref{symm:GFPR}. Let $\mathbb{G}(\lam)$ be the associated transfer function of $\mathbb{L}(\lam)$. Then  $\mathbb{G}(\lam)$ is real and symmetric and has the same Cauchy-Maslov index as $G(\lam),$ that is, $\mathbf{Ind}_{\mathrm{CM}}(G) = \mathbf{Ind}_{\mathrm{CM}}(\mathbb{G}).$
\end{theorem}

\begin{proof} By Corollary~\ref{symm:GFPR} we have $\mathbb{G} (\lam)  =  L(\lam) +  (e_{m-\alpha} \otimes B^T ) ( \lam E - A)^{-1} (e^T_{m-\alpha} \otimes B) $ is symmetric, where $\alpha : = i_0(\textbf{t}_{w_h},\textbf{w}_h)$ and $L(\lam)$ are as given in  Corollary~\ref{symm:GFPR}.

	Next, we show that $\mathbf{Ind}_{\mathrm{CM}}(G) = \mathbf{Ind}_{\mathrm{CM}}(\mathbb{G}).$ Set $G_{sp}(\lam) : = B^T (\lam E -A )^{-1} B$. Then we have $G(\lam) = P(\lam) + G_{sp}(\lam) $ and
	\begin{align}
	\mathbb{G} (\lam) & =  L(\lam) +  (e_{m-\alpha} \otimes B^T ) ( \lam E - A)^{-1} (e^T_{m-\alpha} \otimes B) \nonumber \\
	& =  L(\lam) +  (e_{m-\alpha} \otimes I_n ) B^T( \lam E - A)^{-1}B (e^T_{m-\alpha} \otimes I_n)  \nonumber \\
	&= L(\lam) + \diag (0, \ldots, 0, \hspace{-1.6em} \underbrace{G_{sp}(\lam)}_{(m-\alpha) \text{-th position}} \hspace{-1.5em},  0, \ldots, 0) . \label{eqn:CM:G}
	\end{align}
Since $L(\lam)$ is a matrix pencil, it follows from (\ref{eqn:CM:G}) that the contribution in $\mathbf{Ind}_{\mathrm{CM}}(\mathbb{G})$ comes only from
$\diag (0, \ldots, 0, G_{sp}(\lam),  0, \ldots, 0).$ Hence we have
\begin{align*}
\mathbf{Ind}_{\mathrm{CM}}(\mathbb{G}) &= \mathbf{Ind}_{\mathrm{CM}} \big (\diag (0, \ldots, 0, G_{sp},  0, \ldots, 0) \big)\\
&= \mathbf{Ind}_{\mathrm{CM}} \big (G_{sp} \big) = \mathbf{Ind}_{\mathrm{CM}}(G),
\end{align*} where the last equality follows from the fact that the Cauchy-Maslov index is invariant under perturbation by a matrix polynomial. This completes the proof.
\end{proof}

\begin{remark} Although the Cauchy-Maslov index is defined for real symmetric rational matrices, it can be extended to Hermitian rational matrices.
\end{remark}

\begin{remark} Let $G(\lam) = P(\lam) + B^T (\lam E -A)^{-1} B$ be symmetric, where $P(\lam) = \sum^m_{j=0}A_j \lam^j$  and $m>1$. Then the construction given in  \cite[Theorem~5.3]{rafiran1} generates only one symmetric linearization of $G(\lam)$ which is explicitly given by
	$$ \mathbb{T}(\lambda):= {\small \lambda \left[\arraycolsep=1 pt \begin{array}{rlccc|c}
		& &  &&A_m &  \\
		& &  & \rotatebox{40}{$\cdots$}  & A_{m-1} &\\
		& &   &\rotatebox{45}{$\cdots$}   & \vdots&\\
		& {}^{\rotatebox{40}{$\cdots$}}  & {}^{\rotatebox{45}{$\cdots$}  }  & & \vdots &\\
		A_m~ & A_{m-1}& ~\cdots & & A_1  &\\
		\hline
		&&&&& - E\\
		\end{array}\right] +
		\left[\arraycolsep=1 pt\def\arraystretch{1.4}  \begin{array}{rllcc|c}
		& & & -A_m & &  \\
		&  & {\rotatebox{32}{$\cdots$}}~&-A_{m-1}& & \\
		&{\rotatebox{32}{$\cdots$}}  & {\rotatebox{32}{$\cdots$}} & \vdots & & \\
		-A_m~ & -A_{m-1}& ~\cdots & -A_2  & &\\
		& & & & A_0& B^T\\
		\hline
		&& &  &B& A\\
		\end{array}\right]}.$$
 Further, the pencil  $\mathbb{T}(\lambda)$ is a  Rosenbrock strong linearization of $G(\lambda)$ if and only if  $A_m$ is nonsingular \cite[Theorem~5.3]{rafiran1}. By contrast, the family of  GFPRs enables us to construct an infinite number of symmetric strong linearizations of   $G(\lam)$. In fact, by considering $h =0$, $ \textbf{t}_{w_h} = \emptyset$ and $ \textbf{t}_{v_h} = -m + (0:m-3, 0:m-5, \ldots )$, we have $ \mathbb{L}_\mathcal{S}  (h, \textbf{t}_{w_h}, \textbf{t}_{v_h}, \mathcal{X}, \mathcal{Y}) = \mathbb{T} (\lam)$, where $\mathcal{X}$ and $\mathcal{Y}$ are the trivial matrix assignments.
\end{remark}

\subsection{Hamiltonian linearizations}  Recall that a rational matrix $G(\lam)$ is said to be Hamiltonian (i.e., $T$-even) if $G(- \lam)^T = G(\lam)$. Since $G(\lam) = P(\lam) + G_{sp}(\lam)$, it follows that if $ G(\lam)$ is $T$-even then both  $ P(\lam)$ and $ G_{sp}(\lam)$ are $T$-even.  We now construct $T$-even Rosenbrock strong linearizations of $G(\lam)$. We proceed as follows.

For the rest of the paper,  we define  $ J:= \left[ \begin{array}{@{}cc@{}} 0 & I_{\ell} \\- I_{\ell} & 0 \end{array} \right]$ when $r = 2 \ell$. Note that $J^T = J^{-1} = -J.$ Further, we define $\mathbb{J}_{k,r} := \diag (I_{k}, J)$ for any integer $ k \geq 1$ when $ r = 2 \ell.$

\begin{definition}\cite{hilscot} A matrix $X \in \mathbb{C}^{r}$ with $ r := 2\ell$ is said to be  Hamiltonian (resp., skew-Hamiltonian)  if $ JX$ is symmetric (resp., $JX$ is skew-symmetric), that is, $(JX)^T = JX$ (resp., $(JX)^T = -JX$).
	\end{definition}

If $X$ is  Hamiltonian then $(JX)^T = JX \Rightarrow (XJ)^T = XJ$. Similarly, if $X$ is  skew-Hamiltonian  then we have $(XJ)^T = -XJ$.


\begin{definition} Let $G(\lam)$ be a Hamiltonian (i.e., $T$-even) rational matrix.
 \begin{itemize}

\item[(a)] A realization of $G(\lam)$ of the form  $G(\lam) =P(\lam) + C (\lam I_r - A)^{-1}B$  is said to be a Hamiltonian realization of $G(\lam)$ if $P(\lam)$ is $T$-even, $A$ is Hamiltonian with $ r = 2\ell$ and  $JB = C^T.$

\item[(b)] A system matrix $\mathcal{S}(\lam)$ of the form  $\mathcal{S}(\lam) : =
		\left[ \begin{array}{@{}c|c@{}}
		P(\lam) &  C \\ \hline  B & A - \lam I_r
		\end{array}
		\right] $  is said to be a Hamiltonian system matrix if  $ r = 2\ell$ and  $\mathbb{J}_{n,r} \mathcal{S} (\lam)$ is $T$-even, that is, if  $\big (\mathbb{J}_{n,r} \, \mathcal{S}  (-\lam) \big )^T = \mathbb{J}_{n,r} \, \mathcal{S}  (\lam)$, where $\mathbb{J}_{n,r}: = \diag (I_n , J)$.

\item[(c)] A realization of $G(\lam)$ of the form  $G(\lam) =P(\lam) + C (\lam E - A)^{-1}B$ with $E$ being nonsingular is said to be a $T$-even realization of $G(\lam)$ if $ C= B^T$ and both $P(\lam)$ and $ \lam E - A $ are  $T$-even.

\end{itemize}
\end{definition}

 Note that the system matrix $\mathcal{S}(\lam)$ associated with a  $T$-even realization of $G(\lam)$ is $T$-even, that is, $ \mathcal{S}(-\lam)^T = \mathcal{S}(\lam).$

 \begin{remark}
Observe that  $G(\lam) =P(\lam) + C (\lam I_r - A)^{-1}B$  is a Hamiltonian realization of $G(\lam)$  if and only if  $\mathcal{S}(\lam) : =
		\left[ \begin{array}{@{}c|c@{}}
		P(\lam) &  C \\ \hline  B & A - \lam I_r
		\end{array}
		\right] $  is a Hamiltonian system matrix of $G(\lam).$ On the other hand, $G(\lam) =P(\lam) + C (\lam E - A)^{-1}B$ is a $T$-even realization of $G(\lam)$ if and only if  $\mathcal{S}(\lam) : =
		\left[ \begin{array}{@{}c|c@{}}
		P(\lam) &  C \\ \hline  B & A - \lam E
		\end{array}
		\right] $ is a $T$-even system matrix of $G(\lam).$

\end{remark}

For convenience, we often refer to $\mathcal{S}(\lam)$ as a $T$-even (resp., Hamiltonian) realization of $G(\lam)$ when $\mathcal{S}(\lam)$ is $T$-even (resp., Hamiltonian).

\begin{proposition} \label{prop:hamil} Suppose that  $G(\lam)$  is  Hamiltonian (i.e., $T$-even). Then we have  the following:
	\begin{itemize}
		\item[(a)] There exists a minimal Hamiltonian realization of $G(\lam)$ of the form  $G(\lam) = P(\lam) + C (\lam I_r - A)^{-1}B$  with $ r = 2 \ell$ and $JB = C^T.$   Thus the associated system matrix
		$\mathcal{S} (\lam) : =
		\left[ \begin{array}{@{}c|c@{}}
		P(\lam) &  B^T J^T\\ \hline  B & A - \lam I_r
		\end{array}
		\right] $  is Hamiltonian.

		\item[(b)] There exists a minimal $T$-even realization of $G(\lam)$ of the form $G(\lam) = P(\lam) + B^T (\lam E- A)^{-1}B$.  Thus the  system matrix
		$\mathcal{S} (\lam)  =
		\left[ \begin{array}{@{}c|c@{}}
		P(\lam) &  B^T\\ \hline  B & A - \lam E
		\end{array}
		\right]$  is $T$-even.
	\end{itemize}
\end{proposition}	

\begin{proof}  Since $G(\lam) = P(\lam)+G_{sp}(\lam)$ is $T$-even, we have $P(\lam)$ and $G_{sp}(\lam)$ are $T$-even. Also since $G_{sp}(\lam)$ is strictly proper and $T$-even, there exists a minimal Hamiltonian realization of $G_{sp}(\lam)$ of the form $G_{sp}(\lam) =  C (\lam I_r - A)^{-1} B$ with $ r = 2\ell$ and $ JB = C^T$; see~\cite{hilscot}. Hence  $G(\lam) = P(\lam) + C (\lam I_r - A)^{-1} B$ is a minimal Hamiltonian realization of $G(\lam).$
Obviously the system matrix $\mathcal{S} (\lam)$ is Hamiltonian, that is, $  (\mathbb{J}_{n,r} \, \mathcal{S} (- \lam))^T = \mathbb{J}_{n,r} \,\mathcal{S} (\lam),$ where $\mathbb{J}_{n,r} := \diag (I_n,  J)$. This proves (a).

The results in (b) follow from (a). Indeed, by part (a) we have $G(\lam) = P(\lam) + B^T J^T (\lam I_r - A)^{-1} B = P(\lam) + B^T( \lam J - AJ)^{-1} B.$  Since $A$ is Hamiltonian, it follows that $\lam J - A J$ is $T$-even. Hence setting  $E := J$ and redefining $A := AJ$, it follows that $G(\lam): = P(\lam) + B^T (\lam E- A)^{-1}B$ is a minimal $T$-even realization of $G(\lam).$
Evidently, the system matrix $\mathcal{S} (\lam)$ is $T$-even, that is, $\mathcal{S} (-\lam)^T = \mathcal{S} (\lam)$. This proves (b).
\end{proof}

We construct $T$-even (resp., Hamiltonian) linearizations of $G(\lam)$ corresponding to a $T$-even (resp., Hamiltonian) realization of $G(\lam).$ We proceed as follows.

\begin{definition} \cite{bfs2} A matrix $Q \in \mathbb{C}^{mn \times mn}$ is said to be a quasi-identity matrix if $Q = \epsilon_1 I_n \oplus \cdots \oplus \epsilon_m I_n $, where  $ \epsilon_i  \in \{\pm 1 \}$ for $i=1:m.$ We refer to $\epsilon_j$, $j=1:m$, as the $j$-th parameter of $Q$.
\end{definition}

We need the following result which is a particular case of \cite[Theorem 4.15]{bfs2}.

\begin{theorem}[\cite{bfs2}, Theorem 4.15] \label{eventhm}   Let $0 \leq h \leq m-1$ be even. Let $\textbf{w}$  be the simple admissible tuple of $\{ 0 : h\}$ and $\textbf{c}_{w}$ be the symmetric complement of $\textbf{w}$. Let $\textbf{z} +m$ be any admissible tuple of $\{0: m-h-1\}$ and $\textbf{c}_{z} +m$ be the symmetric complement of $\textbf{z} +m$. Let
	$ L(\lambda) :=  \big( \lambda M^P_{\textbf{z}} - M^P_{\textbf{w}} \big)  M^P_{\textbf{c}_w} M^P_{\textbf{c}_z} .$ Then, up to multiplication by $-1$, there exists a unique quasi-identity matrix $Q$ such that $QL(\lambda)$ is $T$-even (resp., $T$- odd) when $P(\lambda)$ is $T$-even (resp., $T$-odd).
\end{theorem}

We refer to \cite[Algorithm 4.14]{bfs2} for more on the construction of the quasi-identity matrix $Q.$  The next result provides $T$-even linearizations of $G(\lam)$.

\begin{theorem} \label{thm:Hamil:G} Let $G(\lam)$ be $T$-even and
	$\mathcal{S} (\lam)$ be a $T$-even realization of $G(\lam)$ as given in Proposition~\ref{prop:hamil}(b). Let $h, \textbf{w}, \textbf{c}_w, \textbf{z} $ and $ \textbf{c}_z$ be as in Theorem~\ref{eventhm}.  Consider the GFPR  $\mathbb{L}(\lambda) := \big( \lambda \mathbb{M}^\mathcal{S}_{\textbf{z}} - \mathbb{M}^\mathcal{S}_{\textbf{w}} \big)  \mathbb{M}^P_{\textbf{c}_w}  \mathbb{M}^P_{\textbf{c}_z} $ associated with $\mathcal{S}(\lam)$. Then  there exists a unique quasi-identity matrix $\mathbb{Q} := \diag( \mathbf{s}\, Q, I_r) $  such that  	$$ \mathbb{Q}\mathbb{L} (\lambda)   =
	\left[
	\begin{array}{@{}c|c@{}}
	\mathbf{s} \,Q L(\lam) &  e_{m- i_0(\textbf{w})} \otimes B^T \\[.1em] \hline \\[-1em]
	e^T_{m- i_0(\textbf{w})} \otimes B  & A - \lam E \\
	\end{array}
	\right] $$  is  $T$-even, where $Q$ and $L(\lam)$ are as in Theorem~\ref{eventhm} and $\mathbf{s}$ is the $(m - i_0(\textbf{w}))$-th parameter of $Q$.

Assume that  $Ind(\textbf{z} +m) =0$ when the leading coefficient of $P(\lam)$ is singular.  Then $\mathbb{Q} \mathbb{L}(\lam)$ is a  Rosenbrock strong linearization of  $G(\lam)$.
The transfer function $\mathbb{G}(\lam) : = \mathbf{s} \, Q L(\lam) +  (e_{m-i_0(\textbf{w})} \otimes B^T ) ( \lam E - A)^{-1} (e^T_{m-i_0(\textbf{w})} \otimes B) $ of  $\mathbb{Q} \mathbb{L}(\lam)$ is $T$-even.	
\end{theorem}

\begin{proof} By Theorem~\ref{gfprptoge_RCh4}, we have
	$$ \mathbb{L}(\lambda)  =
	\left[
	\begin{array}{@{}c|c@{}}
	L(\lam) &  e_{m-i_0(\textbf{w})} \otimes B^T \\[.1em] \hline \\[-1em]
	e^T_{m-c_0( \textbf{w}, \textbf{c}_w)} \otimes B  &  A - \lam E   \\
	\end{array}
	\right],$$
	where $L(\lam)$ is as given in Theorem~\ref{eventhm}. Since $h$ is even and $ \textbf{w}$ is the simple admissible tuple of $\{0:h\}$, we have $\textbf{w} = ( h-1:h,\ldots,3:4,1:2,0)  $ and $ \textbf{c}_w = (h-1, h-3, \ldots, 1)$. 	 This implies that $i_0(\textbf{w}) =  c_0( \textbf{w}, \textbf{c}_w) = 	0 $  if $h=0$, and $i_0(\textbf{w}) =  c_0( \textbf{w}, \textbf{c}_w) = 1$  if $h > 0$.
	By  Theorem~\ref{eventhm},  $\mathbf{s}\, Q L(\lam)$ is $T$-even. Set $\alpha: = i_0(\textbf{w})$. Then  $ Q (e_{m-\alpha} \otimes I_n ) = \mathbf{s} (e_{m-\alpha} \otimes I_n )$. Note that $\mathbf{s}\mathbf{s} =1 .$ Consequently, we have
	\begin{equation} \label{eqn:even:G}
	\mathbb{Q} \mathbb{L}(\lam)  =
	\left[
	\begin{array}{@{}c|c@{}}
	\mathbf{s} \, Q L(\lam) &  \mathbf{s} \, Q (e_{m-\alpha} \otimes B^T )\\ \hline \\[-1em]
	e^T_{m-\alpha} \otimes B  &  A - \lam E  \\
	\end{array}
	\right] =
	\left[
	\begin{array}{@{}c|c@{}}
	\mathbf{s} \, Q L(\lam) & e_{m-\alpha} \otimes B^T \\ \hline \\[-1em]
	e^T_{m-\alpha} \otimes B  &  A - \lam E  \\
	\end{array}
	\right].
	\end{equation}
	Since $	\mathbf{s}\, Q L(\lam)$  and $A - \lam E$ are $T$-even, it follows from (\ref{eqn:even:G}) that  $ \mathbb{Q} \mathbb{L}(\lam)$ is $T$-even.

Since $h$ is even, by Remark~\ref{oddeven_deg} we have  $ 0 \notin \textbf{c}_{w}$. This implies that the matrix assignment for $\textbf{c}_{w}$ is nonsingular. Hence by taking $ \sig : =\textbf{w}, \tau : = \textbf{z}, \sig_1 := \emptyset, \sig_2 := \textbf{c}_{w}$, $ \tau_1 := \emptyset$ and $ \tau_2 := \textbf{c}_{z}$, it follows from Theorem~\ref{thm:gfpr:RSL} that $ \mathbb{L}(\lam)$ is a Rosenbrock strong linearization of $G(\lam)$ if the matrix assignment for $\textbf{c}_{z}$ is nonsingular. If the leading coefficient of $P(\lam)$ is nonsingular then  the matrix assignment for $\textbf{c}_{z}$ is nonsingular. On the other hand, if the leading coefficient of $P(\lam)$ is singular and $Ind(\textbf{z} +m) =0$, then by Remark~\ref{oddeven_deg}, we have $0 \notin \textbf{c}_z +m \Rightarrow -m \notin \textbf{c}_z$. Hence the matrix assignment for $\textbf{c}_{z}$ is nonsingular. Thus, $\mathbb{Q} \mathbb{L}(\lam)$ is a $T$-even Rosenbrock strong linearization of  $G(\lam)$. Obviously the transfer function $\mathbb{G}(\lam)$ is $T$-even.
\end{proof}

\begin{remark}  Note that if $m$ is even then $ Ind(\mathbf{z}+m) >0$ because $h$ is always even. This implies that $ -m \in \mathbf{c}_z.$ Hence if the leading coefficient of $P(\lam)$ is singular then $\mathbb{QL}(\lam)$ in Theorem~\ref{thm:Hamil:G} is not a linearization of $G(\lam)$ as $\mathbb{M}^P_{\mathbf{c}_z}$ is singular.
\end{remark}

\begin{example} Let $G(\lam) :=\sum_{i=0}^5 \lambda^i A_i  + B^T(\lam E-A)^{-1}B$ be a $T$-even realization of $G(\lam)$ and $\mathcal{S} (\lam) $ be as in Proposition~\ref{prop:hamil}(b). Consider  the GFPR $\mathbb{L} (\lam)  =  (\lam \mathbb{M}^\mathcal{S}_{(-4:-3,-5)} - \mathbb{M}^\mathcal{S}_{(1:2,0)} ) \mathbb{M}^P_{1} \mathbb{M}^P_{-4}$ and $ \mathbb{Q} = \diag (I_n, I_n, -I_n, I_n, -I_n, I_r) $. Then
	$$ \mathbb{Q} \mathbb{L}(\lam) =  \left[ \begin{array}{ccccc|c}
	0 & -I_n & \lam I_n & 0 & 0 &0\\
	- I_n  & \lam A_{5} - A_{4} & \lam A_{4} & 0 & 0 &0\\
	- \lam I_n & - \lam A_{4} & - \lam A_{3} -A_{2} & -A_{1} & I_n&0\\
	0 & 0 & A_{1} & -\lam A_{1} + A_{0} & \lam I_n& B^T \\
	0 & 0 & I_n & - \lam I_n & 0 &0 \\ \hline
	0 & 0 & 0 & B & 0& A- \lam E \end{array}\right]$$
	is a $T$-even Rosenbrock strong linearization of $G(\lam)$.  Observe that $\mathbb{Q} \mathbb{L}(\lam)$ is a block penta-diagonal pencil.

Next, let $G(\lam) :=\sum_{i=0}^4 \lambda^i A_i  + B^T(\lam E-A)^{-1}B$ be a $T$-even realization. Consider  $\mathbb{L} (\lam)  :=  (\lam \mathbb{M}^\mathcal{S}_{(-4:-1)} - \mathbb{M}^\mathcal{S}_{0} ) \mathbb{M}^P_{(-4:-2,-4:-3,-4)}$ and $ \mathbb{Q} := \diag (I_n, -I_n, I_n, -I_n, I_r) $. Then
$$ \mathbb{Q} \mathbb{L} (\lam)  =  \left[ \begin{array}{cccc|c}
0 & 0 &  - A_4 & \lam A_4 &0\\
0 & A_4 & -\lam A_4 + A_3 & -\lam A_3 &0\\
- A_4 & \lam A_4-A_3 & \lam A_3 - A_2 & \lam A_2 &0\\
-\lam A_4 & - \lam A_3 &- \lam A_2 & -\lam A_1 - A_0 &  B^T\\ \hline \\[-1.1em]
0 & 0& 0& B & A- \lam E
\end{array}\right]$$ is a $T$-even Rosenbrock strong linearization of $G(\lam)$ when $A_4$ is nonsingular.

\end{example}

Next, we consider a Hamiltonian realization of $G(\lam)$ and construct a Hamiltonian strong linearization of $G(\lam).$

\begin{theorem}  Let $G(\lam)$ be Hamiltonian and  $\mathcal{S} (\lam)$ be a Hamiltonian realization of $G(\lam)$ as given in Proposition~\ref{prop:hamil}(a). Assume that $Ind(\textbf{z} +m) =0$ when the leading coefficient of $P(\lam)$ is singular, where $\textbf{z}$ is as given in Theorem~\ref{thm:Hamil:G}.   Then
	$\mathbb{T} (\lambda) :   =
	\left[
	\begin{array}{@{}c|c@{}}
		\mathbf{s} \,Q L(\lam) &  e_{m- i_0(\textbf{w})} \otimes B^T J^T \\[.1em] \hline \\[-1em]
		e^T_{m- i_0(\textbf{w})} \otimes B  & A - \lam I_r \\
	\end{array}
	\right]$  is Hamiltonian and is a Rosenbrock strong linearization of   $G (\lam)$, where $\textbf{w}$ and  $\mathbf{s}\, QL(\lam)$ are as given in Theorem~\ref{thm:Hamil:G}.

The transfer function   $\mathbb{G}(\lam) : = \mathbf{s} Q L(\lam) +  (e_{m-i_0(\textbf{w})} \otimes B^TJ^T ) ( \lam I_r - A)^{-1} (e^T_{m-i_0(\textbf{w})} \otimes B) $ of  $ \mathbb{T}(\lam)$  is  Hamiltonian.	
\end{theorem}

\begin{proof}  Define $\widehat{\mathcal{S}}(\lam) : = \mathbb{J}_{n,r} \mathcal{S} (\lam) = \left[ \begin{array}{@{}c|c@{}} P(\lam) &  B^T J^T \\ \hline  JB & JA - \lam J \end{array} \right]. $ Since $ A$ is Hamiltonian, we have $JA -\lam J $ is $T$-even. This shows that  $ \widehat{\mathcal{S}} (\lam)$ is a $T$-even realization of $G(\lam).$  Hence by Theorem~\ref{thm:Hamil:G},
	\begin{equation} \label{eqn:2ndfSSn}
	\widehat{\mathbb{L}} (\lambda) : =
		\left[
	\begin{array}{@{}c|c@{}}
	\mathbf{s} \, Q L(\lam) &  e_{m-i_0(\textbf{w})} \otimes  B^TJ^T \\[.1em] \hline \\[-1em]
	e^T_{m-i_0(\textbf{w})} \otimes J B & JA -\lam J\\
	\end{array}
	\right]
	\end{equation}
	is a $T$-even Rosenbrock strong linearizations of $\widehat{\mathcal{S}} (\lam).$ Note that $\widehat{\mathbb{L}} (\lambda) = \mathbb{J}_{mn,r} \mathbb{T} (\lambda)$, where $\mathbb{J}_{mn,r} := \diag (I_{mn},  J)$. Since  $ \widehat{\mathbb{L}} (\lambda)$ is $T$-even, it follows that $\mathbb{T}(\lam)$ is Hamiltonian, that is, $( \mathbb{J}_{mn,r} \mathbb{T} (-\lambda))^T =\mathbb{J}_{mn,r} \mathbb{T} (\lambda).$  Further, since $\widehat{\mathbb{L}} (\lambda)$ is a Rosenbrock strong linearization of $\widehat{\mathcal{S}} (\lam)$ and $ \widehat{\mathcal{S}} (\lam) = \mathbb{J}_{n,r} \mathcal{S} (\lam)$, it follows that $\mathbb{T}(\lam) $ is a  Rosenbrock strong linearization of $\mathcal{S} (\lam)$. Obviously the transfer function $\mathbb{G}(\lam)$ is  Hamiltonian.   \end{proof}

\subsection{Skew-Hamiltonian linearizations}
Recall that a rational matrix $G(\lam)$ is said to be   skew-Hamiltonian (i.e., $T$-odd) if $G(- \lam)^T = -G(\lam)$.

\begin{proposition} \label{prop:skew_hamil} Let  $G(\lam)$ be   $T$-odd. Then there exists a minimal $T$-odd realization of $G(\lam)$ of the form $G(\lam): = P(\lam) + B^T (\lam I_r - A)^{-1}B$, where  $P(\lam)$ and $\lam I_r -A$ are $T$-odd. Thus the system matrix
	$\mathcal{S} (\lam) : =
	\left[ \begin{array}{@{}c|c@{}}
	P(\lam) &  -B^T \\ \hline  B &  \lam I_r -A
	\end{array}
	\right] $  is $T$-odd.
\end{proposition}

\begin{proof}  Since $G(\lam) = P(\lam) +G_{sp}(\lam)$ is $T$-odd, it follows that both  $P(\lam)$ and $G_{sp}(\lam)$ are $T$-odd. Since $G_{sp}(\lam)$ is $T$-odd and strictly proper,
 there exists a minimal $T$-odd realization of $G_{sp}(\lam)$ of the form  $G_{sp}(\lam) =  B^T (\lam I_r - A)^{-1}B$, where $A$ is skew-symmetric; see~\cite{hilscot}. Since $A$ is skew-symmetric, we have $\lam I_r - A$ is $T$-odd.  This
shows that $G(\lam) = P(\lam)+ B^T (\lam I_r - A)^{-1}B$ is a minimal $T$-odd realization of $G(\lam)$ and that the  system matrix
	$\mathcal{S} (\lam)  $   is $T$-odd.
\end{proof}

The next result gives  $T$-odd Rosenbrock strong  linearizations of $G(\lam)$.

\begin{theorem} \label{thm:skewHamil:G} Let $G(\lam)$ be $T$-odd and
	$\mathcal{S} (\lam) $ be as given in Proposition~\ref{prop:skew_hamil}.  Let $h, \textbf{w}, \textbf{c}_w, \textbf{z} $ and $ \textbf{c}_z$ be as in Theorem~\ref{eventhm}.  Consider the GFPR  $\mathbb{L}(\lambda) := \big( \lambda \mathbb{M}^\mathcal{S}_{\textbf{z}} - \mathbb{M}^\mathcal{S}_{\textbf{w}} \big)  \mathbb{M}^P_{\textbf{c}_w}  \mathbb{M}^P_{\textbf{c}_z} $ associated with $\mathcal{S}(\lam)$. Then  there exists a unique quasi-identity matrix $\mathbb{Q} :=\diag( \mathbf{s}\, Q, I_r)$  such that  	$$ \mathbb{Q}\mathbb{L} (\lambda)   =
	\left[
	\begin{array}{@{}c|c@{}}
	 \mathbf{s} Q L(\lam) &  - e_{m -i_0(\textbf{w})} \otimes B^T \\[.1em] \hline \\[-1em]
	e^T_{m -i_0(\textbf{w})} \otimes B  & \lam I_r -A \\
	\end{array}
	\right] $$  is  $T$-odd, where $Q$ and $L(\lam)$ are as in Theorem~\ref{eventhm} and $\mathbf{s}$ is the $(m - i_0(\textbf{w}))$-th parameter of $Q$.
	
Assume that $Ind(\textbf{z} +m) =0$ when leading coefficient of $P(\lam)$ is singular. Then $\mathbb{Q} \mathbb{L}(\lam)$ is a $T$-odd Rosenbrock strong linearization of  $G(\lam)$.  The transfer function
	$\mathbb{G} (\lam) : =  \mathbf{s} \,Q L(\lam) +   (e_{m-i_0(\textbf{w})} \otimes B^T)  ( \lam I_r -A )^{-1} (e^T_{m-i_0(\textbf{w})} \otimes B) $ of  $\mathbb{Q} \mathbb{L}(\lam)$  is $T$-odd.  	
%
\end{theorem}

\begin{proof} By Theorem~\ref{gfprptoge_RCh4}, we have
	$$ \mathbb{L}(\lambda)  =
	\left[
	\begin{array}{@{}c|c@{}}
	L(\lam) &  e_{m-i_0(\textbf{w})} \otimes (-B^T) \\[.1em] \hline \\[-1em]
	e^T_{m-c_0( \textbf{w}, \textbf{c}_w)} \otimes B  &  \lam I_r -A   \\
	\end{array}
	\right],$$
	where $L(\lam)$ is as given in Theorem~\ref{eventhm}. It is shown in the proof of Theorem~\ref{thm:Hamil:G} that $ i_0(\textbf{w}) = c_0(\textbf{w}, \textbf{c}_w) $. Set $\alpha: = i_0(\textbf{w})$. Then $Q (e_{m-\alpha} \otimes I_n ) = \mathbf{s} \, (e_{m-\alpha} \otimes I_n ). $ Note that $ \mathbf{s}\mathbf{s}=1$. Consequently, we have
	\begin{equation} \label{eqn:odd:G}
	\mathbb{Q} \mathbb{L}(\lam)  =
	\left[
	\begin{array}{@{}c@{\,}|@{\,}c@{}}
	\mathbf{s} \, Q L(\lam) &  \mathbf{s} \, Q (e_{m-\alpha} \otimes (-B^T) )\\ \hline \\[-1em]
	e^T_{m-\alpha} \otimes B  & \lam I_r -A  \\
	\end{array}
	\right] =
	\left[
	\begin{array}{@{}c@{\,}|@{\,}c@{}}
	\mathbf{s} \, Q L(\lam) &  e_{m-\alpha} \otimes (-B^T) \\ \hline \\[-1em]
	e^T_{m-\alpha} \otimes B  & \lam I_r -A  \\
	\end{array}
	\right].
	\end{equation}
By Theorem~\ref{eventhm}, $ \mathbf{s} \, Q L(\lam)$ is $T$-odd. Since $\lam I_r -A$ is  $T$-odd, it follows from (\ref{eqn:odd:G}) that  $ \mathbb{Q} \mathbb{L}(\lam)$ is $T$-odd.
	
By the same arguments as given in the proof of Theorem~\ref{thm:Hamil:G}, it follows that $\mathbb{Q} \mathbb{L}(\lam)$ is a Rosenbrock strong linearization of $G(\lam)$. Obviously, the transfer function $\mathbb{G}(\lam)$  is $T$-odd.
\end{proof}

\begin{example} Let $G(\lam) = \sum_{i=0}^5 \lambda^i A_i + B^T(\lam I_r- A)^{-1}B$ be a $T$-odd realization of $G(\lam)$ and $\mathcal{S} (\lam) $ be as given in Proposition~\ref{prop:skew_hamil}. Set  $ \mathbb{Q} := \diag (I_n, -I_n, I_n, -I_n,-I_n, I_r) $ and consider the GFPR $\mathbb{L} (\lam): =   (\lam \mathbb{M}^\mathcal{S}_{(-4:-3,-5)} - \mathbb{M}^\mathcal{S}_{(1:2,0)} ) \mathbb{M}^P_{1} \mathbb{M}^P_{-4}$. Then
	$$ \mathbb{Q} \mathbb{L}(\lam) = \left[ \begin{array}{ccccc|c} 0 & -I_n & \lam I_n  & 0 & 0& 0\\
	I_n & - \lam A_{5} + A_{4}  & - \lam A_{4} & 0 & 0& 0\\
	\lam I_n & \lam A_{4} & \lam A_{3}  + A_{2} & A_{1} & -I_n& 0\\
	0 & 0 & -A_{1} & \lam A_{1}-A_{0} & -\lam I_n & - B^T\\
	0 & 0 & I_n & -\lam  I_n & 0 & 0\\ \hline
	0 & 0 & 0 & B & 0& \lam I_r -A  \end{array}\right] $$
	is a $T$-odd Rosenbrock strong linearization of $G(\lam)$. Notice that $\mathbb{Q} \mathbb{L}(\lam)$ is a block penta-diagonal pencil.	
	
Next, let $G(\lam) = \sum_{i=0}^5 \lambda^i A_i + B^T(\lam I_r- A)^{-1}B$ be a $T$-odd realization. Consider  $\mathbb{L} (\lam)  := (\lam \mathbb{M}^\mathcal{S}_{(-4:-1)} - \mathbb{M}^\mathcal{S}_{0} ) \mathbb{M}^P_{(-4:-2,-4:-3,-4)}$ and $ \mathbb{Q} := \diag (I_n, -I_n, I_n, -I_n, I_r) $. Then
$$ \mathbb{Q} \mathbb{L} (\lam)  =  \left[ \begin{array}{cccc|c}
0 & 0 &  - A_4 & \lam A_4 &0\\
0 & A_4 & -\lam A_4 + A_3 & -\lam A_3 &0\\
- A_4 & \lam A_4-A_3 & \lam A_3 - A_2 & \lam A_2 &0\\
-\lam A_4 & - \lam A_3 &- \lam A_2 & -\lam A_1 - A_0 &  -B^T\\ \hline \\[-1.1em]
0 & 0& 0& B & \lam I_r -A
\end{array}\right]$$
is a $T$-odd Rosenbrock strong linearization of $G(\lam)$ when $A_4$ is nonsingular.	
	
\end{example}

\subsection{Skew-symmetric linearizations} Suppose that $G(\lam)$ is skew-symmetric, that is, $G( \lam)^T = -G(\lam)$. Since $ G(\lam) = P(\lam)+G_{sp}(\lam)$, it follows that $ P(\lam) $ and $ G_{sp}(\lam)$ are skew-symmetric.

\begin{definition} Suppose that $G(\lam)$ is skew-symmetric.
 \begin{itemize}

\item[(a)] A realization of $G(\lam)$ of the form  $G(\lam) =P(\lam) + C (\lam I_r - A)^{-1}B$  is said to be a skew-Hamiltonian realization of $G(\lam)$ if $P(\lam)$ is skew-symmetric,  $A$ is skew-Hamiltonian with $ r = 2\ell$ and  $C^T=JB.$
\item[(b)] A system matrix $\mathcal{S}(\lam)$ of the form  $\mathcal{S}(\lam) : =
		\left[ \begin{array}{@{}c|c@{}}
		P(\lam) &  -C \\ \hline  B &  \lam I_r -A
		\end{array}
		\right] $  is said to be a skew-Hamiltonian system matrix if  $ r = 2\ell$ and    $\big (\mathbb{J}_{n,r} \, \mathcal{S}  (\lam) \big )^T = - \mathbb{J}_{n,r} \, \mathcal{S}  (\lam)$, where $\mathbb{J}_{n,r}: = \diag (I_n , J)$.
\item[(c)] A realization of $G(\lam)$ of the form  $G(\lam) =P(\lam) + C (\lam E - A)^{-1}B$ with $E$ being nonsingular is said to be a skew-symmetric realization of $G(\lam)$ if $ C= B^T$ and both $P(\lam)$ and $ \lam E - A $ are  skew-symmetric.
\end{itemize}
\end{definition}

 \begin{remark}
Observe that  $G(\lam) =P(\lam) + C (\lam I_r - A)^{-1}B$  is a skew-Hamiltonian realization of $G(\lam)$  if and only if  $\mathcal{S}(\lam) : =
		\left[ \begin{array}{@{}c|c@{}}
		P(\lam) &  -C \\ \hline  B &  \lam I_r -A
		\end{array}
		\right] $  is a skew-Hamiltonian system matrix of $G(\lam).$ On the other hand, $G(\lam) =P(\lam) + C (\lam E - A)^{-1}B$ is a skew-symmetric realization of $G(\lam)$ if and only if  $\mathcal{S}(\lam) : =
		\left[ \begin{array}{@{}c|c@{}}
		P(\lam) &  -C \\ \hline  B & \lam E -A
		\end{array}
		\right] $ is a skew-symmetric system matrix of $G(\lam).$

\end{remark}

For convenience, we often refer to $\mathcal{S}(\lam)$ as a skew-symmetric (resp., skew-Hamiltonian) realization of $G(\lam)$ when $\mathcal{S}(\lam)$ is skew-symmetric (resp., skew-Hamiltonian).

\begin{proposition}\label{prop:skew_symm} Suppose that  $G(\lam)$  is  skew-symmetric. Then we have  the following:
	\begin{itemize}
		\item[(a)] There exists a minimal skew-Hamiltonian realization of $G(\lam)$ of the form  $G(\lam) = P(\lam) + C (\lam I_r - A)^{-1}B$  with $ r = 2 \ell$ and $JB = C^T.$   Thus the system matrix
		$\mathcal{S} (\lam) : =
		\left[ \begin{array}{@{}c|c@{}}
		P(\lam) &  -B^T J^T\\ \hline  B &  \lam I_r -A
		\end{array}
		\right] $ associated with $G(\lam)$ is skew-Hamiltonian.

		\item[(b)] There exists a minimal skew-symmetric realization of $G(\lam)$ of the form $G(\lam) = P(\lam) + B^T (\lam E- A)^{-1}B$.  Thus the   system matrix
		$\mathcal{S} (\lam)  =
		\left[ \begin{array}{@{}c|c@{}}
		P(\lam) &  -B^T\\ \hline  B &  \lam E -A
		\end{array}
		\right]$ associated with $G(\lam)$ is skew-symmetric.
	\end{itemize}
\end{proposition}

\begin{proof}  Since $G(\lam) = P(\lam)+G_{sp}(\lam)$ is skew-symmetric, we have both $P(\lam)$ and $G_{sp}(\lam)$ are skew-symmetric. Also since $G_{sp}(\lam)$ is strictly proper and skew-symmetric, there exists a minimal skew-Hamiltonian realization of $G_{sp}(\lam)$ of the form $G_{sp}(\lam) =  C (\lam I_r - A)^{-1} B$ with $ r = 2\ell$ and $ JB = C^T$; see~\cite{hilscot}. Hence  $G(\lam) = P(\lam) + C (\lam I_r - A)^{-1} B$ is a minimal skew-Hamiltonian realization of $G(\lam).$ Obviously the system matrix $\mathcal{S} (\lam)$ is skew-Hamiltonian, that is, $  (\mathbb{J}_{n,r} \, \mathcal{S} ( \lam))^T = -\mathbb{J}_{n,r} \,\mathcal{S} (\lam),$ where $\mathbb{J}_{n,r} := \diag (I_n,  J)$. This proves (a).

By part (a), $G(\lam) = P(\lam) + B^T J^T (\lam I_r - A)^{-1} B = P(\lam) + B^T( \lam J - AJ)^{-1} B.$  Since $A$ is skew-Hamiltonian, it follows that $\lam J - A J$ is skew-symmetric. Hence setting  $E := J$ and redefining $A := AJ$, it follows that $G(\lam) = P(\lam) + B^T (\lam E- A)^{-1}B$ is a minimal skew-symmetric realization of $G(\lam).$
Evidently, the system matrix $\mathcal{S} (\lam)$ is skew-symmetric, that is, $\mathcal{S} (\lam)^T = -\mathcal{S} (\lam)$. This proves (b).
\end{proof}

Let $\alpha$ be a permutation of $\{0:k \}$ for  $k \geq 0$ with  csf$(\alpha)$ being the column standard form of $\alpha$. Then an index $s \in \{ 0: k-1\}$ is said to be a right index of type-1 relative to $\alpha$ if there is a string $(s : t)$ in the csf$(\alpha)$ such that $s <  t$, see~\cite{bfs2}.

\begin{definition} [\cite{bfs2}, Associated simple tuple] \label{assosimtupp0318oct2254} Let $\alpha$ be a permutation of  $\{0: k \}$ for some $k \geq 0$.  Suppose that csf$(\alpha)= (\textbf{b}_{d}, \textbf{b}_{d-1}, \ldots , \textbf{b}_{1} )$, where $\textbf{b}_{i} = (a_{i-1} +1 : a_i)$ for $i= 2:d$ and  $\textbf{b}_{1} = (0 : a_1).$ If $s$ is a right index of type-1  relative to $\alpha$ then the simple tuple associated with $(\alpha, s)$ is denoted by $z_r(\alpha,s)$ and is given by
	\begin{itemize}
		\item $ z_r(\alpha,s):= (\textbf{b}_{d}, \textbf{b}_{d-1}, \ldots, \textbf{b}_{h+1}, \widetilde{\textbf{b}}_{h}, \widetilde{\textbf{b}}_{h-1},\textbf{b}_{h-2}, \ldots , \textbf{b}_{1} ) $ if $ s = a_{h-1} +1 \neq 0, $
		$ \text{where } \widetilde{\textbf{b}}_{h} = (a_{h-1}+2 : a_h) ~ \mbox{and}~\widetilde{\textbf{b}}_{h-1} = (a_{h-2}+ 1 : a_{h-1}+1 ).$
		
		\item $ z_r(\alpha,s):= (\textbf{b}_{d}, \textbf{b}_{d-1},    \ldots , \textbf{b}_{2}, \widetilde{\textbf{b}}_{1} , \widetilde{\textbf{b}}_{0}  ) $ if $s=0$, where $ \widetilde{\textbf{b}}_{1} = (1 : a_1) ~ \mbox{and}~\widetilde{\textbf{b}}_{0} = (0 ).$
		
	\end{itemize}
	
\end{definition}

\begin{definition}[\cite{bfs2}, Type-1 index tuple] \label{leftrightituple} Let $\alpha$ be a permutation of $\{0: k \}$, $k \geq 0$, and let $\beta := (s_1, \ldots, s_r)$ be an index tuple containing indices from $ \{0: k-1 \}$. Then $\beta$ is said to be a right index tuple of type-1 relative to $\alpha$ if, for $i=1:r$, $s_i$ is a right index of type-1 relative to $z_r(\alpha, (s_1,\ldots, s_{i-1}))$, where $z_r(\alpha,(s_1,\ldots, s_{i-1})): = z_r(z_r(\alpha,(s_1,\ldots, s_{i-2})), s_{i-1})$ for $i>2.$
\end{definition}

We need the following result which is a particular case of \cite[Theorem 3.15]{bfs2}.

\begin{theorem} \cite{bfs2} \label{thm:skew_sym} Let $P(\lambda)$ be skew symmetric and let $0 \leq h \leq m-1$ be even. Let $\textbf{w}$ be the simple admissible tuple of $\{ 0 : h\}$ and $\textbf{c}_{w}$ be the symmetric complement of $\textbf{w}$. Let $\textbf{z}+m$ be any admissible tuple of $\{0: m-h-1\}$. Let $\textbf{c}_{z}+m$ be the symmetric complement of $\textbf{z}+m$. Let $\textbf{t}_w $ containing indices from $\{0 : h-1\} $ and $ \textbf{t}_{z}+m$ containing indices from $ \{0: m-h-2\}$ be right  index tuples of type-1 relative to $rev(\textbf{w})$ and $rev( \textbf{z}+m)$, respectively.  Consider
	$$ L(\lambda) := M^P_{rev(\textbf{t}_z)}  M^P_{rev(\textbf{t}_w)} \big( \lambda M^P_{\textbf{z}} - M^P_{\textbf{w}} \big)  M^P_{\textbf{c}_w} M^P_{\textbf{t}_w} M^P_{\textbf{c}_z}  M^P_{\textbf{t}_z}.$$
	Then, up to multiplication by $-1$, there exists a unique quasi-identity matrix $Q$ such that $QL(\lambda)$ is skew-symmetric.
\end{theorem}

We now construct skew-symmetric Rosenbrock strong linearizations of $G(\lam).$

\begin{theorem}  \label{thm:skew-symmG} Let $G(\lam)$ be skew-symmetric and $\mathcal{S} (\lam) $ be a skew-symmetric realization of $G(\lam)$ as in Proposition~\ref{prop:skew_symm}(b). Let $h, \textbf{w}, \textbf{c}_w, \textbf{t}_w, \textbf{z}, \textbf{c}_z$ and $\textbf{t}_z$ be as in Theorem~\ref{thm:skew_sym}. Consider the GFPR  $	\mathbb{L}(\lambda)  := \mathbb{M}^P_{rev(\textbf{t}_z)}  \mathbb{M}^P_{rev(\textbf{t}_w)}   \big( \lambda \mathbb{M}^\mathcal{S}_{\textbf{z}} - \mathbb{M}^\mathcal{S}_{\textbf{w}} \big)  \mathbb{M}^P_{\textbf{c}_w} \mathbb{M}^P_{\textbf{t}_w}  \mathbb{M}^P_{\textbf{c}_z}  \mathbb{M}^P_{\textbf{t}_z}$ associated with $\mathcal{S}(\lam)$. Then  there exists a unique quasi-identity matrix $\mathbb{Q} := \diag(\mathbf{s} \, Q, I_r)$  such that
	\begin{align*}
	\mathbb{Q}\mathbb{L}(\lambda) & =
	\left[
	\begin{array}{c|c}
	\mathbf{s} \, Q L(\lam) & - e_{m-\alpha} \otimes B^T \\[.1em] \hline \\[-1em]
	e^T_{m-\alpha} \otimes B  &   \lam E -A  \\
	\end{array}
	\right],
	\end{align*}
	is   skew-symmetric, where  $Q$ and $L(\lam)$ are as in Theorem~\ref{thm:skew_sym} and $\mathbf{s}$ is the $(m - \alpha)$-th parameter of $Q$ with $\alpha:= c_0( \textbf{w}, \textbf{c}_w, \textbf{t}_w). $

Assume that $Ind(\textbf{z} +m) =0$ when the leading coefficient $A_m$ of $P(\lam)$ is singular. Further, suppose that $0 \notin \textbf{t}_{w}$ (resp., $-m \notin \textbf{t}_{z}$) when $A_0$ (resp., $A_m$) is singular. Then $\mathbb{Q} \mathbb{L}(\lam)$ is a skew-symmetric Rosenbrock strong linearization of $G(\lam)$. The transfer function $\mathbb{G}(\lam) := \mathbf{s}\, QL(\lam)  + (e_{m-\alpha}\otimes B^T ) (\lam E - A)^{-1} (e_{m-\alpha}^T\otimes B)$ of $\mathbb{Q} \mathbb{L}(\lam)$ is  skew-symmetric.
\end{theorem}

\begin{proof} By Theorem~\ref{gfprptoge_RCh4}, we have
	$$ \mathbb{L}(\lambda)  =
	\left[
	\begin{array}{c|c}
	L(\lam) &  e_{m-i_0( rev(\textbf{t}_w), \textbf{w})} \otimes (- B^T) \\[.1em] \hline \\[-1em]
	e^T_{m-c_0( \textbf{w}, \textbf{c}_w, \textbf{t}_w)} \otimes B  &  \lam E -A  \\
	\end{array}
	\right],$$
	where $L(\lam)$ is as in Theorem~\ref{thm:skew_sym}. Next, we show that $i_0( rev(\textbf{t}_w), \textbf{w}) = c_0( \textbf{w}, \textbf{c}_w, \textbf{t}_w)$. If $ h=0$ then $\textbf{w} = (0)$ and $\textbf{c}_w = \emptyset = \textbf{t}_w$. Thus $i_0( rev(\textbf{t}_w), \textbf{w}) = 0 = c_0( \textbf{w}, \textbf{c}_w, \textbf{t}_w)$. Next, suppose that $h > 0$. Then  we have $\textbf{w}= (h-1:h,h-3:h-2, \ldots,1:2,0)$ and $\textbf{c}_w= ( h-1, h-3, \ldots , 3 ,1).$ This implies that $c_0( \textbf{w}, \textbf{c}_w, \textbf{t}_w) = 2+c_2(\textbf{t}_w)$ and $i_0(rev(\textbf{t}_w),\textbf{w}) =  2+i_2(rev(\textbf{t}_w)) = 2+ c_2(\textbf{t}_w) $. Hence $i_0( rev(\textbf{t}_w), \textbf{w}) = c_0( \textbf{w}, \textbf{c}_w, \textbf{t}_w)$.
	
	By Theorem~\ref{thm:skew_sym}, we have $\mathbf{s}\, QL(\lambda)$ is  skew-symmetric. Note that $Q (e_{m-\alpha} \otimes I_n ) = \mathbf{s} \, (e_{m-\alpha} \otimes I_n ) $ and $ \mathbf{s}\mathbf{s}=1$. Consequently, we have
	\begin{equation} \label{skew:symm:form}
	\mathbb{Q}  \mathbb{L}(\lambda)  =
	\left[
	\begin{array}{@{}c@{\,}|@{\,}c@{}}
	\mathbf{s} \, Q L(\lam) &  \mathbf{s}\, Q (e_{m-\alpha} \otimes (- B^T)) \\[.1em] \hline \\[-1em]
	e^T_{m-\alpha} \otimes B  &  \lam E -A  \\
	\end{array}
	\right] =
	\left[
	\begin{array}{@{}c@{\,}|@{\,}c@{}}
	\mathbf{s} \, Q L(\lam) &  e_{m-\alpha} \otimes (- B^T) \\[.1em] \hline \\[-1em]
	e^T_{m-\alpha} \otimes B  &  \lam E -A  \\
	\end{array}
	\right].
	\end{equation}
Since $\mathbf{s}\, QL(\lambda)$ and $  \lam E -A$ are skew-symmetric, it follows from (\ref{skew:symm:form})  that  $ \mathbb{Q} \mathbb{L}(\lam)$ is skew-symmetric.

Since $0 \notin \textbf{t}_{w}$ (resp., $-m \notin \textbf{t}_{z}$) when $A_0$ (resp., $A_m$) is singular, the matrix assignments of $\textbf{t}_w , rev(\textbf{t}_w), \textbf{t}_z$ and $rev(\textbf{t}_z)$ are nonsingular. Hence by taking $ \sig : =\textbf{w}, \tau : = \textbf{z}, \sig_1 := rev(\textbf{t}_{w}), \sig_2 := (\textbf{c}_{w}, \textbf{t}_{w})$, $ \tau_1 := rev(\textbf{t}_{z})$ and $ \tau_2 := (\textbf{c}_{z}, \textbf{t}_{z})$, it follows from Theorem~\ref{thm:gfpr:RSL} that  $\mathbb{L}(\lam)$ is a Rosenbrock strong linearization of $\mathcal{S}(\lambda)$  if the matrix assignments for $\textbf{c}_{w}$ and $\textbf{c}_{z}$ are nonsingular. By the similar arguments as given in the proof Theorem~\ref{thm:Hamil:G}, it follows that the matrix assignments for  $\textbf{c}_{w}$ and $\textbf{c}_{z}$ are nonsingular.
 \end{proof}

\begin{example} Let $G(\lam) = \sum_{i=0}^5 \lambda^i A_i + B^T(\lam E-A)^{-1} B$ be skew-symmetric and $\mathcal{S} (\lam) $ be as in Proposition~\ref{prop:skew_symm}(b). Define $	 \mathbb{L}(\lambda)  :=    \big( \lambda \mathbb{M}^\mathcal{S}_{(-4:-3,-5)} - \mathbb{M}^\mathcal{S}_{(1:2,0)} \big)  \mathbb{M}^P_{1} \mathbb{M}^P_{-4}$ and $ \mathbb{Q} := \diag (I_n,-I_n, -I_n, -I_n, I_n,I_r)$. Then
	$$\mathbb{Q} \mathbb{L} (\lam) =
	\left[\begin{array}{ccccc|c}
	0 & -I_n & \lam I_n & 0 & 0  & 0\\
	I_n & - \lam A_{5} + A_{4} & -\lam A_{4} & 0 & 0 & 0\\
	-\lam I_n  & - \lam A_{4} & - \lam A_{3}  -A_{2} & -A_{1} & I_n & 0\\
	0 & 0 & -A_{1} & \lam A_{1} - A_{0} & - \lam I_n  & -B^T\\
	0 & 0 & -I_n & \lam I_n  & 0 &0\\ \hline
	0 & 0& 0& B& 0& \lam E - A \end{array}\right] $$
	is a skew-symmetric Rosenbrock strong linearization of $G(\lam)$. Observe that $\mathbb{Q} \mathbb{L} (\lam)$ is a block penta-diagonal pencil.
	
Next, let $G(\lam) = \sum_{i=0}^4 \lambda^i A_i + B^T(\lam E-A)^{-1} B$ be skew-symmetric. Consider the GFPR $ \mathbb{L}(\lam) = (\lam \mathbb{M}^\mathcal{S}_{(-4:-3)} - \mathbb{M}^\mathcal{S}_{(1:2,0)}) \mathbb{M}^P_{1} \mathbb{M}^P_{-4}$ and $\mathbb{Q}= \diag (I_n, I_n, I_n, -I_n,I_r)$. Then
	$$ \mathbb{Q} \mathbb{L} (\lam) =   \left[\begin{array}{cccc|c}
	-A_{4} & \lam A_{4} & 0 & 0 &0\\
	\lam A_{4} & \lam A_{3} + A_2  & A_{1} & - I_n &0\\
	0 & A_{1} & - \lam A_{1} + A_{0} & \lam I_n & -B^T \\
	0 &  I_n  & -\lam I_n & 0 & 0\\ \hline \\[-1.1em]
	0 & 0& B & 0 & \lam E -A
	\end{array}\right]$$
	is a skew-symmetric Rosenbrock strong linearization of  $G(\lam)$ when $A_4$ is nonsingular.

\end{example}

Next, we construct skew-Hamiltonian strong linearizations of $G(\lam).$

\begin{theorem}  \label{thm:skew-symmJS} Let $G(\lam)$ be skew-symmetric and  $\mathcal{S} (\lam) $ be a skew-Hamiltonian realization of $G(\lam)$ as in  Proposition~\ref{prop:skew_symm}(a). Let $ \textbf{w}, \textbf{c}_w, \textbf{t}_w, \textbf{z}, \textbf{c}_z$ and $\textbf{t}_z$ be as in Theorem~\ref{thm:skew-symmG}. Suppose that $0 \notin \textbf{t}_{w}$ (resp., $-m \notin \textbf{t}_{z}$) when $A_0$ (resp., $A_m$) is singular. Assume that $Ind(\textbf{z} +m) =0$ when $A_m$ is singular. Then
	$$\mathbb{T} (\lambda) :   =
	\left[
	\begin{array}{@{}c|c@{}}
	\mathbf{s} \,Q L(\lam) & - e_{m- \alpha} \otimes B^T J^T \\[.1em] \hline \\[-1em]
	e^T_{m- \alpha} \otimes B  & \lam I_r - A \\
	\end{array}
	\right]$$  is a skew-Hamiltonian  Rosenbrock strong linearization of  $G(\lam)$, where $\alpha$ and $\mathbf{s} \,  QL(\lam)$ are as in Theorem~\ref{thm:skew-symmG}. The transfer function $\mathbb{G}(\lam) := \mathbf{s} \, QL(\lam)  + (e_{m-\alpha}\otimes B^T J^T) (\lam I_r - A)^{-1} (e_{m-\alpha}^T\otimes B)$ of  $ \mathbb{T}(\lam)$  is  skew-symmetric.
\end{theorem}

\begin{proof}  Define $\widehat{\mathcal{S}}(\lam) : = \mathbb{J}_{n,r} \mathcal{S} (\lam)   = \left[ \begin{array}{@{}c|c@{}} P(\lam) &  - B^T J^T \\ \hline  JB & \lam J - JA \end{array} \right]. $ Since $ A$ is skew-Hamiltonian, we have $\lam J - JA $ is skew-symmetric. Hence $ \widehat{\mathcal{S}} (\lam)$ is skew-symmetric as $P(\lam)$ and $\lam J - JA$ are skew-symmetric. Now by Theorem~\ref{thm:skew-symmG},
	\begin{equation} \label{eqn:2ndfSS}
	\widehat{\mathbb{L}} (\lambda) : =\left[
	\begin{array}{@{}c|c@{}}
	\mathbf{s}\,Q L(\lam) & - e_{m-\alpha} \otimes  B^TJ^T \\[.1em] \hline \\[-1em]
	e^T_{m-\alpha} \otimes J B & \lam J - J A \\
	\end{array}
	\right]
	\end{equation}
is a skew-symmetric Rosenbrock strong linearizations of $\widehat{\mathcal{S}} (\lam)$, where $\alpha$ and  $\mathbf{s}\, QL(\lam)$ are as in Theorem~\ref{thm:skew-symmG}. Note that $\widehat{\mathbb{L}} (\lambda) = \mathbb{J}_{mn,r} \mathbb{T} (\lambda)$. Since  $ \widehat{\mathbb{L}} (\lambda)$ is skew-symmetric, it follows that $\mathbb{T}(\lam)$ is skew-Hamiltonian, that is, $(\mathbb{J}_{mn,r} \mathbb{T} (\lambda))^T = - \mathbb{J}_{mn,r} \mathbb{T} (\lambda).$  Further, since $\widehat{\mathbb{L}} (\lambda)$ is a Rosenbrock strong linearization of $\widehat{\mathcal{S}} (\lam)$ and $ \widehat{\mathcal{S}} (\lam) = \mathbb{J}_{n,r} \mathcal{S} (\lam)$, it follows that $\mathbb{T}(\lam) $ is a  Rosenbrock strong linearization of $\mathcal{S} (\lam)$. Obviously $\mathbb{G}(\lam)$ is skew-symmetric and is the
 transfer function of  $\mathbb{T} (\lam).$
\end{proof}

\section{Recovery of eigenvectors and minimal bases} \label{fprsection_RCh4} We now describe the recovery of eigenvectors, minimal bases and minimal indices of $G(\lam)$ from those of the GFPRs of $G(\lam)$.  We need the following result.

\begin{theorem} \cite{rafiran1, Verghese_VK_79} \label{StoG}  Let $G(\lam)$ and $\mathcal{S}(\lam)$ be as in (\ref{minrel2_RCh4}) and (\ref{slamsystemmatrix_RCh4}), respectively.
	
	(I) Suppose that  $G(\lam)$ is singular. Let $ Z(\lam) :=\left[\begin{array}{c} Z_n(\lam) \\ Z_r(\lam) \end{array}\right]  $ be a matrix polynomial, where $Z_n(\lam)$ has $n$ rows and $ Z_r(\lam)$ has $r $ rows. If $Z(\lam)$  is a right  (resp., left) minimal  basis of $ \mathcal{S}(\lam)$  then $ Z_n(\lam)$ is a right (resp., left) minimal basis of $G(\lam).$  Further, the right (resp., left) minimal indices of $G(\lambda)$ and $\mathcal{S}(\lambda)$ are the same.

	(II) Suppose that $G(\lam)$ is regular and $ \mu \in \C$ is an eigenvalue of $G(\lam)$. Let $ Z := \left[\begin{array}{c} Z_{n} \\  Z_r \end{array}\right]$ be an $(n+r)\times p $ matrix such that $\rank(Z) = p,$ where $ Z_{n}$ has $ n$ rows and $Z_r$ has $r$ rows. If $Z$ is a  basis of $\mathcal{N}_r(\mathcal{S} (\mu))$ (resp., $\mathcal{N}_l(\mathcal{S} (\mu))$) then $Z_n$  is a  basis of $\mathcal{N}_r(G(\mu))$ (resp., $\mathcal{N}_l(G(\mu))$).
\end{theorem}

 Thus, in view of Theorem~\ref{StoG}, we only need to describe the recovery of eigenvectors, minimal bases and minimal indices of $\mathcal{S}(\lam)$ from those of the GFPRs of $G(\lam).$ To that end, we need the following result.

\begin{theorem} \cite{rafinami3, rafiran2} \label{eigpgf_RCh4} Consider the GF pencil  $\mathbb{T}_{\omega}(\lambda) := \lambda \mathbb{M}^\mathcal{S}_{-\omega_{1}} - \mathbb{M}^\mathcal{S}_{\omega_{0}}$ of $G(\lam)$ associated with a permutation $\omega := ({\omega_{0}}, {\omega_{1}})$ of $\{0: m\}$, where $0 \in \omega_0$ and $m \in \omega_1$. Then we have the following:

	{\bf (I)~Minimal bases}.  Suppose that $\mathcal{S}(\lam)$ is singular.   Then the maps
	\begin{equation*}
	\begin{array}{l}
	\mathbb{F}^{\scalebox{.4}{PGF}
	}_{\omega}(\mathcal{S}): \mathcal{N}_r(\mathbb{T}_\omega) \rightarrow \mathcal{N}_r(\mathcal{S}),~ \left[ \begin{array}{@{}c@{}}
	u(\lambda)\\ \hline
	v(\lambda) \\
	\end{array}
	\right] \mapsto \left[ \begin{array}{@{}c@{}}
	(e^T_{m-c_0(\omega_0)} \otimes I_n )u(\lambda) \\[.2em] \hline \\[-1em]
	v(\lambda) \end{array} \right], \\ \\[-1em]
	\mathbb{K}^{\scalebox{.4}{PGF}
	}_{\omega}(\mathcal{S}): \mathcal{N}_l(\mathbb{T}_\omega) \rightarrow \mathcal{N}_l(\mathcal{S}),~ \left[ \begin{array}{@{}c@{}}
	u(\lambda)\\ \hline
	v(\lambda) \\
	\end{array}
	\right] \mapsto \left[ \begin{array}{@{}c@{}}
	(e^T_{m-i_0(\omega_0)} \otimes I_n )u(\lambda) \\[.2em] \hline \\[-1em]
	v(\lambda) \end{array} \right],
	\end{array}
	\end{equation*}
are linear isomorphisms, where $ u(\lambda)\in\mathbb{C}(\lambda)^{mn}$ and $ v(\lambda)\in\mathbb{C}(\lambda)^{r}.$ Further, $\mathbb{F}^{\scalebox{.4}{PGF}
}_{\omega}(\mathcal{S})$ (resp., $\mathbb{K}^{\scalebox{.4}{PGF}
}_{\omega}(\mathcal{S})$) maps a minimal basis of $\mathcal{N}_r(\mathbb{T}_\omega)$ (resp., $\mathcal{N}_l(\mathbb{T}_\omega)$) to a minimal basis of $\mathcal{N}_r(\mathcal{S})$ (resp., $\mathcal{N}_l(\mathcal{S})$).

 Let $ \omega_1$ be given by $ \omega_1 :=(  \omega^{\ell}_1 , m, \omega^{r}_1).$ Set $\alpha : = ( rev(\omega^{\ell}_1), \omega_0, rev(\omega^{r}_1)) $. Let $ c (\alpha)$ and $ i(\alpha)$ be the total number of consecutions and inversions of the permutation $\alpha$, respectively. If $\varepsilon_1\leq   \cdots \leq\varepsilon_p$ are the right (resp., left) minimal indices of $\mathbb{T}_{\omega}(\lambda)$ then $\varepsilon_1 -i(\alpha)  \leq \cdots \leq \varepsilon_p - i(\alpha) $ (resp., $\varepsilon_1 -c(\alpha) \leq \cdots \leq \varepsilon_p - c(\alpha) $) are the right (resp., left) minimal indices of  $\mathcal{S}(\lambda).$

	{\bf (II)~Eigenvectors.} Suppose that $\mathcal{S}(\lam)$ is regular and $\mu \in \mathbb{C}$ is an eigenvalue of $\mathcal{S}(\lam)$. Let $ Z:= \left[\begin{array}{c} Z_{mn} \\  Z_r \end{array}\right]$ be an $(mn+r)\times p $ matrix such that $\rank(Z) = p,$ where $ Z_{mn}$ has $ mn$ rows and $Z_r$ has $r$ rows. If $Z$ is a basis of $\mathcal{N}_r(\mathbb{T}_{\omega}(\mu))$ (resp., $\mathcal{N}_l(\mathbb{T}_{\omega}(\mu))$) then $\left[\begin{array}{@{}c@{}}
	(e^T_{m-c_0(\omega_0)} \otimes I_n ) Z_{mn}  \\[.3em]   Z_r \end{array} \right] $ (resp., $\left[\begin{array}{@{}c@{}} 	(e^T_{m-i_0(\omega_0)} \otimes I_n ) Z_{mn}  \\[.3em]   Z_r \end{array} \right] $)  is a basis of $\mathcal{N}_r(\mathcal{S}(\mu))$ (resp., $\mathcal{N}_l(\mathcal{S}(\mu))$).
\end{theorem}

The pencil $\mathbb{T}_{\omega}(\lambda)$ in Theorem~\ref{eigpgf_RCh4} is referred to as a PGF (proper generalized Fiedler) pencil of $G(\lam)$ (also refer to as a PGF pencil of $\mathcal{S}(\lambda)$).


 For the rest of the paper, we only consider  GFPRs with nonsingular matrix assignments. Thus, if
$ \mathbb{L}(\lambda):=  \mathbb{M}_{(\tau_1,\sigma_1)} (Y_1,X_1) $ $(\lambda \mathbb{M}^\mathcal{S}_{\tau} -  \mathbb{M}^\mathcal{S}_{\sigma}) \mathbb{M}_{(\sigma_2,\tau_2)} (X_2,Y_2)$ is a GFPR of $\mathcal{S}(\lambda)$ then we assume that $X_j$ and $Y_j$, $j=1,2,$ are nonsingular matrix assignments.

\begin{theorem}\label{gfprrecoslame_RCh4} Let $ \mathbb{L}(\lambda):=  \mathbb{M}_{(\tau_1,\sigma_1)} (Y_1,X_1) (\lambda \mathbb{M}^\mathcal{S}_{\tau} -  \mathbb{M}^\mathcal{S}_{\sigma}) \mathbb{M}_{(\sigma_2,\tau_2)} (X_2,Y_2)$ be a  GFPR of $\mathcal{S}(\lambda)$. Let $ Z(\lam) := \left[\begin{array}{c} Z_{mn}(\lam) \\  Z_r(\lam) \end{array} \right]$ be an $(mn+r)\times p $ matrix polynomial, where $ Z_{mn}(\lam)$ has $ mn$ rows and $Z_r(\lam)$ has $r$ rows.
	
	\noin
	~(a) If $Z(\lam)$ is a right (resp., left) minimal basis of $\mathbb{L}(\lambda)$ then $\left[\begin{array}{@{}c@{}}
	(e^T_{m-c_0(\sigma, \sigma_2)} \otimes I_n ) Z_{mn}(\lambda) \\[.3em]   Z_r(\lam) \end{array} \right] $ (resp., $\left[\begin{array}{@{}c@{}} 	(e^T_{m-i_0(\sigma_1, \sigma)} \otimes I_n ) Z_{mn}(\lambda) \\[.3em]   Z_r(\lam) \end{array} \right] $)  is a right (resp., left) minimal basis of $\mathcal{S}(\lambda).$

	\noin
	~(b) Let $ \tau$ be given by $ \tau :=( \tau_{l}, -m, \tau_r).$ Set $\alpha : = \big( - rev(\tau_{l}), \sigma, -rev(\tau_r) \big) $. Let $c(\alpha) $ and $ i(\alpha)$ be the total number of consecutions and inversions of the permutation $\alpha$. If $\varepsilon_1\leq   \cdots \leq\varepsilon_p$ are the right (resp., left) minimal indices of $\mathbb{L}(\lambda)$ then $\varepsilon_1 -i(\alpha)  \leq \cdots \leq \varepsilon_p - i(\alpha) $ (resp., $\varepsilon_1 -c(\alpha) \leq \cdots \leq \varepsilon_p - c(\alpha) $) are the right (resp., left) minimal indices of  $\mathcal{S}(\lambda).$
	
\end{theorem}

\begin{proof} We have $ \mathbb{L}(\lam) = U \, \mathbb{T}_\omega(\lambda) \, V $, where $ \mathbb{T}_\omega(\lambda) := \lambda \mathbb{M}^\mathcal{S}_{\tau} -  \mathbb{M}^\mathcal{S}_{\sigma}$ is a PGF pencil of $G(\lambda)$ associated with the permutation $\omega:=(\sigma,-\tau)$ of $\{0:m\}$, and $U : = \mathbb{M}_{(\tau_1,\sigma_1)} (Y_1,X_1)$ and $V : = \mathbb{M}_{(\sigma_2,\tau_2)} (X_2,Y_2)$. Since  $V$ is a nonsingular matrix, it is easily seen that the map $ V : \mathcal{N}_r(\mathbb{L}) \rightarrow \mathcal{N}_r(\mathbb{T}_\omega),~ z(\lambda) \mapsto V z(\lambda) ,$ is an isomorphism and maps a minimal basis of $\mathcal{N}_r(\mathbb{L}) $ to a minimal basis of $\mathcal{N}_r(\mathbb{T}_\omega)$. On the other hand, by Theorem \ref{eigpgf_RCh4},
	$ \mathbb{F}^{\scalebox{.4}{PGF}
	}_{\omega}(\mathcal{S}) : \mathcal{N}_r(\mathbb{T}_\omega) \rightarrow  \mathcal{N}_r( \mathcal{S}), \left[
	\begin{array}{@{}c@{}}
	x(\lam)\\ \hline
	y(\lambda) \\
	\end{array}
	\right]  \mapsto \left[
	\begin{array}{@{}c@{}}
	(e^T_{m-c_0(\sigma)}\otimes I_n)x(\lam) \\ \hline
	y(\lambda) \\
	\end{array}
	\right]  ,$  is an isomorphism and maps a minimal basis of $\mathcal{N}_r(\mathbb{T}_\omega)$ to a minimal basis of $\mathcal{N}_r(\mathcal{S})$, where $x(\lambda) \in \mathbb{C}(\lambda)^{mn}$ and $y(\lambda) \in \mathbb{C}(\lambda)^{r}$. Consequently, $\mathbb{F}^{\scalebox{.4}{PGF}}_{\omega}(\mathcal{S})  V: \mathcal{N}_r(\mathbb{L}) \rightarrow \mathcal{N}_r( \mathcal{S}),$ $ z(\lam) \mapsto \mathbb{F}^{\scalebox{.4}{PGF}}_{\omega}(\mathcal{S})  V z(\lam),$  is an isomorphism and maps a minimal basis of $\mathcal{N}_r(\mathbb{L})$ to a minimal basis of $\mathcal{N}_r(\mathcal{S})$. Now, by Lemma~\ref{lemmblockrowncolp0221d17ee_RCh4}, we have $\mathbb{F}^{\scalebox{.4}{PGF}}_{\omega}(\mathcal{S})  V = \mathbb{F}^{\scalebox{.4}{PGF}}_{\omega}(\mathcal{S}) \mathbb{M}_{(\sigma_2,\tau_2)}  (X_2,Y_2) =$
	$$ \left[
	\begin{array}{@{}c@{\,\,\,}|@{\,}c@{}}
	(e^T_{m-c_0(\sigma)} \otimes I_n)\,  M_{(\sigma_2,\tau_2)}  (X_2,Y_2) &  \\ \hline
	& I_r \\
	\end{array}
	\right]=
	\left[
	\begin{array}{@{}c@{\,\,\,}|@{\,}c@{}}
	e^T_{m-c_0(\sigma, \sigma_2)} \otimes I_n &  \\ \hline
	& I_r \\
	\end{array}
	\right],$$ and hence the desired result for the recovery of right minimal bases follows.

 Now we describe the recovery of left minimal bases. Since $U$ is a nonsingular matrix, it is easily seen that the map $ U^T: \mathcal{N}_l(\mathbb{L}) \rightarrow \mathcal{N}_l(\mathbb{T}_\omega),~ z(\lambda) \mapsto U^T z(\lambda) , $ is an isomorphism and maps a minimal basis of $\mathcal{N}_l(\mathbb{L}) $ to a minimal basis of $\mathcal{N}_l(\mathbb{T}_\omega)  $. On the other hand, by Theorem \ref{eigpgf_RCh4},
	$ \mathbb{K}^{\scalebox{.4}{PGF}
	}_{\omega}(\mathcal{S}) : \mathcal{N}_l(\mathbb{T}_\omega) \rightarrow  \mathcal{N}_l( \mathcal{S}), \left[
	\begin{array}{@{}c@{}}
	x(\lam)\\ \hline
	y(\lambda) \\
	\end{array}
	\right]  \mapsto \left[
	\begin{array}{@{}c@{}}
	(e^T_{m-i_0(\sigma)}\otimes I_n)x(\lam) \\ \hline
	y(\lambda) \\
	\end{array}
	\right]  ,$  is an isomorphism and maps a minimal basis of $\mathcal{N}_l(\mathbb{T}_\omega)$ to a minimal basis of $\mathcal{N}_l(\mathcal{S})$, where $x(\lambda) \in \mathbb{C}(\lambda)^{mn}$ and $y(\lambda) \in \mathbb{C}(\lambda)^{r}$. Consequently,  $\mathbb{K}^{\scalebox{.4}{PGF}}_{\omega}(\mathcal{S}) U^T  :\mathcal{N}_l(\mathbb{L}) \rightarrow \mathcal{N}_l(\mathcal{S})$, $z(\lam) \mapsto \mathbb{K}^{\scalebox{.4}{PGF}}_{\omega}(\mathcal{S}) U^T z(\lam),$ is an isomorphism and maps a minimal basis of $\mathcal{N}_l(\mathbb{L})$ to a minimal basis of $\mathcal{N}_l(\mathcal{S})$. Now $\mathbb{K}^{\scalebox{.4}{PGF}}_{\omega}(\mathcal{S}) U^T = \mathbb{K}^{\scalebox{.4}{PGF}}_{\omega}( \mathcal{S}) (  \mathbb{M}_{(\tau_1,\sigma_1)} (Y_1,X_1))^T
	= $
	$${\small  \left[
		\begin{array}{@{}c@{\,}|@{\,}c@{}}
		(e^T_{m-i_0(\sigma)} \otimes I_n) \big(M_{(\tau_1,\sigma_1)} (Y_1,X_1) \big)^T &  \\ \hline
		& I_r \\
		\end{array}
		\right] = \left[
		\begin{array}{@{}c@{\,}|@{\,}c@{}}
		\Big( M_{(\tau_1,\sigma_1)}(Y_1,X_1) (e_{m-i_0(\sigma)} \otimes I_n) \Big)^T&  \\ \hline
		& I_r \\
		\end{array}
		\right] } .$$ By Lemma~\ref{lemmblockrowncolp0221d17ee_RCh4}, we have
	$ M_{(\tau_1,\sigma_1)} (Y_1,X_1) \, (e_{m-i_0(\sigma)} \otimes I_n) =  e_{m-i_0(\sigma_1,\sigma)} \otimes I_n$. Hence the desired result for recovery of left minimal bases follows.

	Finally, let $\varepsilon_1\leq \cdots \leq\varepsilon_p$ be the right (resp., left) minimal indices of $\mathbb{L}(\lambda)$. Since the PGF pencil $\mathbb{T}_{\omega}(\lambda)$ is strictly equivalent to $ \mathbb{L}(\lambda)$, $\varepsilon_1\leq \cdots \leq\varepsilon_p$ are also the right (resp., left) minimal indices of $\mathbb{T}_{\omega}(\lambda)$. Hence by Theorem \ref{eigpgf_RCh4},  $\varepsilon_1 -i(\alpha)  \leq \cdots \leq \varepsilon_p - i(\alpha) $ (resp., $\varepsilon_1 -c(\alpha) \leq \cdots \leq \varepsilon_p - c(\alpha) $) are the right (resp., left) minimal indices of  $\mathcal{S}(\lambda).$
\end{proof}

The next result describes the recovery of  eigenvectors of $\mathcal{S}(\lam)$ from those of the GFPRs of $ \mathcal{S} (\lam)$ when $\mathcal{S}(\lam)$ is regular.

\begin{theorem} \label{gfprrecoslamregular_RCh4} Let $ \mathbb{L}(\lambda):=  \mathbb{M}_{(\tau_1,\sigma_1)} (Y_1,X_1) (\lambda \mathbb{M}^\mathcal{S}_{\tau} -  \mathbb{M}^\mathcal{S}_{\sigma}) \mathbb{M}_{(\sigma_2,\tau_2)}  (X_2,Y_2) $ be a  GFPR of $\mathcal{S}(\lambda)$. Suppose that $\mathcal{S}(\lam)$ is regular and $ \mu \in \C$ is an eigenvalue of $\mathcal{S}(\lam)$. Let $ Z := \left[\begin{array}{c} Z_{mn} \\  Z_r \end{array}\right]$ be an $(mn+r)\times p $ matrix such that $\rank(Z) = p,$ where $ Z_{mn}$ has $ mn$ rows and $Z_r$ has $r$ rows. If $Z$ is a  basis of $\mathcal{N}_r(\mathbb{L} (\mu))$ (resp., $\mathcal{N}_l(\mathbb{L} (\mu))$) then $\left[\begin{array}{@{}c@{}}
	(e^T_{m-c_0(\sigma, \sigma_2)} \otimes I_n ) Z_{mn} \\[.3em]   Z_r \end{array} \right] $ (resp., $\left[\begin{array}{@{}c@{}} 	(e^T_{m-i_0(\sigma_1, \sigma)} \otimes I_n ) Z_{mn}  \\[.3em]   Z_r \end{array} \right] $)  is a  basis of $\mathcal{N}_r(\mathcal{S}(\mu))$ (resp., $\mathcal{N}_l(\mathcal{S}(\mu))$).
	
\end{theorem}

\begin{proof} A verbatim proof of Theorem~\ref{gfprrecoslame_RCh4} together with part (II) of Theorem~\ref{eigpgf_RCh4} yields the desired results.
\end{proof}

Next, we briefly describe the recovery of eigenvectors, minimal bases and minimal indices of a structured $G(\lam)$ from those of the structured linearizations  discussed in Section~\ref{structured_lin_GFPR_Glam}.

Note that if $G(\lam)$ is singular then the left (resp., right) minimal indices of $G(\lam)$ and $ X G(\lam) Y$ are the same for any nonsingular matrices $X$ and $Y.$  Hence it follows that
if $G(\lam)$ is symmetric (resp., skew-symmetric, Hamiltonian, skew-Hamiltonian) then the left minimal indices of $G(\lam)$ are the same as the right minimal indices of $G(\lam).$ Consequently, if $\mathbb{L}(\lam)$ is a structure-preserving linearization of $G(\lam)$  considered in Section~\ref{structured_lin_GFPR_Glam} then the left minimal indices of $\mathbb{L}(\lam)$ are the same as the right minimal indices of $\mathbb{L}(\lam)$.
Since   $\mathbb{L}(\lam)$ is strictly equivalent to a GFPR $ \mathbb{T}(\lambda):=  \mathbb{M}_{(\tau_1,\sigma_1)} (Y_1,X_1) (\lambda \mathbb{M}^\mathcal{S}_{\tau} -  \mathbb{M}^\mathcal{S}_{\sigma}) \mathbb{M}_{(\sigma_2,\tau_2)} (X_2,Y_2)$  of $G(\lambda)$,  the left and right minimal indices of $\mathbb{T}(\lambda)$ are the same. Let $\tau$ be given by $\tau=(\tau_{\ell}, -m , \tau_r)$. Define $\alpha:= (-rev(\tau_{\ell}), \sigma, -rev(\tau_r))$. Then $\alpha$ is a permutation of $\{0:m-1\}$. Let $c(\alpha)$ and $i(\alpha)$, respectively, be the total number of consecutions and inversions of $\alpha$. Let $\varepsilon_1\leq \cdots \leq\varepsilon_k$ be the minimal (left and right) indices of $\mathbb{T}(\lambda)$. Then by Theorem \ref{gfprrecoslame_RCh4}, $\varepsilon_1 -i(\alpha) \leq  \cdots \leq \varepsilon_k - i(\alpha)$ and $\varepsilon_1 -c(\alpha) \leq  \cdots \leq \varepsilon_k - c(\alpha),$ respectively, are the right and left minimal indices of  $G(\lambda).$ Since the left and right minimal indices of $G(\lambda)$ are the same, we must have $i(\alpha)=c(\alpha)$. But $ i(\alpha)+c(\alpha)= m-1.$ Consequently, we have $i(\alpha) =(m-1)/2= c(\alpha) $ which shows that
$\varepsilon_1 - (m-1)/2 \leq \cdots \leq\varepsilon_k -(m-1)/2 $ are the minimal (left and right) indices of $ G(\lambda)$. Recall  that $\mathbb{L}(\lam)$ is not a linearization of $G(\lam)$ if $m$ is even.

Thus, if $\mathbb{L}(\lam)$ is a structure-preserving linearization of $G(\lam)$ considered in Section~\ref{structured_lin_GFPR_Glam} then the left minimal indices of $\mathbb{L}(\lambda)$ are the same as the right minimal indices of $\mathbb{L}(\lambda)$.  Moreover, if $\varepsilon_1\leq \cdots \leq\varepsilon_k$ are the minimal (left and right) indices of $ \mathbb{L}(\lambda)$ then $\varepsilon_1 - (m-1)/2 \leq \cdots \leq\varepsilon_k -(m-1)/2 $ are the minimal (left and right) indices of $ G(\lambda)$. Hence we only need to comment on the recovery of eigenvectors and minimal bases of $G(\lam)$ from those of the $\mathbb{L}(\lam)$.

Note that the left minimal bases of $G(\lam)$ are the same as the right minimal bases of $G(\lam)$  when $G(\lam)$ is symmetric (resp., Hamiltonian, skew-Hamiltonian, skew-symmetric). Hence
 if $\mathbb{L}(\lam)$ is a structure-preserving linearization of $G(\lam)$  considered in Section~\ref{structured_lin_GFPR_Glam} then the left minimal bases of $\mathbb{L}(\lam)$ are the same as the right minimal bases of $\mathbb{L}(\lam)$. Consequently, minimal bases and eigenvectors of $G(\lam)$ can be recovered from those of $\mathbb{L}(\lam)$ as special cases of Theorem~\ref{gfprrecoslame_RCh4} and Theorem~\ref{gfprrecoslamregular_RCh4}. Indeed, for structure-preserving linearizations, we have $c_0(\sig, \sigma_2) = 0$ when $h =0$ and, $c_0(\sig, \sigma_2) $ is given in the Table~\ref{tab1} when $h>0.$

\renewcommand{\arraystretch}{1.5}
	\begin {table}[H]
		\begin{center}
		\begin{tabular}{|c|c|c|c|}
			\hline
			Structure & symmetric  & $T$-even/odd & skew-symmetric \\
			\hline
			$c_0(\sig, \sig_2)$& $2+ i_2(\textbf{t}_{w_h})$   & 1       & $2+ c_2(\textbf{t}_{w})$  \\	\hline
		\end{tabular}   \caption{ } \label{tab1}
	\end{center}
	\end {table}
\vspace*{-\baselineskip}

{\bf Conclusion.}  We have made four major contributions in this paper. First, we have generalized GFPRs of  a matrix polynomial $P(\lam)$ to the case of  a rational matrix $G(\lam).$ Moreover, we have shown that the transition from GFPRs of matrix polynomials to GFPRs of rational matrices is operation-free (Theorem~\ref{gfprptoge_RCh4}). Second, and most importantly, we have utilized  GFPRs of $G(\lam)$  to construct structure-preserving Rosenbrock strong linearizations of a structured  (symmetric, Hermitian, skew-symmetric, even, odd, etc.) rational matrix $G(\lam).$
Third, we have shown that FPs, GFPs and GFPRs of $G(\lam)$  are Rosenbrock strong linearizations of  $G(\lam).$ Fourth, we have described automatic  recovery rules for eigenvectors, minimal bases and minimal indices  of $G(\lam)$ from those of the linearizations of $G(\lam).$

%

\appendix 

\section{ The proof of Lemma~\ref{lem_full_expre}}\label{appendix1}
\begin{proof} For simplicity, we write $\widehat{\Lambda}_{i,j}$, $\Lambda_{i,j},$ $ \widehat{\Omega}_{i,j}  $ and $\Omega_{i,j} $ for $\widehat{\Lambda}_{i,j} (\lam),$  $\Lambda_{i,j}(\lam)$, $\widehat{\Omega}_{i,j} (\lam) $ and $\Omega_{i,j} (\lam)$, respectively, where $\widehat{\Lambda}_{i,j} (\lam),$  $\Lambda_{i,j}(\lam)$, $\widehat{\Omega}_{i,j} (\lam) $ and $\Omega_{i,j} (\lam)$ are defined in (\ref{eqn-bc}) and (\ref{eqn-br}). For $ t \in \{ 1:m-1\}$, we have the following:
	\begin{equation}\label{Q_cole}
	Q_t^{\mathcal{B}} (e_i \otimes I_n) =
	\left\{
	\begin{array}{ll}
	(e_i \otimes I_n)  + (e_{i+1} \otimes \lam I_n)  & \mbox{if}~ t=i \text{ for } i=1:m-1, \\
	e_i \otimes I_n  & \text{if } t \neq i  \text{ for } i=1:m,
	\end{array}
	\right.
	\end{equation}
	\begin{equation}\label{R_cole}
	R_t^{\mathcal{B}} (e_i \otimes I_n) =
	\left\{
	\begin{array}{ll}
	e_{i+1} \otimes I_n & \text{if } t=i \text{ for } i=1:m-1,\\
	e_{i} \otimes I_n & \text{if } t \notin \{i , i-1\} \text{ for } i=1:m.   \end{array}
	\right.
	\end{equation}
	
	Let $1 \leq k \leq m-1$ and $  p \geq 0, ~q \geq 0$ be such that $ k +p+q-1 \leq m-1.$
	Consider
	\begin{align*}
	 Z(\lam) : =  \underbrace{ R^\mathcal{B}_{k+p+q-1}  \cdots R^\mathcal{B}_{k+p+1} R^\mathcal{B}_{k+p} }_{Y(\lam) } \, \underbrace{Q^\mathcal{B}_{k+p-1} \cdots Q^\mathcal{B}_{k+1} Q^\mathcal{B}_{k}}_{X(\lam)}.
	\end{align*}
 Then $X(\lam)$ (resp., $Y(\lam)$) is a product of $ p$ (resp., $q$) $ Q^\mathcal{B}$'s (resp., $ R^\mathcal{B}$'s). We show that
	\begin{equation} \label{eqnepe}
	Z(\lam)  (e_k \otimes I_n)=
	\left[
	\begin{matrix}
	0_{(k-1)n \times n}\\
	\widehat{\Lambda}_{p,q}\\
	\lam^{p} I_n\\
	0_{(m-k-p-q)n \times n}
	\end{matrix}\right].
	\end{equation}
By applying (\ref{Q_cole}) repeatedly, we have
	\begin{align*}
		&   X(\lam) \;  (e_k \otimes I_n)=  \Big(  Q^\mathcal{B}_{k+p-1} \cdots Q^\mathcal{B}_{k+1} \Big) \;     \Big( (e_{k+1} \otimes \lam I_n) + (e_{k} \otimes I_n)  \Big )\\
		& = \Big(  Q^\mathcal{B}_{k+p-1} \cdots Q^\mathcal{B}_{k+2}  \Big)    \;  \Big((e_{k+2} \otimes \lam^2 I_n) + (e_{k+1} \otimes \lam I_n) + (e_{k} \otimes I_n)  \Big )\\
		& =   Q^\mathcal{B}_{k+p-1}    \; \Big(  (e_{k+p-1} \otimes \lam^{p-1} I_n) + \sum\nolimits_{j=k}^{k+p-2} (e_{j} \otimes \lam^{j-k} I_n)  \Big)  \\
		& =    (e_{k+p} \otimes \lam^{p} I_n) +(e_{k+p-1} \otimes \lam^{p-1} I_n) + \sum\nolimits_{j=k}^{k+p-2} ( e_{j} \otimes \lam^{j-k} I_n)   \\
		& =    (e_{k+p} \otimes \lam^{p} I_n) + \sum\nolimits_{j=k}^{k+p-1} (e_{j} \otimes \lam^{j-k} I_n).
\end{align*}
Now, by applying (\ref{R_cole}) repeatedly, we have $Z(\lam)  (e_k \otimes I_n) =$
	\begin{align*}
		 & \Big ( R^\mathcal{B}_{k+p+q-1}   \cdots R^\mathcal{B}_{k+p+1}  R^\mathcal{B}_{k+p} \Big) \; \Big(  (e_{k+p} \otimes \lam^{p} I_n) + \sum\nolimits_{j=k}^{k+p-1} (e_{j} \otimes \lam^{j-k} I_n ) \Big)\\
		&=    (e_{k+p+q} \otimes \lam^{p} I_n) + \sum\nolimits_{j=k}^{k+p-1} (e_{j} \otimes \lam^{j-k} I_n)  \\
		&= \left[
		\begin{matrix}
			0_{(k-1)n \times n}\\
			\widehat{\Lambda}_{p,q}\\
			\lam^{p} I_n\\
			0_{(m-k-p-q)n \times n}
		\end{matrix}\right], \text{ which proves } (\ref{eqnepe}).
	\end{align*}

We now prove that $U(\lam) (e_1 \otimes I_n) = \Lambda_\alpha (\lam)$. Recall the definitions of $\Lambda_{\alpha} (\lam)$, $m_j$ and $s_j$ associated with $\RCISS(\alpha) = (c_1,i_1,c_2,i_2, \ldots, c_\ell, i_\ell)$. If $\ell =1$ then by (\ref{eqnepe}) we have $ U(\lam) (e_1 \otimes I_n)= U_{(c_1,i_1)}(e_1 \otimes I_n) = \Lambda_\alpha (\lam). $  Next, if $\ell > 1$ then by using (\ref{Q_cole}), (\ref{R_cole}) and (\ref{eqnepe}) repeatedly we have $U(\lam) (e_1 \otimes I_n) = \Lambda_\alpha (\lam)$. Indeed, we have the following. Recall that $ \widehat{\Lambda}_{c_j,i_j} \in \mathbb{C}[\lam]^{(c_j+i_j)n \times n}$.  We denote by $ \text{\Large 0}$ the zero matrix of an appropriate size.
Then  we have
\begin{align*}
 &U(\lam) (e_1 \otimes I_n) = U_{(c_\ell,i_\ell)}  \cdots U_{(c_2,i_2)} U_{(c_1,i_1)} (e_1 \otimes I_n)\\
  & = U_{(c_\ell,i_\ell)}  \cdots U_{(c_2,i_2)}  \left[
 \begin{matrix}
 \widehat{\Lambda}_{c_1,i_1}\\
 \lam^{c_1} I_n\\
 \text{\Large 0}\\
 \end{matrix}\right] ~[\text{by } (\ref{eqnepe}) \text{ since } s_0 = 0 ]\\
 & = U_{(c_\ell,i_\ell)}  \cdots U_{(c_2,i_2)} \Big ( (e_{s_1+1} \otimes \lam^{c_1} I_n)  + \left[
 \begin{matrix}
 \widehat{\Lambda}_{c_1,i_1}\\
 0\\
 \text{\Large 0}\\
 \end{matrix}\right] \Big ) ~[\text{since } s_1 = c_1 +i_1] \\
 & = U_{(c_\ell,i_\ell)}  \cdots U_{(c_2,i_2)} (e_{s_1+1} \otimes \lam^{c_1} I_n)  + \left[
 \begin{matrix}
 \widehat{\Lambda}_{c_1,i_1}\\
 0\\
 \text{\Large 0}\\
 \end{matrix}\right] ~[\text{by } (\ref{uBij}), (\ref{Q_cole}) \text{ and }  (\ref{R_cole})]\\
   & = U_{(c_\ell,i_\ell)}  \cdots U_{(c_3,i_3)}
 \left[
  \begin{matrix}
  0_{s_1n \times n}\\
 \lam^{c_1} \widehat{\Lambda}_{c_2,i_2}\\
  \lam^{c_1} \lam^{c_2} I_n\\
   \text{\Large 0}\\
  \end{matrix}\right] +
  \left[
 \begin{matrix}
 \widehat{\Lambda}_{c_1,i_1}\\
 0\\
  \text{\Large 0}\\
 \end{matrix}\right] ~[\text{by } (\ref{eqnepe})]\\
& = U_{(c_\ell,i_\ell)}  \cdots U_{(c_3,i_3)}  (e_{s_2+1} \otimes \lam^{c_1+c_2} I_n)+
 \left[
 \begin{matrix}
 0_{s_1n \times n}\\
 \lam^{c_1} \widehat{\Lambda}_{c_2,i_2}\\
 0\\
 \text{\Large 0}\\
 \end{matrix}\right] +
 \left[
 \begin{matrix}
 \widehat{\Lambda}_{c_1,i_1}\\
 0\\
 \text{\Large 0}\\
 \end{matrix}\right] ~\begin{array}{l} [\text{by }  (\ref{uBij}), (\ref{Q_cole}) \\ \text{ and }  (\ref{R_cole}) ]\end{array}  \\
  & = U_{(c_\ell,i_\ell)}  \cdots U_{(c_3,i_3)}  (e_{s_2+1} \otimes \lam^{m_2} I_n)+
   \left[
 \begin{matrix}
 \widehat{\Lambda}_{c_1,i_1}\\
 \lam^{m_1} \widehat{\Lambda}_{c_2,i_2}\\
 0\\
 \text{\Large 0}\\
 \end{matrix}\right] ~[\text{since } m_1 =c_1 \text{ and } m_2 = c_1 + c_2]\\
  & = U_{(c_\ell,i_\ell)}   (e_{s_{\ell-1} +1} \otimes \lam^{m_{\ell -1}} I_n)  +
\left[
 \begin{matrix}
 \widehat{\Lambda}_{c_1,i_1} \\
 \lambda^{m_1}\widehat{\Lambda}_{c_2,i_2} \\
 \vdots\\
 \lambda^{m_{\ell-2}}\widehat{\Lambda}_{c_{\ell - 1},i_{\ell-1}} \\[1em]
 \text{\Large 0}\\
 \end{matrix}
 \right] ~\begin{array}{l}
[\text{by repeated application of } \\
~(\ref{Q_cole}) ,(\ref{R_cole}) \text{ and }   (\ref{eqnepe})]
 \end{array} \\
  & = \left[
 \begin{matrix}
 0_{s_{\ell -1} n \times n}\\
 \lam^{m_{\ell-1}} \widehat{\Lambda}_{c_\ell,i_\ell}\\
 \lam^{m_{\ell -1}} \lam^{c_\ell} I_n\\
  \end{matrix}\right] + \left[
  \begin{matrix}
  \widehat{\Lambda}_{c_1,i_1} \\
  \lambda^{m_1}\widehat{\Lambda}_{c_2,i_2} \\
  \vdots\\
  \lambda^{m_{\ell-2}}\widehat{\Lambda}_{c_{\ell - 1},i_{\ell-1}} \\[1em]
  \text{\Large 0}\\
  \end{matrix}
  \right]  =
  \left[
 \begin{matrix}
 \widehat{\Lambda}_{c_1,i_1} \\
 \lambda^{m_1}\widehat{\Lambda}_{c_2,i_2} \\
 \vdots\\
 \lambda^{m_{\ell-2}}\widehat{\Lambda}_{c_{\ell - 1},i_{\ell-1}} \\
 \lambda^{m_{\ell-1}} \Lambda_{c_\ell, i_\ell}
 \end{matrix}
 \right] = \Lambda_\alpha (\lam) ~[\text{by } (\ref{eqnepe})].
 \end{align*}
This proves that $U(\lam) (e_1 \otimes I_n) = \Lambda_\alpha (\lam)$.

Next we prove that $ (e^T_1 \otimes I_n) V(\lam) = \Omega_\alpha (\lam)$. For $ t \in \{ 1:m-1\}$, we have
	\begin{equation}\label{Q_rowe}
	(e^T_i \otimes I_n) Q_t =
	\left\{
	\begin{array}{ll}
	(e^T_i \otimes I_n)  + (e^T_{i+1} \otimes \lam I_n)  & \mbox{if } t=i \text{ for } i=1:m-1, \\
	e^T_i \otimes I_n  & \text{if } t \neq i  \text{ for } i=1:m,
	\end{array}
	\right.
	\end{equation}
\begin{equation}\label{R_rowe}
	(e^T_i \otimes I_n) R_t =
	\left\{
	\begin{array}{ll}
	e^T_{i+1} \otimes I_n & \text{if } t=i \text{ for } i=1:m-1,\\
	e^T_{i} \otimes I_n & \text{if } t \notin \{i , i-1\} \text{ for } i=1:m.   \end{array}
	\right.
	\end{equation}

Let $1 \leq k \leq m-1$. Let $  p \geq 0 $ and $q \geq 0$ be such that $ k +p+q -1\leq m-1.$
Consider $W(\lam) : =  R_{k} R_{k+1} \cdots R_{k+p-1}  Q_{k+p} Q_{k+p+1} \cdots Q_{k+p+q-1}.$ Then (\ref{Q_rowe}) and (\ref{R_rowe}) and similar arguments as those in the proof of (\ref{eqnepe}) give
\begin{equation} \label{eqnepee}
(e^T_k \otimes I_n) W(\lam)=
\left[
\begin{matrix}
0_{(k-1)n \times n}\\
\widehat{\Omega}_{p,q}\\
\lam^q I_n\\
0_{(m-k-p-q)n \times n}
\end{matrix}
\right]^{\mathcal{B}} 		
\end{equation}
Hence by (\ref{eqnepee}), (\ref{Q_rowe}), (\ref{R_rowe}) and by similar arguments as those in the proof of $U(\lam) (e_1 \otimes I_n) = \Lambda_\alpha (\lam)$, we have $ (e^T_1 \otimes I_n) V(\lam) = \Omega_\alpha (\lam)$.
\end{proof}

\section{The proof of Proposition~\ref{prop_elimination}} \label{appendix2}

\begin{proof} Define $Z : = \left[
	\begin{array}{c|c}
	I_{(m-1)n} & 0 \\
	\hline
	0 &0\\
	\end{array}
	\right]
	+ X (\lam)  Y (\lam). $
	Let $Z_{i,j}$ be the $(i,j)$-th block entry of $Z$. Since  $\mathbf{x}_i \mathbf{y}_i = 0$, we have $Z_{i,i} = I_n$ for $i = 1:m-1$ and  $Z_{m,m} =\mathbf{x}_m \mathbf{y}_m. $
	Further, note that we have either
	$Z_{i,j} = \lam^{p_i+q_j} I_n$ or $Z_{i,j} =0$ for all  $i \neq j.$  Hence we have
	$Z = \left[ \begin{array}{ccccc}
	I_n & Z_{1,2} & \cdots &Z_{1, m-1} &Z_{1,m} \\
	Z_{2,1}  & I_n  & \cdots & Z_{2,m-1} &Z_{2,m} \\
	\vdots  & \vdots & \ddots & \vdots & \vdots \\
	Z_{m-1,1}  &Z_{m-1,2} & \cdots &  I_n & Z_{m-1,m} \\
	Z_{m,1}  & Z_{m,2} & \cdots & Z_{m, m-1}& Z_{m,m} \\
	\end{array} \right].$ Now define $L(\lam)$ and $U(\lam)$ by
	$L(\lam) := \left[ \begin{array}{@{}cccc@{}}
	I_n &  & &  \\
	-Z_{2,1}  & I_n  & &  \\
	\vdots  & \cdots & \ddots &  \\
	-Z_{m,1}  & -Z_{m,2} & \cdots & I_n \\
	\end{array} \right]$ and  $U (\lam) := \left[ \begin{array}{@{}cccc@{}}
	I_n & -Z_{1,2} & \cdots & -Z_{1,m} \\
	& I_n  & & -Z_{2,m} \\
	&  & \ddots & \vdots \\
	&  &  & I_n \\
	\end{array} \right].$
	
	Note that $\mathbf{x}_i \mathbf{y}_i = 0 \Rightarrow \mathbf{y}_i \mathbf{x}_i =0$ for $i = 1:m-1$. Hence it follows
	that $Z_{i,j} Z_{j,k} = \mathbf{x}_i \mathbf{y}_j \mathbf{x}_j \mathbf{y}_k = 0 $ for $ i,k \in \{1:m\}$ and $j \in \{1:m-1\}$. Consequently, by block Gaussian elimination, we have $L(\lam) Z U(\lam) = \diag (I_{(m-1)n}, ~\mathbf{x}_m \mathbf{y}_m ).$ \end{proof}

\end{document}